\documentclass[11pt, a4paper,leqno]{amsart}
\usepackage{amsmath,amsthm,amscd,amssymb,amsfonts, amsbsy}
\usepackage{latexsym}
\usepackage{txfonts}
\usepackage{exscale}
\usepackage{enumerate}
\usepackage{enumitem}
\usepackage{moreenum}
 \usepackage{float}
 \usepackage{bbm}
\usepackage{mathrsfs}
\usepackage{currvita}
\usepackage[colorlinks,citecolor=red,pagebackref,hypertexnames=false]{hyperref}
\usepackage{color}
\usepackage[active]{srcltx}

\definecolor{orange}{rgb}{.75,.25,0}

\calclayout
\allowdisplaybreaks

\theoremstyle{plain}
\newtheorem{theorem}[equation]{Theorem}
\newtheorem{lemma}[equation]{Lemma}
\newtheorem{corollary}[equation]{Corollary}
\newtheorem{proposition}[equation]{Proposition}

\theoremstyle{definition}
\newtheorem{definition}[equation]{Definition}

\newtheorem{example}[equation]{Example}

\theoremstyle{remark}
\newtheorem{remark}[equation]{Remark}

\newcommand{\ga}{\gamma}

\numberwithin{equation}{section}

\newcommand{\RR}{{\mathbb{R}}}

\DeclareMathOperator\Max{\mathcal{M}} %HL Max Operator

\def\Xint#1{\mathchoice
{\XXint\displaystyle\textstyle{#1}}%
{\XXint\textstyle\scriptstyle{#1}}%
{\XXint\scriptstyle\scriptscriptstyle{#1}}%
{\XXint\scriptscriptstyle%
\scriptscriptstyle{#1}}%
\!\int}
\def\XXint#1#2#3{{\setbox0=\hbox{$#1{#2#3}{%
\int}$ }
\vcenter{\hbox{$#2#3$ }}\kern-.6\wd0}}
\def\barint{\,\Xint -} % \, corrects the \! used in the definition
\def\bariint{\barint_{} \kern-.4em \barint}
\def\bariiint{\bariint_{} \kern-.4em \barint}
\renewcommand{\iint}{\int_{}\kern-.34em \int} %\, minor space between the integrals
\renewcommand{\iiint}{\iint_{}\kern-.34em \int} %\, minor space between the integrals

\newcommand{\CC}{\mathcal{C}}

\newcommand{\ar}{\partial}
\newcommand{\la}{\lambda}
\newcommand{\La}{\Lambda}
\newcommand{\sem}{\setminus}
\newcommand{\De}{\Delta}
\newcommand{\dist}{\operatorname{dist}}
\newcommand{\ran}{\rangle}
\newcommand{\lan}{\langle}

\newcommand{\ti}{\widetilde}

\newcommand{\Ang}{\operatorname{Angle}}

\newcommand{\re}{\mathbb{R}}
\newcommand{\rn}{\mathbb{R}^n}

\newcommand{\ree}{\mathbb{R}^{n+1}}

\newcommand{\dd}{\mathbb{D}}
\newcommand{\si}{\sigma}

\newcommand{\F}{\mathcal{F}}

\newcommand{\B}{\mathcal{B}}

\newcommand{\sbf}{{\bf S}}

\newcommand{\G}{\mathcal{G}}

\def\H{\mathcal H}

\renewcommand{\d}{\, \mathrm{d}}

\newcommand{\ds}{\displaystyle}

\newcommand{\ph}{\phi}

\newcommand{\rar}{\rightarrow}

\newcommand{\Hpn}{\mathcal{H}^{n+1}_{\text{p}}}

\newcommand{\dhalf}{D_t^{1/2}} %half time derivative
\newcommand{\HT}{H_t} %Hilbert transform

\newcommand{\pc}{\mathcal{P}}

\DeclareMathOperator{\diam}{diam}
\DeclareMathOperator{\Hil}{H}
\DeclareMathOperator{\osc}{osc}

\begin{document}
\allowdisplaybreaks

\title[Parabolic Corona Decompositions]{Corona Decompositions for Parabolic\\ Uniformly Rectifiable Sets}

\author{S. Bortz}
\address{Department of Mathematics
\\
University of Alabama
\\
Tuscaloosa, AL, 35487, USA}
\email{sbortz@ua.edu}
\author{J. Hoffman}
\address{Department of Mathematics
\\
University of Missouri
\\
Columbia, MO 65211, USA}
\email{jlh82b@mail.missouri.edu}
\author{S. Hofmann}
\address{
Department of Mathematics
\\
University of Missouri
\\
Columbia, MO 65211, USA}
\email{hofmanns@missouri.edu}
\author{J.L. Luna-Garcia}
\address{Department of Mathematics \& Statistics \\ McMaster University \\ Hamilton, ON L8S 3L8, Canada} \email{lunagaj@mcmaster.ca}
\author{K. Nystr\"om}
\address{Department of Mathematics, Uppsala University, S-751 06 Uppsala, Sweden}
\email{kaj.nystrom@math.uu.se}

\thanks{The authors J. H., S. H., and J.L. L.-G.  were partially supported by NSF grants
DMS-1664047 and DMS-2000048. K.N. was partially supported by grant  2017-03805 from the Swedish research council (VR)}

\date{\today}

\begin{abstract}
We prove that parabolic uniformly rectifiable sets admit (bilateral) corona decompositions with respect to regular Lip(1,1/2) graphs. Together with our previous work, this allows us to conclude that if $\Sigma\subset\mathbb R^{n+1}$ is parabolic Ahlfors-David regular, then the following statements are equivalent.
\begin{enumerate}
\item $\Sigma$ is parabolic uniformly rectifiable.
\item $\Sigma$ admits a corona decomposition with respect to regular Lip(1,1/2) graphs.
\item $\Sigma$ admits a bilateral corona decomposition with respect to regular Lip(1,1/2) graphs.
\item $\Sigma$ is big pieces squared of regular Lip(1,1/2) graphs.
\end{enumerate}
\end{abstract}

\maketitle

\tableofcontents

\section{Introduction}

The results proved in this paper are part of a program devoted to providing parabolic analogues of the influential work of G. David and S. Semmes \cite{DS1,DS2} concerning uniformly rectifiable sets. At the time  \cite{DS1} was written, tools from the Calder\'on-Zygmund school of harmonic analysis were being rapidly adapted to less classical and rougher geometrical settings. A particularly important development was the proof of the $L^2$ boundedness  of the Cauchy integral on
Lipschitz curves \cite{CMM}, and subsequent developments concerning the boundedness of singular integral operators of Calder\'on type on Lipschitz graphs and
curves \cite{CDM}  (see also \cite{CCFJR, D4,M,LMcS, CS}). In addition, the Calder\'on-Zygmund theory
was extended beyond the Euclidean setting (see e.g. \cite{Ch}).

When restricting to the setting of $d$-dimensional subsets of $n$-dimensional Euclidean space,
David and Semmes \cite{DS1,DS2} found necessary and sufficient conditions on these sets for the validity of the
$L^2$ boundedness of $d$-dimensional singular integral operators. David and Semmes coined the sets having this property
{\it uniformly rectifiable}, and through a remarkable series of works they
provided numerous characterizations of these sets. Of particular relevance to this paper is the following theorem stating a subset of the equivalent characterizations of
 uniformly rectifiable sets. We refer to \cite{DS1} for  exact definitions of the different notions appearing in Theorem \ref{DS1URCorona.thrm}.

\begin{theorem}[\cite{DS1}]\label{DS1URCorona.thrm}
Let $E \subset \ree$ be $n$-dimensional Ahlfors(-David) regular. Then the following statements are equivalent.
\begin{enumerate}[label=(\alph*)]
\item $E$ is uniformly rectifiable.
\item $E$ admits a corona decomposition with respect to Lipschitz graphs.
\item  All sufficiently regular  convolution type
Calderon-Zygmund operators  with odd kernels are $L^2$ bounded on $E$.
\end{enumerate}
\end{theorem}

The purpose of this paper is to prove that the implication (a) $\implies$ (b) holds in the context of {\it parabolic} uniformly rectifiable sets (see Definition \ref{def1.UR}).  To prove that the implication (c) $\implies$ (b) holds in the parabolic setting is, as of writing, an interesting and open problem.  Combining the main results established in this paper, see Theorem \ref{main} and Theorem \ref{main.bil.thrm} below,  with the results of our previous work \cite{BHHLN1}, we obtain the following.

\begin{theorem}\label{thrm.char}
Let $\Sigma \subset \ree$ be parabolic Ahlfors-David regular (see Definition \ref{def1.ADR}). Then the following statements are equivalent.
\begin{enumerate}[label=(\alph*)]
\item $\Sigma$ is parabolic { uniformly} rectifiable.
\item $\Sigma$ admits a corona decomposition with respect to regular Lip(1,1/2) graphs.
\item $\Sigma$ admits a bilateral corona decomposition with respect to regular Lip(1,1/2) graphs.
\item $\Sigma$ is big pieces squared of regular Lip(1,1/2) graphs.
\end{enumerate}
\end{theorem}

All implications stated in Theorem \ref{thrm.char}, but the implications (a) $\implies$ (b) and (a) $\implies$ (c), are contained\footnote{The implication (c) $\implies$ (b) is trivial.} in \cite{BHHLN1}. We will prove\footnote{We could also follow the more indirect method in \cite{HMM1}, but we present an alternative approach here.} directly that (a) and (b) imply (c) by
essentially re-running the machine that we use to prove that
(a) $\implies$ (b), and by using additional information gleaned from the
fact that (b) $\implies$ (d). This means that the bulk of the paper is mainly devoted to the proof  of  the implication
(a) $\implies$ (b) and our proof will involve
a very careful adaptation of the proof in \cite{DS1}, along with some additional
arguments utilizing ideas from \cite{HLN}. The implication (a) $\implies$ (b) is the content of Theorem \ref{main} stated in precise form
below. The precise statement of the result that (a) $\implies$ (c) is given in Theorem \ref{main.bil.thrm}.

The study of parabolic uniformly rectifiable sets $\Sigma \subset \ree$ emanates in \cite{HLN}, \cite{HLN1} where the third and fifth author, together with J. Lewis,  introduced the  notion of parabolic uniformly rectifiable sets. In \cite{HLN} the existence of big pieces of regular Lip(1,1/2) graphs, under the additional assumption that $\Sigma$ is Reifenberg flat in the parabolic sense, is established. The results in \cite{HLN}, \cite{HLN1} were the first of their kind in the context of parabolic problems and the studies \cite{HLN}, \cite{HLN1} were motivated by the study of parabolic/caloric measure in rough time-varying domains.

By definition, a
Lip(1,1/2) graph is the graph of a function which is Lipschitz with respect to the parabolic metric. A Lip(1,1/2) graph is a regular Lip(1,1/2) graph if, in addition, the $1/2$-order time derivative of the function defining the graph is in (parabolic) BMO (see Definition \ref{goodgraph.def}). Combining the results of Section 2 in \cite{HLN1} with the developments in \cite{LM,H,HL}, it follows that
a Lip(1,1/2) graph is parabolic { uniformly} rectifiable in the sense of  \cite{HLN}, \cite{HLN1} if and only if it is a regular Lip(1,1/2) graph. This explains the importance
of regular Lip(1,1/2) graphs in the statement of Theorem \ref{thrm.char}, and in the statements of Theorem \ref{main} and Theorem \ref{main.bil.thrm}.  As proved in \cite{LS},
\cite{KW}, the class of regular Lip(1,1/2) graphs is strictly smaller than the class of  Lip(1,1/2) graphs.

The notions of regular Lip(1,1/2) graphs and parabolic uniform rectifiability are deeply rooted in the study of the Dirichlet problem for the heat equation in time-varying (graph) domains, and the solvability of the $L^p$-Dirichlet problem for the heat equation is intimately connected to quantitative mutual absolute continuity of the caloric measure and the surface measure. In particular, in \cite{KW} it was proved that there are Lip(1,1/2) graph domains for which caloric measure and surface measure are mutually singular. Later, in their pioneering work,  Lewis, Silver and Murray \cite{LS,LM} proved that for regular Lip(1,1/2) graph domains, the caloric measure and the surface measure are quantitatively related in the sense that they are mutual absolutely continuous, and the associated parabolic Poisson kernel satisfies a scale-invariant reverse H{\"o}lder inequality in $L^p$ for some $p\in (1,\infty)$. The importance and relevance of regular Lip(1,1/2) graph domains, from the perspective of parabolic singular integrals, layer potentials, boundary value problems and inverse problems,  is emphasized through the works in \cite{LS,LM,H1,H,HL,HL2,LN,N1}. In particular, in \cite{HL}
the solvability of the $L^2$-Dirichlet
problem (and of the $L^2$-Neumann problem)  for the heat equation, using layer potentials, was { obtained} 
in the region above a regular Lip(1,1/2) graph under the restriction that { the} 1/2-order time derivative (measured in BMO) of the function defining the graph is small\footnote{This smallness is sharp in the
sense that there are regular Lip(1,1/2) graph domains for which the $L^2$-Dirichlet problem is not solvable.
On the other hand, the $L^p$-Dirichlet problem is solvable for some $p<\infty$ for all
regular Lip(1,1/2) graph domains \cite{LM}.}. The highly influential work \cite{HL+} (see also \cite{DPP}), devoted to parabolic operators with singular terms, should also be mentioned in this context. In particular, in \cite{HL+} the method of extrapolation of Carleson measure estimates was introduced, a method which was crucial in the resolution of the Kato conjecture, see \cite{AHLeMcT}, \cite{AHLMcT}.

 In \cite{HLN}, \cite{HLN1} parts of the analysis on regular Lip(1,1/2) graph domains mentioned above was extended beyond the  setting of graphs. However, other than \cite{HLN}, \cite{HLN1} only a few notable works in rougher parabolic settings have appeared \cite{E,GH,GH2,N,NS,MP,MPT}. This is in contrast to the corresponding elliptic contexts, see, for example,  \cite{DJ,HMM1,GMT, AGMT,HLMN,MT,AHMMT},  and it is a long term goal of the authors to develop a corresponding
parabolic theory. The present paper is one of the starting points of that effort.
In particular, in forthcoming work we will show,  using Theorem \ref{thrm.char}, that the main theorem in \cite{HMM1} can be extended to the parabolic setting.
In addition, in  \cite{BHHLN2} we have recently obtained,  by expanding on \cite{NS}, a flexible parabolic
analogue of the work of David and Jerison \cite{DJ}. Still, to establish parabolic analogues of, for example,
\cite{GMT, AHMMT,HLMN} remain open problems which may require new methods and tools.

As mentioned, in the  elliptic setting, some of the most significant recent progress in this area relies
on methods which are currently unavailable in the parabolic setting. Thus, to
prove parabolic versions of those results, one will require
significantly new ideas.  However, to prove the results of this paper, we are able to develop parabolic versions of the methods in \cite{DS1,HMM1}. We prove two main results, whose statements may be found below. First, Theorem \ref{main},
whose elliptic analogue appears in \cite{DS1}, and we give the proof of Theorem \ref{main} in  Section \ref{subsgraphconst}, Section \ref{pushsqfn} and Section \ref{packsttime}. Second, Theorem \ref{main.bil.thrm}, whose elliptic analogue appears in \cite{HMM1}, and we give the proof of Theorem \ref{main.bil.thrm} in Section \ref{s5}.  Concerning Theorem \ref{main}, it turns out that we can follow quite closely the
corresponding arguments in \cite{DS1}.  The reader who is familiar with that work will certainly recognize substantial parts of the proof. The novelties
in Section \ref{subsgraphconst}, Section \ref{pushsqfn} and Section \ref{packsttime} are in many respects mainly of technical nature:  they arise here to
treat the extension to the parabolic setting, and they rely in part on some ideas from \cite{HLN}.
On the other hand, the corona decomposition is of central importance and utility in the theory of
quantitative rectifiability, and we believe that this fact justifies a complete and careful treatment
in the parabolic case.  Finally, we note that the converse to Theorem \ref{main}, which we have established
(in particular) in \cite{BHHLN1}, does indeed rely on some new methods that are significantly different to those
 previously appearing in the literature.

As we have previously mentioned in \cite{BHHLN1, BHHLN2}, it is true that  in \cite{RN1, RN2,RN3},  the author took on the ambitious challenge to develop the theory of
parabolic uniformly rectifiable sets. Unfortunately though, 
at least in the papers in \cite{RN1,RN2},
the author either gives no proofs of his claims or he 
supplies proofs which have fundamental errors.  
In the present work we give a correct proof of the main result claimed in \cite{RN1}, and we shall explain momentarily
 two principal errors (as well as a third, more technical error)
 in the latter paper.  
 The errors in \cite{RN2} are discussed in some detail in 
 our paper  \cite{BHHLN2}, which gives a correct proof of those results. 
 
In \cite{RN1} the author claims to prove that parabolic uniform rectifiability implies corona decompositions with respect to regular Lip(1,1/2) graphs.  In \cite{RN1}, the author, as we do here, follows rather closely the outline of the proof of the corresponding elliptic result in \cite{DS1}, while also, as we do, employing some ideas from \cite{HLN} in order to adapt the arguments of \cite{DS1} to the parabolic case.  It is in the latter context that the proof in \cite{RN1} breaks down, for the following reason:
in Lemma 2.7 in \cite{RN1},  the constant $\hat{c}$ cannot necessarily be taken small, given only the assumption of parabolic uniform rectifiability. In \cite{HLN}, the constant is small by virtue of the assumption of Reifenberg flatness used in that paper. 
Without smallness, the sets $\{E_i\}$, introduced in subsection 2.3.1 in \cite{RN1}, need not, as claimed,  have bounded overlaps and therefore the estimate in display (20) in \cite{RN1} is not valid.  This is a critical gap in the proof: it means that in \cite{RN1} there is no proof of the ``regularity" (see Definition \ref{goodgraph.def} below for the definition of a {\em regular} Lip(1,1/2) graph) of the approximating graphs, which is the essential feature required of the graphs in the parabolic setting (and which has no analogue in the elliptic case); thus, the proof in \cite{RN1} breaks down precisely at the point in the argument where one needs to do something new in the parabolic setting that has no counterpart in the corresponding elliptic
proof.
Below we deal with this problem by using a counting function introduced in \cite{DS1}, see Lemma \ref{count} below.

A second critical error in \cite{RN1} appears in the claimed proof of \cite[Theorem 1.5]{RN1}, which in turn relies\footnote{Even the reduction of  \cite[Theorem 1.5]{RN1} to \cite[Theorem 3.1]{RN1} is not clear to the authors. Indeed, the claimed bilateral approximation does not follow as in the elliptic setting as `nice' parabolic Littlewood-Paley kernels are always odd in $x$, but {\bf not}  necessarily in $(x,t)$ jointly.
As a consequence, they are unable to detect the difference between a $t$-independent plane and the same plane omitting the points with $t \in (a,b)$. In particular, they do not detect all `holes'; this would be essential for \cite[Claim 1, Section 3]{RN1} to ``be easily adapted from \cite[page 24]{DS1}" , as claimed therein.} on the claimed result  \cite[Theorem 3.1]{RN1}.  In fact, the latter is false.
Indeed, in contrast to the situation in the elliptic setting, parabolic uniform rectifiability is {\em not} characterized by Carleson set conditions such as the ``Bilateral Weak Geometric Lemma", even in the case of a Lip(1,1/2) graph;  see \cite[Observation 4.19]{BHHLN1} for a discussion of this issue and a counter-example.  Thus the proof of  \cite[Theorem 1.5]{RN1} collapses.

In addition, we mention a technical (but still serious) error, 
in the crucial estimate in the display
between displays (32) and (33) in \cite{RN1}.  Here, the author claims a Poincar{\'e} inequality that need not hold if the function $\psi_2$ depends on time, which it certainly does since the problem is parabolic; thus, again the error arises because the author in \cite{RN1} did not address the difference between the elliptic and parabolic settings.
In this context, the correct estimates, based on a more refined inequality of Poincar\'e type 
appropriate to the parabolic setting, are given in our Lemma \ref{l3} below.

The rest of the paper is organized as follows. Section \ref{pre} is mainly of preliminary nature and we here provide notation and definitions. In Section \ref{mr} we state our main result: Theorem \ref{main} and Theorem \ref{main.bil.thrm}. In Subsection \ref{s4} we briefly discuss the outline of the proof of Theorem \ref{main} to be presented in subsequent sections. The proof is divided into several parts and our argument
adapts the corresponding arguments  of \cite{DS1}, also utilizing some ideas from \cite{HLN}, to deal with the extension to the parabolic setting.
The proof of Theorem \ref{main} is given in Section \ref{subsgraphconst}, Section \ref{pushsqfn} and Section \ref{packsttime} and progresses from the construction of a Lip(1,1/2) graph with small constant for each stopping time regime, to the pushing of the geometric square function estimate (the Carleson measure estimate) to the graph, also verifying that the graph is actually a regular Lip(1,1/2) graph, and finally to the packing of the stopping time regimes. In Section \ref{s5} we then prove Theorem \ref{main.bil.thrm} by improving Theorem \ref{main}(4) to give a
bilateral approximation as in Theorem \ref{main.bil.thrm}(4). The paper ends with two appendix containing a few technical estimates and observations used in the paper.

 \section{Preliminaries}\label{pre}
In this section we provide notation and definitions. We will for simplicity assume $n\geq 2$. Our results also remain valid in the case $n=1$, but in this case some notations used in the paper have to be adjusted. Points in Euclidean $ (n+1) $-space
$ \mathbb R^{n+1} $ are denoted by $ (X,t) = ( x_1,
 \dots,  x_n,t)$,  where $ X = ( x_1, \dots,
x_{n} ) \in \mathbb R^{n } $ are the spatial coordinates and $t$ represents the
time-coordinate. We  set $$d := n+1.$$
In general, $c$ will denote a positive constant satisfying $1\leq c<\infty$.
We write  $c_1\lesssim c_2$, if $c_1/c_2$ is bounded from above by a positive
constant depending on the structural constants of a  lemma or theorem
(e.g. $n, d,$ the constant defining the Ahlfors-David regularity or the constants defining parabolic uniform rectifiability). We write $c_1\sim c_2$, if $c_1\lesssim c_2$ and $c_2\lesssim c_1$.

We let $  \lan \cdot ,  \cdot  \ran $  denote  the standard inner
product on $ \mathbb R^{n} $ and we let  $  | X | = \lan X, X \ran^{1/2} $ be
the  Euclidean norm of $ X. $  We let $||(X,t)||:=|X|+|t|^{1/2}$. Given $(X,t), (Y,s)\in\mathbb R^{n+1}$ we let $$d_p(X,t,Y,s):=d_p((X,t),(Y,s))=|X-Y|+|t-s|^{1/2},$$ and
we define $ d_p( X,t, E ) $  to equal the parabolic distance, defined with respect to $d_p(\cdot,\cdot)$, from
 $  (X,t) \in \mathbb R^{n+1} $ to $ E\subset\mathbb R^{n+1}$. We let $\diam(E)$ denote the parabolic diameter of $E$, that is,
 \[\diam(E) := \sup\{d_p(X,t,Y,s): (X,t), (Y,s) \in E\}.\]

We let
\begin{align}\label{cube1}
 C_r ( X, t ) \, := \, \{ ( Y, s )\in \mathbb R^{n} \times \mathbb{R} : | y_i  - x_i| < r, | t - s | < r^2  \},
 \end{align}
 whenever $(X,t)\in
\mathbb R^{n+1}$, $r>0$, and we will refer to $C_r(X,t) $ as a parabolic cube
of ``length" $r$ (in $\mathbb{R}^{n+1}$). For $(x,t) \in \mathbb{R}^n$, we will use the notation
\begin{align}\label{cube2} C'_r ( x, t ) \, := \, \{ ( y, s )\in \mathbb R^{n-1} \times \mathbb{R} : | y_i  - x_i| < r, | t - s | < r^2  \},
\end{align}
to denote a parabolic cube in $\mathbb{R}^n$.

We let $ \d X  $, $ \d x  $, denote  Lebesgue  measure on  $ \mathbb R^{n}$, $ \mathbb R^{n-1}$, respectively, and, given $\eta> 0$, we let $\H^\eta$ denote the (standard Euclidean) $\eta$-dimensional Hausdorff measure. To continue we need to define Hausdorff measure adapted to the parabolic setting. This is done in a straightforward way replacing the standard Euclidean diameter with its parabolic counterpart.

\begin{definition}
For $\eta, \delta > 0$, $E \subseteq \ree$, we define
\[\H^\eta_{p,\delta}(E) = \inf\left\{ \sum \diam(E_j)^\eta: E \subseteq \cup_j E_j, \diam(E_j) \le \delta \right\}.\]
We define the parabolic Hausdorff measure, of homogeneous dimension $\eta$, as
\begin{align}\label{hme}
\H^\eta_p(E) := \lim_{\delta \to 0^+} \H^\eta_{p,\delta}(E).
\end{align}
\end{definition}

Note that the limit in \eqref{hme} exists (possibly as $+ \infty$) since $\H^\eta_{p,\delta}(E)$ is a decreasing function on $(0, \infty)$.

In the sequel, $\Sigma \subset \mathbb R^{n+1}$ will denote a closed set which will be clear from the context. Below we will, for  $(X, t ) \in \Sigma $ and $r>0$, use the notation
$\Delta(X,t,r)=\Delta_r(X,t):=\Sigma\cap C_r(X,t)$ to denote a (parabolic) ``surface ball".   We define
\begin{equation}\label{sigdef}
\sigma := \sigma_{\Sigma} := \H^{n+1}_p|_\Sigma,
\end{equation}
to denote the parabolic surface measure on $\Sigma$.

Also other notions of surface measure can be considered in the parabolic context. Indeed, we let
\begin{equation}\label{slicewise.def}
 \mu ( E ) := \int_\re \int_{\rn \times \{t\}} \mathbbm{1}_{E}(X,t) \d \H^{n-1}(X) \, \d\H^{1}(t),
 \end{equation}
 where $\mathbbm{1}_{E}$ is the indicator function for the set $E$, denote a slice-wise or product-like measure of the set $E\subset\mathbb R^{n+1}$. This measure has been used in previous works concerning parabolic uniform rectifiability as in  \cite{KW,LM,H,HL,HL2,HLN,HLN1}, the parabolic surface measure was defined as
 \begin{equation}\label{sigdef+}
\sigma^s := \mu|_\Sigma.
\end{equation}

In this paper we do not use the measure in \eqref{sigdef+}, instead we use the measure defined in \eqref{sigdef}. We here collect a number of remarks concerning this choice. The non-trivial statements in Remark \ref{r-measures},  (i), (v), and (vi), are proved in Appendix \ref{slicevspara.sect} below.

\begin{remark}\label{r-measures} \ \smallskip
\begin{itemize}
\item[(i)] If $\sigma^{\bf s}$ (or for that matter
{\em any} measure $\mathfrak{m}$ defined on $\Sigma$) satisfies the parabolic
Ahlfors-David Regularity (p-ADR) condition (see Definition \ref{def1.ADR} below),
then so does $\si$, and in that case the two measures are of course equivalent.
This follows easily from the definition of the $\Hpn$ measure, and it is really just the same phenomenon that occurs in the classical (elliptic) case,
see \cite{DS1}. We give the proof for $\sigma^{\bf s}$ in Appendix \ref{slicevspara.sect}.

\smallskip

\item[(ii)] Consequently, if $\Sigma$ is a Lip(1,1/2) graph, then $\sigma \approx \sigma^{\bf s}$.
In particular, on a hyperplane $\pc\subset \ree$ parallel to the $t$-axis, which we may identify with
Euclidean space $\re^n$, we have that
$\H^n|_\pc \approx \Hpn|_\pc$, since the former is just $n$-dimensional Lebesgue measure
on $\pc$, which is parabolic ADR on $\Sigma = \pc$.

\smallskip

\item[(iii)]  If $\pc$ is a hyperplane parallel to the $t$-axis, and if $\pi$ is the
orthogonal projection operator onto $\pc$, then $\Hpn$ measure does not increase
under the action of $\pi$.  In particular, by virtue of (ii), we have, for any Borel set $A$,
that $\H^n(\pi(A)) = \Hpn(\pi(A)) \leq \Hpn(A)$.

\smallskip

\item[(iv)] If $\Sigma$ is parabolic uniformly rectifiable (p-UR set; see Definition
\ref{def1.UR} below), defined with respect\footnote{This means replacing $\sigma$ by $\sigma^{\bf s}$ in the definition of p-UR below and working with a set such that $\sigma^{\bf s}$ is p-ADR.} to $\sigma^{\bf s}$, then using  (i) one deduces\footnote{By nothing more than chasing definitions.} that $\Sigma$ is also p-UR defined with respect to $\sigma$. Thus, {\it a priori} p-UR with respect to $\sigma^{\bf s}$ is a stronger notion, that is, it implies p-UR with respect to $\sigma$.

\item[(v)]If  $\Sigma$ is parabolic uniformly rectifiable (p-UR; see Definition
\ref{def1.UR} below), with respect  to $\sigma$,
then the two measures $\sigma^{\bf s}$ and $\sigma$ are equivalent and hence one deduces (in the same way as (iv)) that $\Sigma$ is parabolic uniformly rectifiable, defined with respect to $\sigma^{\bf s}$. Thus, {\it a posteriori} the notions of p-UR are equivalent.

\smallskip

\item[(vi)] The measures are not equivalent in general, even in the p-ADR setting.  In fact,
$\sigma^{\bf s} \leq c(n) \sigma$,
but the other direction does not need to hold.
\end{itemize}
\end{remark}

\begin{definition}\label{tindplane.def}
We say that a $n$-dimensional hyperplane $P \subset \ree$ is a $t$-independent plane if $P$ contains a line in the $t$-direction. Equivalently, if $\vec{N}_P$ is the normal to $P$, then $\vec{N}_P \cdot (\vec{0},1) = 0$. We let $\mathcal{P}$ denote the collection of all $t$-independent planes. If $P \in \mathcal{P}$ we let $P_x\subset \mathbb{R}^n$ be the $(n-1)$-dimensional plane that defines $P = P_x \times \re$.
\end{definition}

Given a function $\psi:\mathbb R^{n-1}\times\mathbb R\to \mathbb R$  we let $D_{1/2}^t \psi  (x, t) $ denote
the $ 1/2 $ derivative in $ t $ of $ \psi ( x, \cdot ), x $ fixed.
This half derivative in time can be defined by way of the Fourier
transform or by
\begin{eqnarray} \label{1.8}
 D_{1/2}^t  \psi (x, t)  \equiv \hat c \int_{ \mathbb R }
\, \frac{ \psi ( x, s ) - \psi ( x, t ) }{ | s - t |^{3/2} } \, \d s,
\end{eqnarray} for properly chosen $ \hat c$. We let $ \| \cdot \|_* $ denote the
norm in parabolic $BMO(\mathbb R^{n})$ (replace standard cubes by parabolic cubes in the definition of $BMO$).

\begin{definition}\label{goodgraph.def}
A function $\psi:\mathbb R^{n-1}\times\mathbb R\to \mathbb R$ is called Lip(1,1/2) with constant $b_1$, if
\begin{eqnarray}\label{1.1}
|\psi(x,t)-\psi(y,s)|\leq b_1(|x-y|+|t-s|^{1/2})\
\end{eqnarray}
whenever $(x,t)\in\mathbb R^{n}$, $(y,s)\in\mathbb R^{n}$. If $\Sigma = \{(x, \psi(x,t), t): (x,t) \in \mathbb{R}^{n-1} \times \mathbb{R}\}$ in the coordinates $P_x \times (P_x)^\perp \times \re$, for some $t$-independent plane $P \in \mathcal{P}$ and Lip(1,1/2) function $\psi$, then we say that $\Sigma$ is a Lip(1,1/2) graph. We say that $ \psi = \psi ( x, t ) : \mathbb R^{n-1}\times\mathbb R\to \mathbb R$ is a {\it regular} Lip(1,1/2)  function
with parameters $b_1$ and $b_2$,
if $\psi$ satisfies \eqref{1.1} and if
 \begin{eqnarray} \label{1.7}
 D_{1/2}^t\psi\in BMO(\mathbb R^n), \ \ \|D_{1/2}^t\psi\|_*\leq b_2<\infty.
\end{eqnarray}
  If $\Sigma = \{(x, \psi(x,t), t): (x,t) \in \mathbb{R}^{n-1} \times \mathbb{R}\}$ in the coordinates $P_x \times (P_x)^\perp \times \re$, for some $t$-independent plane $P \in \mathcal{P}$ and {\it regular} Lip(1,1/2) function $\psi$, then we say that $\Sigma$ is a {\it regular} Lip(1,1/2) graph.
\end{definition}

\begin{definition}\label{def1.ADR} Let
$\Sigma \subset \mathbb R^{n+1}$ be a closed set.  We say that a measure $\mathfrak{m}$, defined
on $\Sigma$, is {\em parabolic Ahlfors-David regular}, parabolic ADR for short
(or simply p-ADR, or just ADR)
with constant $M\geq 1$,  if
\begin{equation} \label{eq1.ADRha-general}
M^{-1}\, r^d \leq \mathfrak{m}(\Delta(X,t,r)) \leq M\, r^d,\end{equation}
whenever $0<r<\diam{\Sigma}$,  $(X,t)\in \Sigma$, $T_0<t<T_1$ and where $\diam{\Sigma}$ is the
(parabolic) diameter of $\Sigma$ (which may be infinite). If \begin{equation} \label{eq1.ADRha}
M^{-1}\, r^d \leq \sigma(\Delta(X,t,r)) \leq M\, r^d,\end{equation}
then we simply say that
$\Sigma$ is  parabolic ADR (p-ADR, or just ADR).
\end{definition}

As noted above (see Remark \ref{r-measures} (i)), if
\eqref{eq1.ADRha-general} holds for a measure $\mathfrak{m}$ on $\Sigma$, then it holds for
$\sigma$ as in \eqref{sigdef}, i.e. \eqref{eq1.ADRha} holds for a possibly different but still universal choice of $M$.

Following \cite{HLN}, \cite{HLN1} we next introduce parabolic $\beta_2$ numbers, which we will denote by $\gamma$. These are defined in a way analogous to the usual $\beta$ numbers in \cite{DS1,DS2} except that they incorporate  parabolic scalings, parabolic surface measure and they only allow for approximation by $t$-independent planes. Indeed, assume that $\Sigma  \subset \mathbb R^{n+1}$ is parabolic ADR  in the sense of Definition \ref{def1.ADR}. We then let
\begin{eqnarray*} \gamma( Z, \tau, r  ):=\gamma_\Sigma( Z, \tau, r  ):= \inf_{P \in \mathcal{P}}  \biggl ( \, \bariint_{  \Delta ( Z, \tau,r) }  \, \biggl (\frac {d_p ( Y,s, P )}{r}\biggr )^2  \d \sigma (Y, s )\biggr )^{1/2}, \end{eqnarray*}
whenever $(Z,\tau)\in \Sigma $, $r>0$, and where  $\mathcal{P}$  is the set of $ n $-dimensional hyperplanes $ P $ containing a line
parallel to the $ t $ axis. We also introduce
\begin{eqnarray}\label{me}\quad  \d \nu( Z, \tau,
r  ) :=\d \nu_\Sigma ( Z, \tau,
r  )  \, =  \, (\ga_\Sigma  ( Z, \tau, r))^2 \, \d \si ( Z, \tau) \, r^{ - 1 }
\d r. \end{eqnarray}
Recall that $ \nu$  is defined  to be  a Carleson measure on $    \Delta( Y,s,R )  \times ( 0, R )$,  if there exists $ \Gamma <
\infty $ such that
\begin{eqnarray}\label{1.9}
 \nu ( \Delta( X, t,\rho )  \times ( 0, \rho) ) \, \leq \,
\Gamma  \, \rho^{d},
\end{eqnarray}
 whenever $ ( X, t  ) \in \Sigma  $ and $ C_\rho ( X, t ) \subset C_R ( Y, s )$.  The least such $ \Gamma  $ in \eqref{1.9} is
called the Carleson norm of $\nu$ on  $\Delta( Y,s,R ) \times ( 0, R ) $.

We are now ready to give the definition of parabolic uniform rectifiability.

\begin{definition}\label{def1.UR}
Assume that $\Sigma  \subset \mathbb R^{n+1}$ is parabolic ADR  in the sense of Definition \ref{def1.ADR} with constant $M$. Let $\nu=\nu_\Sigma$ be defined as in \eqref{me}. Then
 $\Sigma$ is parabolic uniformly rectifiable,  parabolic UR  for short,  with UR constants $(M,\Gamma)$ if
\begin{eqnarray}\label{eq1.sf}
\| \nu \|:=\sup_{(X,t)\in\Sigma,\ \rho>0}  \rho^{ -d }\nu ( \Delta(X,t,\rho) \times ( 0, \rho) ) \, \leq \,
\Gamma.
\end{eqnarray}
\end{definition}

We will also need a parabolic dyadic decomposition of $\Sigma$. This is a special case of a more general result in \cite{Ch}.

\begin{lemma}\label{cubes}  Assume that $\Sigma  \subset \mathbb R^{n+1}$ is parabolic ADR  in the sense of Definition \ref{def1.ADR} with constant $M$. Then $\Sigma$ admits a parabolic dyadic decomposition in the sense that there exist
constants $ \alpha>0,\, \beta>0$ and $c_*<\infty$,  such that for each $k \in \mathbb{Z}$
there exists a collection of Borel sets, $\mathbb{D}_k$,  which we will call (dyadic) cubes, such that
$$
\mathbb{D}_k:=\{Q_{j}^k\subset\Sigma: j\in \mathfrak{I}_k\},$$ where
$\mathfrak{I}_k$ denotes some  index set depending on $k$, satisfying
\begin{eqnarray*}%\label{eqcubes}
(i)&&\mbox{$\Sigma=\cup_{j}Q_{j}^k\,\,$ for each
$k\in{\mathbb Z}$.}\notag\\
(ii)&&\mbox{If $m\geq k$ then either $Q_{i}^{m}\subset Q_{j}^{k}$ or
$Q_{i}^{m}\cap Q_{j}^{k}=\emptyset$.}\notag\\
(iii)&&\mbox{For each $(j,k)$ and each $m<k$, there is a unique
$i$ such that $Q_{j}^k\subset Q_{i}^m$.}\notag\\
(iv)&&\mbox{$\diam\big(Q_{j}^k\big)\leq c_* 2^{-k}$.}\notag\\
(v)&&\mbox{Each $Q_{j}^k$ contains $\Sigma\cap C_{\alpha2^{-k}}(Z^k_{j},t^k_{j})$ for some $(Z^k_{j},t^k_j)\in\Sigma$.}\notag\\
(vi)&&\mbox{$\sigma(\{(Z,t)\in Q^k_j: d_p(Z,t,\Sigma\setminus Q^k_j)\leq \varrho \,2^{-k}\big\})\leq
c_*\,\varrho^\beta\,\sigma(Q^k_j),$}\notag\\
&&\mbox{for all $k,j$ and for all $\varrho\in (0,\alpha)$.}
\end{eqnarray*}
\end{lemma}

\begin{remark} We denote by  $\mathbb{D}=\mathbb{D}(\Sigma)$ the collection of all $Q^k_j$, i.e., $$\mathbb{D} := \cup_{k} \mathbb{D}_k.$$ For a dyadic cube $Q\in \mathbb{D}_k$, we let $\ell(Q) := 2^{-k}$, and we will refer to this quantity as the size
of $Q$.  Evidently, $\ell(Q)\sim\diam(Q)$ with constants of comparison depending at most on $n$ and $M$.
Note that $(iv)$ and $(v)$ of Lemma \ref{cubes} imply, for each cube $Q\in\mathbb{D}_k$,
that there is a point $(X_Q,t_Q)=(X_Q,t_Q)\in \Sigma$, and  a cube $C_{r}(X_Q,t_Q)$, such that
%$r\approx 2^{-k} \approx \diam (Q)$ and 
\begin{equation}\label{cube-ball}
r\approx 2^{-k} \approx \diam (Q)\,,\quad \text{and}\quad
\Sigma\cap C_{r}(X_Q,t_Q)\subset Q \subset \Sigma\cap C_{cr}(X_Q,t_Q),\end{equation}
for some uniform constant $c$. We will denote the associated surface ball by
\begin{equation}\label{cube-ball2}
\Delta_Q(r):= \Sigma\cap C_{r}(X_Q,t_Q),\end{equation}
and we shall refer to the point $(X_Q,t_Q)$ as the center of $Q$. Given a dyadic cube $Q\subset\Sigma$ and $\lambda > 1$, we define
\begin{equation}\label{dilatecube}
\lambda Q:= \Delta_Q(\lambda \diam(Q)).
\end{equation}
\end{remark}

\begin{definition}\cite{DS1}\label{d3.11}
Let $\sbf\subset \dd(\Sigma)$. We say that $\sbf$ is
coherent if the following conditions hold.
\begin{itemize}\itemsep=0.1cm

\item[$(a)$] $\sbf$ contains a unique maximal element $Q(\sbf)$ which contains all other elements of $\sbf$ as subsets.

\item[$(b)$] If $Q$  belongs to $\sbf$, and if $Q\subset \widetilde{Q}\subset Q(\sbf)$, then $\widetilde{Q}\in {\bf S}$.

\item[$(c)$] Given a cube $Q\in \sbf$, either all of its children belong to $\sbf$, or none of them do.

\end{itemize}
We say that $\sbf$ is semi-coherent if only conditions $(a)$ and $(b)$ hold.
\end{definition}

\section{Statement of the main results}\label{mr}
As stated in the introduction, the purpose of this work is to prove (a) $\implies$ (b) and (a) $\implies$ (c) in Theorem \ref{thrm.char}.  To prove (a) $\implies$ (b) we show the following.

\begin{theorem}\label{main}  Suppose that
 $\Sigma\subset\mathbb R^{n+1}$ is parabolic uniformly rectifiable with constants $(M,\Gamma)$.
 Let $\delta\ll 1$
and $\kappa\gg 1$ be two given  positive constants. Then there exists a disjoint decomposition
$\dd(\Sigma) = \B\cup\G$, satisfying the following properties.
\begin{enumerate}
\item The collection $\B$ satisfies the Carleson
packing condition:
$$\sum_{Q'\subseteq Q, \,Q'\in\B} \sigma(Q')
\leq\, c(n,M,\Gamma,\delta,\kappa)\, \sigma(Q)\,,
\quad \forall Q\in \dd(\Sigma)\,.$$
\item  The collection $\G$ is further subdivided into
disjoint stopping time regimes $\{\sbf\}_{\sbf \in \mathcal{F}}$, such that each such regime {\bf S} is coherent.

\item The maximal cubes $Q(\sbf)$ for the  stopping time regimes satisfy the Carleson
packing condition:
$$\sum_{\sbf: Q(\sbf)\subseteq Q}\sigma\big(Q(\sbf)\big)\,\leq\, c(n,M,\Gamma,\delta,\kappa)\, \sigma(Q)\,,
\quad \forall Q\in \dd(\Sigma)\,.$$
\item For each $\sbf$,  there exists a coordinate system\footnote{This means that we identify $P_x \times P_x^\perp \times \mathbb{R}$ with $\mathbb{R}^n \times \mathbb{R} $ for some $t$-independent plane $P \in \mathcal{P}$, see Definition \ref{tindplane.def}.} and a
regular Lip(1,1/2)  function $ \psi_{\sbf}:=\psi = \psi ( x, t ) : \mathbb R^{n-1}\times\mathbb R\to \mathbb R$ with
 parameters $b_1=c(n,M)\cdot\delta$ and $b_2=b_2(n,M,\Gamma)$, such if we define
 $\Sigma_{\sbf}:=\Sigma_\psi:=\{(x,\psi(x,t),t):\ (x,t)\in \mathbb R^{n-1}\times\mathbb R\}$, then
\begin{equation}\label{eq2.2a-RE}
\sup_{(X,t)\in \kappa Q} d_p(X,t,\Sigma_{\sbf} )\,
\leq\, \delta\,\diam(Q),
\end{equation}
for all $Q\in \sbf$.
\end{enumerate}
\end{theorem}

The second theorem we prove  is a strengthening of Theorem \ref{main} in the sense that it  gives instead a bilateral approximation in  (4) and proves (a) $\implies$ (c) in Theorem \ref{thrm.char}.
\begin{theorem}\label{main.bil.thrm}  Suppose that
 $\Sigma\subset\mathbb R^{n+1}$ is parabolic uniformly rectifiable with constants $(M,\Gamma)$. Let $\delta\ll 1$
and $\kappa\gg 1$ be two given  positive constants. Then there exists a disjoint decomposition
$\dd(\Sigma) = \B^*\cup\G^*$, satisfying the following properties.
\begin{enumerate}
\item The collection $\B^*$ satisfies the Carleson
packing condition:
$$\sum_{Q'\subseteq Q, \,Q'\in\B^*} \sigma(Q')
\leq\, c(n,M,\Gamma,\delta,\kappa)\, \sigma(Q)\,,
\quad \forall Q\in \dd(\Sigma)\,.$$
\item  The collection $\G^*$ is further subdivided into
disjoint stopping time regimes $\{\sbf^*\}_{\sbf^* \in \mathcal{F}^*}$, such that each such regime $\sbf^*$ is coherent.

\item The maximal cubes $Q(\sbf^*)$ for the stopping time regimes satisfy the Carleson
packing condition:
$$\sum_{\sbf^*: Q(\sbf^*)\subseteq Q}\sigma\big(Q(\sbf^*)\big)\,\leq\, c(n,M,\Gamma,\delta,\kappa)\, \sigma(Q)\,,
\quad \forall Q\in \dd(\Sigma)\,.$$
\item For each $\sbf^*$,  there exists a coordinate system and a
regular  Lip(1,1/2)  function $ \psi_{\sbf^*}:=\psi = \psi ( x, t ) : \mathbb R^{n-1}\times\mathbb R\to \mathbb R$ with
 parameters $b_1=c(n,M)\cdot\delta$ and $b_2=b_2(n,M,\Gamma)$, such if we define
 $\Sigma_{\sbf^*}:=\Sigma_\psi:=\{(x,\psi(x,t),t):\ (x,t)\in \mathbb R^{n-1}\times\mathbb R\}$, then
\begin{align*}
\sup_{(X,t)\in \kappa Q} d_p(X,t,\Sigma_{\sbf^*} )\, + \,  \sup_{(Y,s)\in C_{\kappa \diam(Q)}(X_Q,t_Q) \cap \Sigma_{\sbf^*}} d_p(Y,s, \Sigma) \,
\leq\, \delta\,\diam(Q),
\end{align*}
for all $Q\in \sbf^*$.
\end{enumerate}
\end{theorem}

\subsection{The proof of Theorem \ref{main}}\label{s4}
We here briefly discuss the outline of the proof of Theorem \ref{main} to be presented in subsequent sections.

 In the sequel,
 $\Sigma\subset\mathbb R^{n+1}$ is parabolic uniformly rectifiable with constants $(M,\Gamma)$.
Throughout the proof, four parameters will consistently appear: $K$, $\epsilon$, $\delta$ and $\kappa$. The parameters $\delta$ and $\kappa$ are exactly as in the statement of Theorem \ref{main}.
$K$ and $\epsilon$ are auxiliary parameters. Both $\epsilon$ and $\delta$ will be small, $\epsilon\ll\delta\ll 1$, and $\epsilon$ is chosen/defined in \eqref{geo9-}. This implies that
$\epsilon=\epsilon(n,M,\Gamma,\kappa,K,\delta)$. $K$ and $\kappa$ will be large, $K\gg\kappa\gg 1$, and to achieve that the constants in Theorem \ref{main} only depend on $n,M,\Gamma,\delta$ and $\kappa$, we will enforce the relation $K/\kappa = c(n,M)\gg 1$.

For $K \gg\kappa\gg 1$ fixed, we define,  for each $Q \in \mathbb{D}(\Sigma)$,
$$\mbox{$\beta_\infty(Q):= \beta_\infty(Q,K)$ and $\beta_2(Q):= \beta_2(Q,K)$,}$$
 where
\begin{equation}\label{3}
\beta_\infty(Q,K) :=  \inf_{P \in \mathcal{P}} \diam(Q)^{-1} \sup_{\{(Y,s)\in 8KQ\}} d_p(Y,s,P)\,,
\end{equation}
and
\begin{equation}\label{4}
\quad\quad\quad\beta_2(Q,K) := \inf_{P \in\mathcal{P}} \left(\diam(Q)^{-d} \iint_{8KQ} \biggl (\frac{ d_p(Y,s,P)}{\diam(Q)}\biggr )^2 \,d\sigma(Y,s)\right)^{1/2}.
\end{equation}
In the following we drop $K$ to ease the notation. Note that
\begin{equation}\label{5}
\beta_\infty(Q)^{n+3}\lesssim \beta_2^{2}(Q^*),
\end{equation}
where $Q^*$ is an appropriately chosen ancestor to $Q$ such that  $$100c\diam(Q) \le \diam(Q^*) \lesssim \diam(Q)$$ for $c=c(n,M)\gg 2$ to be chosen. Indeed, let $ \ti P $ be a plane realizing the infimum in
the definition of $\beta_2(Q^*)$. Clearly  there exists $ (Z_1, \tau_1 )\in 8KQ$ with
\begin{equation*} \rho := \diam(Q) \beta_\infty(Q)\leq d_p( Z_1, \tau_1 , \ti P ).
\end{equation*}
We have $ \si ( \De ( Z_1, \tau_1, \rho/8) )\gtrsim \rho^{d}$ as $\Sigma$ is parabolic ADR.  Moreover every
 point in $ \De ( Z_1, \tau_1, \rho/8) $ is contained in
$8cKQ $, for some $c=c(n,M)\gg 2$, and lies at least $ \rho/2 $ from $ \ti
P. $  Thus, \eqref{5} is valid.

Using that $\nu$ is a Carleson measure, a standard argument (see e.g. \cite[Lemma 4.8]{Rigot}) shows that
\begin{equation}\label{beta2cmest.eq}
\sum_{Q \subset Q^*} \beta_2^2(Q) \sigma(Q) \lesssim \sigma(Q^*), \quad  Q^* \in \mathbb{D}(\Sigma),
\end{equation}
with implicit constant depending on $K, n, M$ and $\Gamma$.

To this end, we define
\begin{align*}
\mathcal{B} &:= \mathcal{B}(\epsilon,K) := \{Q \in \mathbb{D}(\Sigma): \beta_\infty(Q) \ge \epsilon\},\\
\mathcal{G} &:= \mathcal{G}(\epsilon, K) = \mathbb{D}(\Sigma) \setminus  \mathcal{B}(\epsilon,K).
\end{align*}
Then \eqref{5} and \eqref{beta2cmest.eq} imply that
\begin{equation}\label{bad-pack}
\sum_{\substack{Q \subseteq {\widetilde{Q}}\\ Q \in \mathcal{B}}}  \sigma(Q)
\lesssim \sigma( \widetilde{Q}), \quad  \widetilde{Q} \in \mathbb{D}(\Sigma),
\end{equation}
where the implicit constant now depends on $\epsilon, K, n, M$ and $\Gamma$.

For $Q \in \mathbb{D}(\Sigma)$ we let $P_Q \in \mathcal{P}$ be a $t$-independent plane which achieves the infimum in the definition of $\beta_\infty(Q)$. Then, by the definition of $\mathcal{G}$,
\begin{equation}\label{3++}
 d_p(Y,s,P_{Q})\leq\epsilon\diam(Q)\mbox{ for all }(Y,s)\in 8KQ,
\end{equation}
and for all $Q \in \mathcal{G}$.

Following \cite[Section 7]{DS1} (or equivalently,
\cite[Lemma IV.2.35]{DS2}) we may
construct an augmented collection $\mathcal{B}'\supset \mathcal{B}$, and a
refined collection
$\mathcal{G}' \subset \mathcal{G}$, for which we have a
disjoint decomposition of the dyadic cubes
$\mathbb{D}(\Sigma) = \mathcal{B}' \cup \mathcal{G}'$. Moreover, $\mathcal{B}'$
still satisfies the packing condition \eqref{bad-pack},
and $\mathcal{G}'$ may be partitioned into
 a collection of coherent
stopping time regimes $\{\sbf\}_{\sbf \in \mathcal{F}}$, with maximal cubes $\{Q(\sbf)\}_{\sbf \in \mathcal{F}}$,
with the following properties:
\begin{align}\label{dico}
&\mbox{(A) If $Q \in \sbf$, then $\Ang(P_Q, P_{Q(\sbf)}) \le \delta$.}\notag\\
&\mbox{(B) If $Q$ is a minimal cube of $\sbf$, then either a child of $Q$ is in $\mathcal{B}'$}\notag\\
&\quad\quad\mbox{or $\Ang(P_Q,P_{Q(\sbf)}) \ge \delta/2$.}
\end{align}

Given a stopping time regime $\sbf$, we define $m(\sbf)$ to be the collection of minimal cubes in $\sbf$.

Following \cite{DS1}, we will construct a Lip(1,1/2) graph associated to each stopping time regime. After doing so, proving Theorem \ref{main} is a matter of demonstrating that the graph is, in fact, a {\it regular} Lip(1,1/2) graph and  that the maximal cubes $\{Q(\sbf)\}$ pack. These objectives will be achieved in a somewhat simultaneous manner in a way similar to \cite{DS1}. The insight (and courage) of David and Semmes in providing such a proof is quite remarkable.

The proof of Theorem \ref{main} is divided up into three sections:
in Section \ref{subsgraphconst}, we construct a Lip(1,1/2) graph with 
small constant for each stopping time regime.
In Section \ref{pushsqfn}, we push the geometric 
square function estimate (the Carleson measure estimate for $\nu$) to the graph. 
The control of the geometric square function on the graph will verify, by \cite{HLN,HLN1}, that the graph is indeed regular.
In Section \ref{packsttime}, we use the fact that we pushed the geometric square function in a sharp manner to pack the stopping time regimes.
%\end{itemize}

Our argument will also rely on two simple lemmas concerning the approximation of surface cubes by planes, these lemmas can be found in Appendix \ref{subsplanes}. The construction of the graph relies on condition  \eqref{dico} (A) and that the geometric lemma holds for $\Sigma$. It is also worth  making the remark that we only use condition  \eqref{dico} (B)  above when packing the stopping time regimes in Section \ref{packsttime}.

\section{Construction of  Lip(1,1/2) graphs for stopping time regimes}\label{subsgraphconst}
In this section we shall follow very closely the arguments in \cite[Chapter 8]{DS1} to 
construct, for $\sbf \in \mathcal{F}$ fixed, a Lip(1,1/2) graph which approximates $Q \in \sbf$ as in Theorem \ref{main}(4). We can without loss of generality assume that 
$$P_{Q(\sbf)} := \{(x,x_n, t) \in \mathbb{R}^{n-1} \times \mathbb{R} \times \mathbb{R}: x_n = 0\}.$$
We will often identify $P_{Q(\sbf)}$ with $\mathbb{R}^n$ and employ the notation $C'_r(x,t)$ introduced in \eqref{cube2}  for a $n$-dimensional parabolic cube in $P_{Q(\sbf)}$. We let $\pi$ denote the projection onto $P_{Q(\sbf)}$, that is, $\pi(x,x_n,t) = (x,t) \in \rn$. We let $\pi^{\perp}$ denote the projection onto the normal to $P_{Q(\sbf)}$, that is, $\pi^\perp(x,x_n,t) = x_n$. Given $Q \in \mathbb{D}(\Sigma)$ we let $(x_Q, t_Q): = \pi(X_Q,t_Q)$ denote the projected center of $Q$.

Given $\sbf\in\mathcal{F}$, and following \cite{DS1}, we introduce distance functions adapted to the stopping regime. Indeed, if $(X,t)\in\mathbb R^{n+1}$ then we define
$$d(X,t)=d_{\sbf}(X,t):=\inf_{Q\in\sbf}[d_p(X,t,Q)+\mbox{diam}(Q)],$$
and we let $D=D_{\sbf} : \mathbb R^n\to [0,\infty)$ be
defined as
\begin{align}\label{6-}D(x,t):=\inf_{(X,t)\in \pi^{-1}(x,t)}d(X,t)=\inf_{Q\in\sbf}[d_p(x,t,\pi(Q))+\mbox{diam}(Q)].
\end{align}
In the sequel, we set
\[ F:=\big\{(X,t)\in \Sigma: d(X,t)=0\big\}\,.\]

\begin{lemma}\label{lem1} Let $\sbf\in\mathcal{F}$ and let $P_{ Q(\sbf)}=\mathbb R^n$.
Then $\pi$ is one-to-one on $F$ and we can define a function $\hat\psi$ on $\pi(F)\subset \mathbb R^n$ as follows. Given
$(X,t)=(x,x_n,t)\in F$, we set  $\hat\psi(x,t):=\pi^\perp(X,t)=x_n$.  Then 
$\hat \psi$ is a well-defined Lip(1,1/2) function with constant bounded by $2\delta$ on $\pi(F)$, i.e.,
\begin{eqnarray}\label{hatlip}
|\hat \psi(x,t)-\hat \psi(y,s)|\leq 2\delta d_p(x,t,y,s),
\end{eqnarray}
whenever $(x,t), (y,s)\in \pi(F)$.  More generally, suppose that 
$(X,t),\ (Y,s)\in 16 \kappa Q(\sbf)$, and that for some $N\geq1$, we have $\min\{d(X,t),d(Y,s)\} \le \,N
d_p(X,t, Y,s)$; then
\begin{align}\label{6}
|\pi^\perp(X,t)-\pi^\perp(Y,s)|\leq 2\delta d_p\big(\pi(X,t),\pi(Y,s)\big)\,,
\end{align}
provided that $\epsilon \ll \delta$ is small enough, depending on $N$, and $K\gg 1$ large enough, 
depending on $N$ and $\kappa$. ({\tt Remark:   in the sequel, we apply the lemma 
 with $N$ equal to an absolute constant}.)
\end{lemma}

\begin{proof} Clearly, it suffices to prove \eqref{6}.
The proof follows that of \cite[Lemma 8.4]{DS1}. Let 
$(X,t),\ (Y,s)\in 16 \kappa Q(\sbf)$ and without loss of generality, suppose 
that $d(X,t) \le N
d_p(X,t, Y,s)$. 
Let $Q\in \sbf$ be such that
\begin{align}\label{6a}d_p(X,t,Q)+\mbox{diam}(Q)\leq 2N d_p(X,t,Y,s).
\end{align}
We can, if necessary, replace
$Q$ by one of its ancestors (which is still in $\sbf$) to obtain 
$\kappa^{-1} d_p(X,t,Y,s)\lesssim  \diam Q\lesssim N d_p(X,t,Y,s)$.  
By construction, $Q \in \mathcal{G}$ and hence 
$\beta_\infty(Q)\leq\epsilon$, so for  $K=K(N,\kappa)$ large enough,
there is a plane $P_Q$ such that
$$ d_p(X,t,P_Q)+ d_p(Y,s,P_Q)\lesssim  \epsilon \diam Q \lesssim
 \epsilon N d_p(X,t,Y,s)\ll \delta  d_p(X,t,Y,s),$$
as $\epsilon\ll \delta$. % and where we have also used \eqref{6a}. 
The conclusion 
in \eqref{6} now follows since, by construction of the stopping time, the angle between 
$P_{Q(\sbf)}$ and $P_Q$ is at most $\delta$.  In particular, the reader can verify this claim using elementary geometry
and the fact that $\sin^2(\delta) \approx \delta^2$ for $\delta$ small.
\end{proof}

With Lemma \ref{lem1} as the starting point we next construct a function  $\psi$ on $\mathbb R^n=P_{ Q(\sbf)}$ which coincides with $\hat\psi$ on $\pi(F)$. We will
extend ${\hat\psi}$ off $\pi( F)$ using an appropriate Whitney type extension following the corresponding construction in \cite{DS1}. For $(x,t)\in  \mathbb R^n$ and $(x,t)$ 
not on the boundary of any dyadic cube, let $I_{(x,t)}$ be the largest (closed)
dyadic cube in
$\mathbb R^n$ containing $(x,t)$ and satisfying
$$\diam(I_{(x,t)})\leq \frac 1 {20}\inf_{(z,\tau)\in I_{(x,t)}}D(z,\tau).$$
Let $\{I_i\}$ be a labelling of the set of all these cubes $I_{(x,t)}$ without repetition. By construction the collection
$\{I_i\}$ consists of pairwise non-overlapping closed dyadic cubes in $\mathbb R^n$. These cubes cover $ \mathbb R^{n} \sem  \pi( F) $ and they do not intersect
$\pi( F)$.  For the construction we need the following, whose proof 
follows that of \cite[Lemma 8.7]{DS1}, essentially verbatim.
\begin{lemma}\label{10-60lemma}
Given $I_i$ we have
\begin{align}\label{10-60.eq}
10 \diam I_i \leq D(y,s) \leq 60 \diam I_i \indent \forall (y,s) \in 10I_i.
\end{align}
In particular, if  $10I_i \cap 10I_j \neq \emptyset$, then
\begin{align}\label{eq1}
\diam I_i \approx \diam I_j.
\end{align}
\begin{proof}
Fix $i$ and let $(y,s) \in I_i$.  As $D(\cdot,\cdot)$ is Lip$(1,1/2)$ with norm $1$, we have
\[D(y,s) \geq \inf_{(x,t)\in I_i} D(x,t) - 10\diam I_i \geq 10 \diam I_i,\]
where we used that $\diam I_i \leq 20^{-1} \inf_{(x,t)\in I_i} D(x,t)$.  On the other hand, if $I$ is the dyadic parent of $I_i$, then there exists $(z,\tau) \in I$ such that $D(z,\tau) <20 \diam I = 40 \diam I_i$.  Thus,
\[D(y,s) \leq D(z,\tau) + 20 \diam I_i \leq 60\diam I_i.\]
This proves \eqref{10-60.eq}, and hence also \eqref{eq1}.
\end{proof}
\end{lemma}
We continue to follow the constructions in \cite[Chapter 8]{DS1}.   Set
\begin{align}\label{nota2}R=R_{\sbf}:=\diam ( Q(\sbf)).
\end{align}
Given the center $(X_{ Q(\sbf)},t_{ Q(\sbf)})$ of $ Q(\sbf)$, we recall that
$$(x_{ Q(\sbf)},t_{ Q(\sbf)})=\pi(X_{ Q(\sbf)},t_{ Q(\sbf)}),$$ and we introduce the index set
\begin{equation}\label{la-def}
 \La \, = \{ i :
 I_i \cap C'_{2\kappa R}(x_{ Q(\sbf)},t_{ Q(\sbf)}) \not = \emptyset \}.
\end{equation}

We claim, for each $i \in \Lambda$, that there exists $Q(i) \in \sbf$ such that
\begin{align}\label{eq1+}
\kappa^{-1} \diam I_i \lesssim  \diam Q(i) \leq 120 \diam I_i, \quad d_p(\pi(Q(i)),I_i) \leq 120 \diam I_i.
\end{align}
To see this, note that given $I_i$ and $(x,t)\in I_i$, there exists a cube $Q\in \sbf$ such that
$d_p(x,t,\pi(Q))+\mbox{diam}(Q)\leq 2D(x,t)\sim \diam (I_i)$.
Moreover, $D(x,t) \le c \kappa \diam(Q(\sbf))$, for a constant $c\geq 1$ 
depending only on $n$, $M$, and 
\[
d_p(x,t,\pi(Q^*)) \le 2D(x,t)\leq 120 \diam I_i\,,
\] 
for all $Q  \subseteq Q^* \subseteq Q(\sbf)$. Therefore we can choose $Q(i)$ so that  $Q  \subseteq Q(i) \subseteq Q(\sbf)$ and
\begin{equation}\label{IiQi.eq}
(c \kappa)^{-1}D(x,t) \le \diam(Q(i)) \le 2D(x,t) \leq 120 \diam I_i\,,
\end{equation}
provided $\kappa$ is chosen large enough.

In the sequel, we will, at instances, use the notation
\begin{align}\label{nota1}r_i:=\diam I_i,\,\, \text{ so that } \,  r_i/\kappa \lesssim
\diam Q(i) \lesssim r_i,
\end{align}
by \eqref{eq1+}.

Let $ \{ \tilde \nu_i \} $ be a class of infinitely differentiable functions on $\mathbb R^{n}$ such that
\begin{align*}
 (i)&\ \tilde \nu_i  \equiv 1  \mbox{ on }  2I_i
 \mbox{ and }   \tilde \nu_i \equiv 0  \mbox{ in } \mathbb R^{n} \sem 3I_i \mbox{ for all } i,\notag\\
 (ii)&\ r_i^{ l} \, \left| \frac{\partial^l}{\partial x^l }  \tilde \nu_i \right|
\, + \,    r_i^{ 2 l}  \left| \frac{\partial^l}{\partial t^l } \tilde \nu_i \right|  \,  \leq
\, c ( l, n ) \mbox{ for } l = 1, 2, \dots.
\end{align*}
In ($ii$), $
\frac{\ar^l}{\ar x^l } $ denotes an arbitrary partial derivative with respect to the space variable $ x $ and of order $l$.  {We also introduce
$ \{ \nu_i \} $ defined as a partition of unity adapted to $ \{ 2I_i \}_{i}$, i.e.,
$$\nu_i(x,t)=\tilde \nu_i(x,t)\big (\sum_{j} \tilde \nu_j(x,t)\big)^{-1}.$$}
Then $ \{  \nu_i \} $  is also a class of infinitely differentiable functions on $\mathbb R^{n}$ and
\begin{eqnarray}\label{2.16+}
r_i^{  l} \, \left| \frac{\partial^l}{\partial x^l }  \nu_i \right|
\, + \,    r_i^{  2 l}  \left| \frac{\partial^l}{\partial t^l }  \nu_i \right|  \,  \leq
\, c ( l, n ) \mbox{ for } l = 1, 2, \dots.
\end{eqnarray}

We now ready to construct the extension of $  \hat \psi  $ off $ \pi( F) $. {If $i \in \Lambda$, then for $I_i$} we have an associated dyadic cube $Q(i)\in \sbf$ as in \eqref{eq1+}, we first express the associated hyperplane $P_{Q(i)}$ as
\begin{align}\label{2.19-} P_{Q(i)}=\{ (x, B_i  ( x,t),t ) : ( x,t) \in \mathbb R^n \},
\end{align}
i.e., $B_i:\mathbb R^n\to \mathbb R$ is the affine function whose graph is $P_{Q(i)}$. {If $i \not \in \Lambda$, set $B_i = 0$ and we note, in this case, $B_i$ is the affine function whose graph is $P_{Q(\sbf)}$.} Using this notation we let
\begin{align}\label{2.19}
\psi ( x,t)&=\hat \psi ( x,t)\mbox{ when } ( x,t) \in \pi (
F),\notag\\
\psi ( x,t) &=
 { \ds \sum_{ i } } \,  B_i ( x,t)
\,   \nu_i ( x,t ) \mbox{ when } ( x,t ) \in \mathbb R^n \sem \pi (
F). \end{align}
This defines $\psi $ on $\mathbb R^n$, with $\psi=\hat\psi$ on
$\pi (F)$ and by construction
\begin{equation}\label{supportpsi}
\psi\equiv 0 \text{ on } \mathbb R^n\setminus C'_{4\kappa R}(x_{ Q(\sbf)},t_{ Q(\sbf)}).
\end{equation}
Based on the construction we can prove the following lemma.
\begin{lemma}\label{lemma2.21}
For $ ( y,s ),\  (z, \tau) \in \mathbb R^{n}$ we have
\begin{align*}
(i)&\  | B_i ( y, s ) - B_i ( z, \tau ) | \, \lesssim \delta d_p(y,s,z,\tau) \mbox{ for all } i, \notag\\
(ii)&\  | B_i ( y, s ) - B_j ( y, s ) |\lesssim \epsilon  \min ( r_i, r_j ),  \mbox{ if  }
 ( y, s ) \in 100 I_i \cup 100 I_j ,\   10  I_i  \cap 10  I_j  \not =\emptyset.
 \end{align*}
\end{lemma}
\begin{proof}
 By construction of the stopping regime, the Lip(1,1/2) 
 norm of $B_i$ is bounded by $2\delta$, so $(i)$ is trivial.
 To prove $(ii)$, we follow the proof of \cite[Lemma 8.17]{DS1}.  
 Let $I_i$ and $I_j$ be two cubes such that $10 I_i  \cap 10 I_j  \not =
\emptyset$, thus $\diam I_i=r_i \approx r_j=\diam I_j$ by Lemma \ref{10-60lemma}.
% where here and below in the present proof, implicit constants may depend on $\kappa$.  
We claim that if $i \in \Lambda$ or $j \in \Lambda$, then
\begin{align}\label{6++-}d_p(Q(i),Q(j))\lesssim r_i\approx r_j.
\end{align}
Let us prove the claim.
Let $i \in \Lambda$.  %Then by \eqref{nota1},
%\begin{equation}\label{ijcomp}
% \diam Q(i) \approx  r_i \approx r_j \approx  \diam Q(j)\,.
%\end{equation}  
Let $(X,t)\in Q(i)$ and $(Y,s)\in Q(j)$,
We may assume that
$d_p(X,t,Y,s)\geq \diam Q(i)$, the other case being trivial.
Since $d(X,t)\leq \diam Q(i)$, by \eqref{6} we see that
\begin{align}\label{6++}
|\pi^\perp(X,t)-\pi^\perp(Y,s)|\leq d_p(\pi(X,t),\pi(Y,s)).
\end{align}
By \eqref{eq1+}, % and \eqref{ijcomp}, 
$d_p(\pi(X,t),\pi(Y,s))\lesssim r_i\approx r_j$,  so the claim \eqref{6++-} holds.
%hence
%\begin{align*}d_p(Q(i),Q(j))\leq d_p(X,t,Y,s)\lesssim \min ( r_i, r_j )
%\end{align*}
%giving the validity of \eqref{6++-}.

 Suppose that $i \in \Lambda$. 
 %, and that  $\rho_j \leq \rho_i$ and hence also that $r_j \lesssim r_i$.  
 Then by \eqref{6++-} and \eqref{eq1+}, %\eqref{ijcomp}, 
 there is a ``surface ball" $\Delta_r$, centered in 
 $Q(i)$, such that $r\approx r_i\approx r_j$, and
\[Q(i) \cup Q(j) \subset \Delta_r \subset \Delta_{2r} \subset KQ(i)\cap KQ(j)\,,\]  
for $K$ large enough. 
 %as $\diam Q(i) \approx \diam Q(j)$, see \eqref{IiQi.eq}.  
Using %\eqref{6++-}, and 
 the second part of 
 Lemma \ref{aux2} (with $\Delta=\Delta_r$, $P_1=P_{Q(i)}$, $P_2=P_{Q(j)}$),  
 we see that if $( y, s ) \in 100 I_i \cup 100 I_j$, with
 $10 I_i  \cap 10 I_j  \not =
\emptyset$, then \begin{align}\label{2.21+}
| B_i ( y, s ) - B_j ( y, s ) |\lesssim \epsilon r.
\end{align}
The case %$\rho_j \geq \rho_i$ can be treated similar and together 
$j\in \Lambda$ is exactly the same, so
this completes the proof in the case that at least one of $i$ or $j$ belongs to $\Lambda$.

If, on the other hand, $i \notin \Lambda$ and $j \notin \Lambda$, then  $Q(i)=Q(j)=Q(\sbf)$, and hence
$B_i \equiv B_j$, so the estimate holds trivially.  This completes the proof of Lemma \ref{lemma2.21}.
\end{proof}

Before proving that $\psi$ is Lip(1,1/2), we present one more preliminary lemma, which
is the parabolic version of \cite[Lemma 8.21]{DS1}.

\begin{lemma}\label{d=D}
Let
 \begin{align}\label{2.21+++b}
 (x,t)\in C'_{4\kappa R}(x_{ Q(\sbf)},t_{ Q(\sbf)})\mbox{ and $r>0$ with }D(x,t)\leq r<2\kappa R,
\end{align}
and let $Q\in\sbf$ be such that $d_p(x,t,\pi(Q))\leq \lambda r$ and 
$\lambda^{-1}r\leq\diam Q\leq \lambda r$ for some $\lambda\geq 1$. Then
 \begin{align}\label{2.21+++c}\mbox{$\pi^{-1}\left(C'_r(x,t)\right)\cap 4\kappa  Q(\sbf)$ 
 is contained in $\tilde\lambda Q$},
\end{align}
where $\tilde\lambda$ depends on $\lambda$ and $\kappa$. 

Furthermore, for some $\hat\lambda$ depending on $\kappa$, we have
 \begin{align}\label{2.21+++d}\mbox{$\hat\lambda^{-1}d(X,t)\leq D(\pi(X,t))\leq d(X,t)$ 
 for all $(X,t)\in 2\kappa  Q(\sbf)$}.
\end{align}
\end{lemma}

\begin{proof}
The proof follows that of \cite[Lemma 8.21]{DS1} essentially verbatim.
To prove \eqref{2.21+++c}, 
we consider $(Y,s) \in Q$, $(Z,\tau) \in \pi^{-1}\left(C'_r(x,t)\right)\cap 4\kappa  Q(\sbf)$.
First, if
$$d_p\big((Y,s), (Z,\tau)\big) \leq \diam Q,$$
then $(Z,\tau) \in \tilde\lambda Q$, for, say, $\tilde\lambda \geq 5$. 
Second, if
$$d_p\big((Y,s), (Z,\tau)\big) >\diam Q,$$
then $$d(Y,s) \leq \diam Q \leq d_p\big((Y,s), (Z,\tau)\big) .$$
Hence, by Lemma \ref{lem1}
$$| \pi^{\perp}(Y,s) - \pi^{\perp}(Z,\tau)| \leq d_p\big(\pi(Y,s) , \pi(Z,\tau)\big) \lesssim_\lambda  \diam Q.$$
(The last inequality follows from our assumptions that $d_p(x,t,\pi(Q))\lesssim  r\approx \diam Q$,
and that $(Y,s) \in Q$, and $\pi(Z,\tau) \in C'_r(x,t)$.)
This proves \eqref{2.21+++c}.   

To deduce \eqref{2.21+++d}, we note that the second inequality is 
trivial by definition.  To prove the first inequality,  we consider first the case that  
$\pi(X,t) \in \pi(F)$, thus, $\pi(X,t) = \pi(X',t)$ for some $(X',t)\in F$, so that $d(X',t)=0$ by definition.
But then by Lemma \ref{lem1}, $(X,t)=(X',t)$, so $d(X,t) = 0$, and  \eqref{2.21+++d} holds in this case.
Otherwise, if $\pi(X,t) \notin \pi(F)$, we then set $(x,t):= \pi(X,t)$, and $r:= D(x,t)>0$, and note that by 
definition of $D(x,t)$, there is a $Q\in\sbf$ such that
$d_p(x,t,\pi(Q))\lesssim  r\approx \diam Q$ (here, there can be possible implicit dependence on
$\kappa$ in the case that $R\ll r\lesssim \kappa R$).  
Hence, we may apply \eqref{2.21+++c} to deduce that
$(X,t) \in \tilde\lambda Q$, and therefore, 
\[d(X,t) \leq d_p(X,t,Q) +\diam Q \lesssim \diam Q \approx r=D\big(\pi(X,t)\big). \]
\end{proof}

\begin{lemma}\label{graphlip.lem}Let $\psi:\mathbb R^n\to\mathbb R$ be defined as in
\eqref{2.19}.  Then $\psi: \mathbb R^n\to\mathbb R$ is Lip(1,1/2) with 
constant on the order of $\delta$, i.e.
\begin{equation}\label{lemma4.26est}
|\psi(y,s)- \psi(z,\tau)| \lesssim \delta d_p(y,s,z,\tau)\,, \quad
\forall \, (y,s),(z,\tau) \in \mathbb{R}^n \,.
\end{equation}
\end{lemma}
\begin{proof}
Again, we shall follow the corresponding arguments in \cite[Chapter 8]{DS1}.  However, we point out 
that the proof in \cite{DS1} is not fully explicated, and in particular, does not 
address the case that $F=\emptyset$ (which, in contrast to the case of the classical Whitney extension
theorem, can actually happen in the present setting).
Fortunately, the methods of \cite{DS1} require only a modest refinement to address this situation.

Recall that we use the notation $r_i:= \diam I_i$.
There are four cases to consider:
\begin{enumerate}
\item $(y,s),(z,\tau) \in \pi(F)$.
\smallskip
\item $(y,s)\in  \mathbb{R}^n\setminus \pi(F) $, with $(y,s) \in 2I_j$,
and  $(z,\tau) \in \pi(F) $.
\smallskip
\item  $(y,s), (z,\tau) \in  \mathbb{R}^n\setminus \pi(F)$, with $(y,s) \in 2I_j$,
$(z,\tau) \in 2I_k$, and
\[%j,k\in \Lambda\,,\quad 
\text{ (WLOG) } \, r_k \leq r_j\,,\quad (z,\tau)\in 10I_j\,.\]
%\[d_p\big(y,s,z,\tau\big) < 5 \max( \diam I_j, \diam I_k)\,.\]
\item $(y,s), (z,\tau) \in  \mathbb{R}^n\setminus \pi(F)$, with $(y,s) \in 2I_j$,
$(z,\tau) \in 2I_k$, and 
\[%j,k\in \Lambda\,,\quad 
\text{ (WLOG) } \, r_k \leq r_j\,,\quad (z,\tau)\notin 10I_j\,.\]
%\[d_p\big(y,s,z,\tau\big) \geq 5 \max( \diam I_j, \diam I_k)\,.\]
%\item Either $j\notin \Lambda$ or $k\notin \Lambda$.
\end{enumerate}

We make one more preliminary observation:  by \eqref{supportpsi}, we may suppose that
$(y,s)$ and $(z,\tau)$ both lie in the closure of 
$C_*':=C'_{4\kappa R}(x_{ Q(\sbf)},t_{ Q(\sbf)})$.   Indeed, if both points lie 
outside of this closed cube, then \eqref{lemma4.26est} is trivial, and if one (say $(y,s)$) lies outside, 
and the other inside, then we may replace $(y,s)$ by another point $(y',s') \in \partial C'_*$
such that $d_p(z,\tau,y',s') \leq d_p(z,\tau,y,s)$, since for such points $(y',s')$ and $(y,s)$, we 
have $\psi(y,s) = 0 = \psi(y',s')$.

Let us now discuss the various cases (1)-(4).

Case (1) follows immediately from Lemma \ref{lem1}, and the definition of $\psi$ (see \eqref{2.19}).

Consider next case (3).
In this case, $(y,s), (z,\tau)\in 10I_j$.
By Lemma \ref{10-60lemma}, \eqref{2.16+} and Lemma \ref{lemma2.21}, and the fact that 
$\sum_{i}  (\nu_i(y,s) - \nu_i(z,\tau)) =0$,
\begin{multline}\label{2.21++}
| \psi( y, s ) - \psi ( z,\tau) | \,\leq \, \big|\sum_{i} (B_i(y,s) - B_i(z,\tau))\nu_i(y,s)\big|
\\[4pt]
 +\, \big|\sum_{i}  (B_i(z,\tau)-B_j(z,\tau))(\nu_i(y,s) - \nu_i(z,\tau))\big| \\[4pt]
\lesssim \,\delta d_p(y,s,z,\tau)\,+\,r_j^{-1}\epsilon d_p(y,s,z,\tau)r_j 
\, \lesssim \,\delta d_p(y,s,z,\tau),
\end{multline}
 since each $B_i$ has Lip(1,1/2) norm at most $2\delta$, and $\epsilon\ll\delta$.  
 
 It remains to treat cases (2) and (4), which we shall do more or less simultaneously. 
 We remark that if $F$ is non-empty, one can reduce case (4) to case (2), as in the 
 proof of the classical Whitney extension theorem. However, in the present setting, it may be that
$F$ is empty, so we shall treat case (4) directly.  For the sake of specificity, let us do this first, 
as the proof in case (2) will be similar, but a bit simpler.

  In case (4), we decompose
$$|\psi(y,s) - \psi(z,\tau)| \leq a_1 + a_2 + a_3 + a_4 +a_5+a_6+a_7,$$
where for some $(U,t) \in Q(j)$,  $(U',t') \in Q(k)$ to
be specified more precisely below,  setting $(u,t):=\pi(U,t)$, and $(u',t'):=\pi(U',t')$,
we define
\begin{align*}
&a_1 = |\psi(y,s) - B_j(y,s)|\,, \quad a_2 = |B_j(y,s) - B_j(u,t)|\,, \\[4pt]
&a_3 = |B_j(u,t) - \pi^{\perp}(U,t)|\,, \quad a_4 = |\pi^{\perp}(U,t) -\pi^\perp(U',t')|\,, \\[4pt]
&a_5= |\pi^\perp(U',t') -B_k(u',t')|\,, \quad a_6 =|B_k(u',t') - B_k(z,\tau)|\,, \\[4pt]
&a_7 = |B_k(z,\tau) - \psi(z,\tau)|\,.
\end{align*}

We estimate these terms as follows.  For future reference, we observe that since $(z,\tau)\notin 10I_j$,
and $(y,s)\in 2I_j$, therefore, for a purely dimensional implicit constant,
\begin{equation}\label{4.29}
r_k\leq r_j \lesssim d_p(y,s,z,\tau)\,.
\end{equation}

%We divide case (4) into 2 sub-cases, the first being that  $d_p(y,s,z,\tau) \geq 1000 r_j$. In this scenario, 
Consider for now any
$(U,t) \in Q(j)$,
and any $(U',t') \in Q(k)$.
By definition of $\psi$, since $\sum_i \nu_i(y,s) = 1$, we have
\[
a_1 = \big| \sum_i \nu_i(y,s)\, \big(B_i(y,s) -B_j(y,s)\big)\big| \,\lesssim \,\epsilon  r_j \,\lesssim \,
\epsilon d_p(y,s,z,\tau)\,,
\]
by  Lemma \ref{lemma2.21} (ii) and \eqref{4.29}.  Similarly,
\[ a_7\,\lesssim\, \epsilon d_p(y,s,z,\tau)\,.\]

Since each $B_i$ has Lip(1,1/2) norm at most $2\delta$,
by the second inequality in \eqref{eq1+}, 
\begin{equation}\label{a2bound}
a_2 \lesssim \delta  d_p(y,s,u,t) \lesssim \delta r_j \lesssim \delta d_p(y,s,z,\tau)\,,
\end{equation}
again using \eqref{4.29}.  Similarly,
\[ a_6\,\lesssim\, \delta d_p(y,s,z,\tau)\,.\]

Since $P_{Q(j)}$ is the graph of $B_j$, which has small slope, and since
$\beta_\infty(Q(j))\leq\epsilon$,  we have
\begin{equation}\label{a3bound}
a_3\approx \dist\big((U,t), P_{Q(j)}\big) \leq \epsilon \diam Q(j) \leq 120 \epsilon r_j \lesssim
\epsilon d_p(y,s,z,\tau)\,,
\end{equation}
where we have used \eqref{eq1+} and then \eqref{4.29} in the last two steps.
Similarly,
\[ a_5 \lesssim \epsilon d_p(y,s,z,\tau)\,.
\]

It remains to deal with $a_4$.  Since
$(U,t) \in Q(j)$ and % $(U',t') \in Q(k)$, and since
$(y,s)\in 2 I_j$, we have
\begin{equation}\label{4.29a}
d(U,t) \leq \diam Q(j)  \leq 120  r_j\,,\quad d_p(u,t,y,s) \leq 300 r_j\,,
\end{equation}
by \eqref{eq1+}.  Similarly, since $(U',t') \in Q(k)$, and 
$(z,\tau)\in 2 I_k$,
\begin{equation}\label{4.31}
d(U',t') \leq \diam Q(k)  \leq 120  r_k\leq 120 r_j\,,\,\,\, d_p(u',t',z,\tau) \leq 300 r_k \leq 300r_j\,.
\end{equation}

Consequently, by \eqref{4.29a} and  \eqref{4.29}, %(again) \eqref{eq1+}, we may apply Lemma \ref{d=D} to obtain
\begin{equation}\label{4.30} % C_\kappa^{-1} 
d(U,t)  \leq 120 r_j \lesssim d_p(y,s,z,\tau) \,.
 \end{equation}
% where in the last step we have used \eqref{4.29}. 

To treat $a_4$, we now consider two sub-cases:
\begin{itemize}
\item (Sub-case A) $d_p(y,s,z,\tau) \geq 1000 r_j$.
\item (Sub-case B) $d_p(y,s,z,\tau) < 1000 r_j$.
\end{itemize}
We first consider the former.  In this case,
since $d_p(y,s,z,\tau) \geq 1000 r_j$, 
the right hand inequalities in \eqref{4.29a} and \eqref{4.31} imply that
\[ d_p(y,s,z,\tau) \approx d_p(u,t,u',t') \leq d_p(U,t,U',t') \,,\]
Combining the latter estimate with \eqref{4.30}, we obtain
\[
d(U,t)\, \lesssim  \,d_p(U,t,U',t')\,.
\]
Thus, we may apply Lemma \ref{lem1}, with $N$ equal to an absolute constant, and with 
$\epsilon \ll \delta$, % depending on $\kappa$, 
to obtain
\begin{equation}\label{4.32}
a_4 \leq 2 \delta d_p(u,t,u',t') \approx \delta d_p(y,s,z,\tau)\,.
\end{equation}

On the other hand, consider now sub-case B, i.e., $d_p(y,s,z,\tau) < 1000 r_j$.  Thus 
\begin{equation}\label{4.33}
d_p(y,s,z,\tau)  \approx r_j\,,
\end{equation}
by \eqref{4.29}.
In the current scenario, we consider two further sub-cases. 

\smallskip
\noindent (Sub-case B1):
$Q(j)$ meets $Q(k)$.  
In this case we simply let $(U,t) = (U',t')$ be any point in $Q(j)\cap Q(k)$, so that, trivially, $a_4=0$.

\smallskip
\noindent (Sub-case B2):  $Q(j)$ and $Q(k)$ are disjoint.  In this case, we
let $(U',t')$ be any point in $Q(k)$, and we select $(U,t)$
as follows.  Let $C(j):= C_r(X_{Q(j)},t_{Q(j)})$, with $r \approx \diam Q(j)$, be the
cube defined in \eqref{cube-ball}, with respect to $Q=Q(j)$.  Since $\Sigma \cap C(j)\subset Q(j)$ by construction (see
\eqref{cube-ball}), it follows that
$(U',t') \notin C(j)$.
We then set $(U,t):= (X_{Q(j)},t_{Q(j)})$ (the center of $C(j)$), and observe that in this case
\begin{equation}\label{4.34}
d_p(U,t,U',t') \gtrsim \diam Q(j)\,,
\end{equation}
where the implicit constants here depend on those in  \eqref{cube-ball}.  

Note that \eqref{4.29a} and 
\eqref{4.31} continue to hold, since we still have
$(U,t) \in Q(j)$, 
$(y,s)\in 2 I_j$,
 $(U',t') \in Q(k)$, and 
$(z,\tau)\in 2 I_k$.  
In particular,
\begin{equation}\label{4.35}
d(U,t) \leq \diam Q(j)\,,\quad  d_p(u,t,y,s) + d_p(u',t',z,\tau)  \lesssim\, r_j\,.
\end{equation}
By  \eqref{4.33} and the second inequality in \eqref{4.35}, we see that
\begin{equation}\label{4.36}
d_p(u,t,u',t') \lesssim r_j \approx d_p(y,s,z,\tau)\,.
\end{equation}
Combining  \eqref{4.34} and the first inequality in \eqref{4.35}, %and using \eqref{eq1+}, 
we find that
\[
 d(U,t) \leq \diam Q(j) \lesssim d_p(U,t,U',t')\,.
\]
Consequently, we may as above apply Lemma \ref{lem1}
to obtain
\begin{equation*}
a_4 \leq 2 \delta d_p(u,t,u',t') \lesssim \delta d_p(y,s,z,\tau)\,,
\end{equation*}
where in the last step we have used \eqref{4.36}.

Summing our estimates for the terms 
$a_1$ through $a_7$, and using that $\epsilon \ll \delta$, we obtain the desired bound
\eqref{lemma4.26est} in case (4), in general.

In case (2), we proceed in a similar fashion to case (4), 
but now matters are a bit simpler: we decompose
$$|\psi(y,s) - \psi(z,\tau)| \leq a_1 + a_2 + a_3 + a_4,$$
with $a_1,a_2$ and $a_3$ exactly as above (and enjoying the very same bounds),
and with
\[a_4 = |\pi^{\perp}(U,t) - \psi(z,t)|\,,\]
where $(U,t)$ is an arbitrary point in $Q(j)$.  
Thus, again, \eqref{4.29a} holds, and combining the latter with \eqref{4.29}, we see as before that
\begin{equation}\label{4.41}
d_p(u,t,y,s) \lesssim d_p(y,s,z,\tau)\,.
\end{equation} 
Recall that in case (2), we assume $(z,\tau)\in \pi(F)$.  Consequently, by (the first part of)
Lemma \ref{lem1}, and the fact that $\psi =\hat{\psi}$ on $\pi(F)$, we have
$d(z,\tau,\psi(z,\tau)) = 0$, and $\pi^\perp(z,\tau,\psi(z,\tau)) =\psi(z,\tau)$.  Hence, by the
second part of Lemma \ref{lem1}, and \eqref{4.41},
we find that 
\[
a_4 \leq 2\delta d_p(u,t,z,\tau) \lesssim \delta d_p(y,s,z,\tau)
\]
We now sum the bounds for terms $a_1$ through $a_4$ to complete the proof of the lemma.
\end{proof}

\begin{lemma}\label{closegraph.lem} Let $(X,t)\in 2\kappa Q(\sbf)$.  Then
\begin{align}\label{2.21+++a} d_p(X,t,(x,\psi(x,t),t))\lesssim \epsilon d(X,t),\quad (x,t):=\pi(X,t).
\end{align}
In particular,  \eqref{eq2.2a-RE} holds provided $\epsilon \ll \delta/\kappa$.
\end{lemma}
\begin{proof}  
Note that if $Q \in \sbf$, then  $d(X,t) \lesssim \kappa \diam(Q)$ for all
$(X,t) \in \kappa Q$.  Using this observation, we see that \eqref{2.21+++a} implies \eqref{eq2.2a-RE} provided $\epsilon \ll \delta/\kappa$.  

We turn now to \eqref{2.21+++a}, following the proof of \cite[Proposition 8.2]{DS1}.
If $(X,t)\in F$, there is nothing to 
prove, so we consider 
 $(X,t)\in  2\kappa Q(\sbf)\setminus F$. 
In particular,
 $d(X,t)>0$. Set $(x,t):=\pi(X,t)$. 
Estimate \eqref{2.21+++d} implies that $D(x,t)>0$ and hence $(x,t)\in I_i$ for some $i$. 
Using \eqref{2.21+++c} with $Q:=Q(i)$ and $r:=D(x,t) \approx \diam I_i$, we see 
that $(X,t)\in\tilde \lambda Q(i)$, where by \eqref{eq1+}, the constant $ \lambda$ in Lemma \ref{d=D}
depends on $\kappa$, and therefore $\tilde \lambda$ depends on $\kappa$.
Since we assume that $K$ is large, depending on $\kappa$, we may in particular take $K$ much larger than 
$\tilde\lambda$, so that
 $$|\pi^\perp(X,t)-B_i(x,t)|\leq 2\epsilon\diam(Q(i)) \lesssim \epsilon \diam I_i \approx \epsilon D(x,t) \lesssim \epsilon d(X,t)\,,$$
 %Combined with the observations above, \eqref{2.21+++a} follows.
 where we have used \eqref{3++} with $Q=Q(i)$, along with the definition of $B_i$. 
Estimate \eqref{2.21+++a} now follows, by Lemma \ref{lemma2.21} (ii) and the definition of $\psi$.
\end{proof}

\begin{lemma}
\begin{align}\label{2.21++bo}
|\nabla_x^2\psi(y,s)|+|\partial_t\psi(y,s)|\lesssim \epsilon %\delta
r_j^{-1}\mbox{ if } (y,s)\in 2I_j.
\end{align}
\end{lemma}
\begin{proof} We first prove the bound for $|\nabla_x^2\psi(y,s)|$ by following \cite{DS1}.   
By the definition of $\Lambda$ (see \eqref{la-def}), and the adapted partition of unity $\{\nu_i\}$, 
we have
\begin{align}\label{diff}
\sum_{i \in \Lambda} \nu_i(y,s) =1, \text{ and } \,
\sum_{i \in \Lambda} \nabla_x\nu_i(y,s) =
\nabla_x\Big(\sum_{i \in \Lambda} \nu_i(y,s) \Big) =  0\,,
\end{align}
whenever $(y,s)\in 2I_j$. Let $\alpha$, $\beta$ be two spatial indices and let
$\partial_\alpha$, $\partial_\beta$, and $\partial_{\alpha\beta}$ denote the corresponding 
partial differential operators of order one and two. Using \eqref{diff}  
we see that if $(y,s)\in 2I_j$,  then
\begin{align*}
&\partial_{\alpha\beta} \psi(y,s) = \partial_{\alpha\beta} \left(\sum_{i}  \nu_i(y,s) B_i(y,s) \right) \\
&= \sum_{i }(\partial_{\alpha\beta} \nu_i(y,s))B_i(y,s) 
+ \sum_{i} (\partial_{\alpha} \nu_i(y,s))(\partial_\beta
 B_i(y,s)) + \sum_{i} (\partial_{\beta} \nu_i(y,s))(\partial_\alpha B_i(y,s))
\\
&= \sum_{i}(\partial_{\alpha\beta}\nu_i(y,s))(B_i(y,s)-B_j(y,s)) + \sum_{i} (\partial_{\alpha} \nu_i(y,s))
(\partial_{\beta} B_i(y,s) -\partial_\beta B_j(y,s)) \\
&+ \sum_{i} (\partial_{\beta} \nu_i(y,s)) (\partial_{\alpha} B_i(y,s) -\partial_\alpha B_j(y,s))\, =:\, I+II+III.
\end{align*}
The desired bound for term $I$ follows immediately from Lemma \ref{lemma2.21} (ii) and \eqref{2.16+},
and the fact that the cubes $3I_i$ (which contain the support of $\nu_i$) have bounded overlaps, by virtue of Lemma
\ref{10-60lemma} (specifically \eqref{eq1}).
A similar argument yields the desired bound for terms $II$ and $III$, once we observe that,
if $10I_i \cap 10I_j \neq \emptyset$, then $|\nabla_x B_i -\nabla_x B_j| \lesssim \epsilon$. %$\delta$.  
To see the latter,  we translate and (spatially) rotate coordinates (only for purposes of the present argument) so that 
$P_{Q(j)}$ is the hyperplane $\{(x,x_n,t): \,x_n = 0\}$, hence in these new coordinates $B_j\equiv 0$.  By Lemma \ref{aux2}, the angle between $P_{Q(i)}$ and $P_{Q(j)}$ is at 
most $C\epsilon$, i.e., in the new coordinates $B_i$ is affine with 
$|\nabla_x B_i -\nabla_x B_j|= |\nabla_x B_i | \leq C\epsilon$.

The bound for $|\partial_t\psi(y,s)|$ can be produced similarly. In this 
case we first note that $\partial_t B_i = 0$ (since $P_{Q(i)} \in \mathcal{P}$) and that $$0 = \partial_t \left(\sum_i \nu_i(y,s)\right) = \sum_i \partial_t \nu_i(y,s)$$ for $(y,s)\in 2I_j$. Using this, Lemma \ref{lemma2.21} (ii), and \eqref{2.16+},  we deduce, for $(y,s)\in 2I_j$,
\begin{align*}
|\partial_t \psi(y,s)| &= \bigl|\partial_t \sum_i \nu_i(y,s) B_i(y,s)\bigr| = \bigl| \sum_i \partial_t\nu_i(y,s) B_i(y,s)\bigr|
\\ &=  \bigl| \sum_i \partial_t\nu_i(y,s) (B_i(y,s)-B_j(y,s)) \bigr|\\
& \lesssim \sum_i |\partial_t \nu_i | \epsilon r_j \lesssim \epsilon r_j^{-1}.
\end{align*}
Here we have also used $\diam(I_i) \sim \diam(I_j)$ whenever $10I_i\cap 10I_j\neq\emptyset$.
\end{proof}

\section{Pushing the geometric square function to the graph}\label{pushsqfn}

In the previous section we proved that there exists, for arbitrary but fixed $\sbf\in \mathcal{F}$, a coordinate system and a Lip(1,1/2)  function
$\psi_{\sbf}:=\psi=\psi ( x, t ): \mathbb R^{n-1}\times\mathbb R\to \mathbb R$ with
 parameter $b_1\lesssim \delta$,  such that if we define
 $$\Sigma_{\sbf}:=\Sigma_\psi:=\{(x,\psi(x,t),t):\ (x,t)\in \mathbb R^{n-1}\times\mathbb R\},$$ then
\begin{equation}\label{eq2.2a}
\sup_{(X,t)\in 2\kappa Q} d_p(X,t,\Sigma_{\sbf} )\,
\lesssim\epsilon\ell(Q)\,,\qquad \forall\,Q\in \sbf.
\end{equation}
(The latter bound \eqref{eq2.2a} is proved in Lemma \ref{closegraph.lem}.) Recall that
we have already fixed some of the relations amongst the parameters $K$,  $\epsilon$, $\delta$ and $\kappa$:
e.g., we have specified that $\epsilon \ll \delta$, depending on $\kappa$ (see Lemma \ref{closegraph.lem}), and that $K$ is very large, also depending on $\kappa$  (see Lemma \ref{lem1} and the 
proof of Lemma \ref{closegraph.lem}).  A little more precisely, we have $0<C(\kappa) \,\epsilon\ll \delta\ll 1$, for some large enough constant $C(\kappa)$, and 
$K=c(\kappa,n,M)\gg 1$.
In addition to these existing relations, in the final part of the argument, see
\eqref{geo9-}, $\epsilon$ will be chosen to depend further
upon $n, K$, the ADR constant $M$, and the parabolic UR constant 
$\Gamma$ (see Definition \ref{def1.UR}), as well as on $\kappa$ and $ \delta$.

To complete the proof of Theorem \ref{main} we have to prove two things. First, we have to prove that the constructed
function $ \psi_{\sbf}:=\psi = \psi ( x, t ) : \mathbb R^{n-1}\times\mathbb R\to \mathbb R$ in fact is a {\it regular} parabolic Lip(1,1/2)  function with parameters $b_1\lesssim\delta$ and $b_2=b_2(n,M,\Gamma)$. Second, we have to prove that the maximal cubes $Q(\sbf)$ for the  stopping time regimes satisfy the Carleson
packing condition
\begin{align}\label{pack}\sum_{\sbf: Q(\sbf)\subset Q}\sigma\big(Q(\sbf)\big)\,\leq\, c(n,M,\Gamma,\delta,\kappa)\, \sigma(Q)\,,
\quad \forall Q\in \dd(\Sigma).
\end{align}
The key to both of these arguments is to make use of the geometric square function which is part of the definition of parabolic uniform rectifiability. In this section we prove how this information can be pushed to the constructed graph.

Recall that $\Sigma$ is equipped with the geometric square function $\gamma$ and the Carleson measure $\nu$. We let
$\ti \Sigma$ be the graph of $\psi$ and we denote the corresponding geometric square function and  Carleson measure by $\tilde\gamma$ and  $\tilde \nu$, respectively.  In addition, for $ r > 0, ( z, \tau  ) \in
\mathbb R^{n} $ we introduce
\begin{align}\hat\gamma ( z, \tau, r ) \, := \, \biggl (  \inf_{ L }  \, \bariint_{ C'_r (  z, \tau  ) } \, \biggl (\frac {| \psi ( y, s ) - L ( y ) |}{r}\biggr )^2 \, \d y\d s \biggr )^{1/2},
\end{align} where the infimum is with respect to all linear functions of $ y $ (only). Using the regularity of the $\psi$  constructed,  we see that\footnote{See Remark \ref{r-measures}.}
\begin{align}\label{measure}
\ d \si_{\ti\Sigma} \approx \ d(\mu|_{\ti\Sigma})
( Y, s ) = \sqrt{ 1 + | \nabla_y \psi  ( y, s ) |^2  } \, \d y\d s\sim \d H^{n-1 } ( y)\d s,
\end{align}
for all $( Y, s ) \in \ti\Sigma$, and we recall that $\mu$ is the slice-wise measure (see \eqref{slicewise.def}).  The reader can easily verify that the slice-wise measure on the graph is parabolic ADR, and hence after an application
of Proposition \ref{muhpsim.prop} that the implicit constants in the first comparability in display \ref{measure} above depend only on constants which depend on dimension and $\delta$.
Combining \eqref{measure} with \eqref{lemma4.26est} and \eqref{2.21++bo} it follows that
\begin{align}\label{2.30}\hat\gamma  (  z,  \tau, r
)\sim \tilde \ga (  Z , \tau, r )\mbox{ for all }(Z,\tau)=(z,\psi(z,\tau),\tau)\in\ti\Sigma,\ r>0.
\end{align}

The essence of the section is to prove that we can control integrals of $\hat\gamma  (  z,  \tau, r
)$, and hence integrals of $\tilde \ga (  Z , \tau, r )$,  with  corresponding quantities involving $\gamma$, with an error controlled by $\epsilon^2$. In particular, the most important lemma in this section is the following.

\begin{lemma}\label{count++} Fix $\sbf \in \F$.
Consider $(\hat z,\hat\tau)\in\mathbb R^n$, $\rho>0$, and let
\begin{align}\label{pl11n}
 \hat\nu(\hat z,\hat\tau,\rho):=
 \int_0^{\rho}\iint_{C'_{\rho}(\hat z,\hat\tau)}(\hat\gamma ( z, \tau, r ))^2\frac {\d z\d \tau\d r}r.
\end{align}
Assume further that
\begin{align} \label{local}
C'_{2\rho}(\hat z,\hat\tau)\subseteq C'_{2\kappa R}(x_{ Q(\sbf)},t_{ Q(\sbf)}),\quad
\rho\leq\kappa R/10.
\end{align}
Then \begin{align}\label{pl11n+}
 \hat\nu(\hat z,\hat\tau,\rho)\lesssim  \epsilon^2\rho^{d}+\iiint_{E(\hat z,\hat\tau,\rho)} (\ga  ( Z, \tau, Kr))^2 \, \frac{\d \si ( Z, \tau)\d r} {r}
\end{align}
where $E(\hat z,\hat\tau,\rho)$ is the set of all 
$(Z,\tau,r)\in \big[\Sigma\cap\pi^{-1}(C'_{K\rho}(\hat z,\hat\tau))\big]\times(0,\infty)$ such that
$ K^{-1}d(Z,\tau)\leq r\leq K\rho$.
\end{lemma}

From the previous lemma we shall be able to deduce the following lemma which will yield the desired regularity of the graph (see Lemma \ref{pushlemma2}).

\begin{lemma}\label{pushlemma1} Let $(\hat z,\hat\tau)\in\mathbb R^n$, $\rho>0$, and let $\hat\nu(\hat z,\hat\tau,\rho)$ be defined as in Lemma \ref{count++}. Then
 \begin{align}\label{pl11}
  \hat\nu(\hat z,\hat\tau,\rho)&\lesssim (\epsilon^2+||\nu||)\rho^{d}.
\end{align}
\end{lemma}

The rest of this section is devoted to the proofs of Lemma \ref{count++} and Lemma \ref{pushlemma1}.  To achieve this we will prove a number of auxiliary lemmas. Recall that $R = \diam(Q(\sbf))$.
\subsection{Auxiliary lemmas}
\begin{lemma}\label{count} Let $(z,\tau)\in  C'_{2\kappa R}(x_{ Q(\sbf)},t_{ Q(\sbf)})$ and $D(z,\tau)/60\leq r\leq \kappa R/10$. Let $ I( z,\tau, r )$ be the set of all $i\in \Lambda$ such that $C_r(z,\tau)\cap I_i\neq \emptyset$. Let, for $i\in\Lambda$, $J(i)$ be the subset of $ I( z,\tau, r )$ which consists of those $j$ which satisfy
$\diam (Q(j))\leq \diam (Q(i))$ and $2Q(j)\cap 2Q(i)\neq\emptyset$. Let
$$N_i(Y,s):=\sum_{j\in J(i)}\chi_{2Q(j)}(Y,s).$$
Then
\begin{equation}\label{Niavgbnd.eq}
\iint N_i(Y,s)\, \d\sigma (Y,s)\lesssim \sigma(Q(i)).
\end{equation}
Furthermore, there exists, for any $\beta > 1$, a constant $c=c(n,M,\beta)> 1$, such that
\begin{equation}\label{Niptwsbnd.eq}
\sum_{i\in I( z,\tau, r )}(N_i(Y,s))^{-\beta}\chi_{2Q(i)}(Y,s)\leq c,
\end{equation}
for all $(Y,s)\in\cup_{i\in I( z,\tau, r )}(2Q(i))$. Here we interpret\footnote{This interpretation will be valid when we use the functions $N_i$, since by \eqref{Niavgbnd.eq} the set of such points has measure zero.} $(N_i(Y,s))^{-\beta} = 0$ if $N_i(Y,s) = \infty$.

\end{lemma}
\begin{proof}
To prove \eqref{Niavgbnd.eq} of the lemma it suffices to note that
\begin{align}
\iint N_i(Y,s)\, \d\sigma (Y,s)\lesssim \sum_{j\in J(i)}\sigma(Q(j))\lesssim \sum_{j\in J(i)}\sigma(I_j)\lesssim
\sigma(Q(i))
\end{align}
since $\{I_j\}$ is a disjoint collection,  with $d_p\big(I_j,\pi(Q(i))\big)\lesssim r_i$ for each $j\in J(i)$.
% lies at a distance $\lesssim r_i$ from $\pi(Q(i))$.

To prove \eqref{Niptwsbnd.eq}, we consider $(Y,s) \in \bigcup_{i \in I(z,\tau, r)} 2Q(i)$ fixed. Note that in the sum in \eqref{Niptwsbnd.eq}  we only consider indices $i$ in $I(z,\tau, r)$ such that $(Y,s) \in 2Q(i)$ and in the following we, for simplicity, denote this collection $\widetilde{I}$. An important observation is that there exists $a'=a'(n,M) \in \mathbb{N}$, such that if $i, i' \in \widetilde{I}$, $Q(i) \in \mathbb{D}_k$ and $Q(i') \in \mathbb{D}_{k'}$,  for some
$k$ and $k'$ such that $k \le k' - a'$, then $i' \in J(i)$. This is a consequence of  the facts that $\diam(Q) \approx 2^{-k}$ for all $Q \in \mathbb{D}_k$, and that $(Y,s) \in 2Q(i) \cap 2Q(i')$.

We divide the proof of \eqref{Niptwsbnd.eq} into two cases. First, assume that $\# \widetilde{I} = \infty$. 
In this case, since each $Q(i) \subseteq Q({\bf S})$, 
there must exist a  sequence of cubes $\{Q(i_m)\}_m$ with $i_m \in \widetilde{I}$, such that $\diam(Q(i_m)) \to 0$ as $m \to \infty$. Using the observation
made above, this implies that $N_i(Y,s) = \infty$ for all $i \in \widetilde{I}$. In particular, $$\sum_{i \in I(z,\tau, r)} (N_i(Y,s))^{-\beta} \mathbbm{1}_{2Q(i)}(Y,s)  = 0.$$

Second, assume that $\# \widetilde{I} < \infty$. In this case, the sets $\{Q(i)\}$, where $i \in \widetilde{I}$,  belong to a finite number of generations of $\mathbb{D}$. With this in mind, we let $\{k_m\}_{m = 1}^{m_0}$, $k_1 > k_2  > \dots > k_{m_0}$, be all the 
integers for which there exists $Q(i) \in \mathbb{D}_{k_m}$ with $i \in \widetilde{I}$. For $m \in \{1,2,\dots,m_0\}$ we introduce $$\mathcal{G}_m := \{Q(i) : i \in \widetilde{I}, Q(i) \in \mathbb{D}_{k_m}\},$$ and
we note that $\mathcal{G}_m$ is non-empty. Again by the observation made above, we can conclude that if $Q(i) \in \mathcal{G}_m$ and $Q(i') \in \mathcal{G}_{m'}$, with $m \ge m' + a'$, then $i'\in J(i)$. Thus,
\[N_i(Y,s) \ge \max\{1, m-a'\}, \quad \forall Q(i) \in \mathcal{G}_m.\]
For each dyadic generation $k$, every
$Q \in \dd_k$
such that $(Y,s)\in 2Q$, must lie within a distance 
$\lesssim \diam Q \approx 2^{-k}$ of $(Y,s)$, hence, by the dyadic cube construction, 
there are at most a uniformly bounded number of such cubes;
in particular, $\#\mathcal{G}_m \le c(n,M)$. Consequently,
\begin{align*}
\sum_{i \in I(z,\tau, r)} (N_i(Y,s))^{-\beta} \mathbbm{1}_{2Q(i)}(Y,s)  &= \sum_{m = 1}^{m_0} \sum_{Q(i) \in \mathcal{G}_m} (N_i(Y,s))^{-\beta} \mathbbm{1}_{2Q(i)}(Y,s)
\\& \le \sum_{m = 1}^{m_0} \#\mathcal{G}_m (\max\{1, m-a'\})^{-\beta}\\
& \le c'(n,M,\beta),
\end{align*}
where we have used that $a' = a'(n, M)$.\end{proof}

The following lemma will be an important ingredient in the proof of Lemma \ref{count++}.

\begin{lemma}\label{count+} Let $(z,\tau)\in  C'_{2\kappa R}(x_{ Q(\sbf)},t_{ Q(\sbf)})$ and $D(z,\tau)/60\leq r\leq \kappa R/10$. Let $ I( z,\tau, r )$ be the set of all $i\in \Lambda$ such that $C'_r(z,\tau)\cap I_i\neq \emptyset$. Let  $(X,t)\in\Sigma$, $(X,t)=(X(z,\tau,r),t(z,\tau,r))$, be such that
$(X,t)\in Q(\sbf)$ and $d_p(z,\tau,\pi(X,t))\lesssim r$.  Then
\begin{multline}\label{5.17}
(\hat\gamma (z,\tau, r ))^2
\\[4pt]
\lesssim  r^{-d}\iint_{\Sigma\cap C_r(X,t)}(\gamma(Y,s,Kr/10))^2\, \d\sigma(Y,s) \,
+\,\sum_{i\in I( z, \tau, r )} \epsilon^2(r_i/r)^{d+2},
\end{multline}
\end{lemma}
\begin{proof} Let $(z,\tau)\in  C'_{2\kappa R}(x_{ Q(\sbf)},t_{ Q(\sbf)})$, $ I( z,\tau, r )$ and $$(X,t)=(X(z,\tau,r),t(z,\tau,r))\in\Sigma,$$ be as stated in the lemma. We emphasize that by construction the point $(X,t)$ depends on $(z,\tau)$ and $r$. Let $P:=P( X,t, r ) \in \mathcal{P}$ be the time independent plane for which the infimum in the definition of $\gamma(X,t,Kr/10)$ is realized.  Recall the notation $$\ti\Sigma=\{(y,\psi(y,s),s):\ (y,s)\in\mathbb R^n\}.$$
To denote points on the approximating graph, we will in the following also use the notation
$\ti\Sigma(y,s):=(y,\psi(y,s),s)$.

To start the proof we first note that
\begin{align}\label{aeq1}
(\hat\gamma ( z,\tau, r ))^2&\lesssim r^{-d}\iint_{C'_r( z,\tau)}\biggl (\frac{d_p(\tilde\Sigma(y,s),P)}r\biggr )^2\, \d y\d s,
\end{align}
and we introduce
\begin{align}\label{aeq2}
T&:= r^{-d}\iint_{C'_r( z,\tau)\cap \pi(F) }\biggl (\frac{d_p(\tilde\Sigma(y,s),P)}r\biggr )^2\, \d y\d s,\notag\\
T_i&:=r^{-d}\iint_{C'_r( z,\tau)\cap I_i }\biggl (\frac{d_p(\tilde\Sigma(y,s),P)}r\biggr )^2\, \d y\d s.
\end{align}
Using this notation we can continue the estimate in \eqref{aeq1} and write
\begin{align}\label{aeq3}
(\hat\gamma ( z,\tau, r ))^2&\lesssim  T+\sum_{i\in I( z,\tau, r )}T_i.
\end{align}
It is straightforward to lift the integral in the  definition of $T$ to the graph, and to produce the estimate
\begin{align}\label{aeq4}
T&\lesssim r^{-d}\iint_{\pi^{-1}(C'_r( z,\tau))\cap F }\biggl (\frac{d_p((Y,s),P)}r\biggr )^2\, \d \sigma(Y,s)\lesssim (\gamma(X,t,Kr/10))^2.
\end{align}
Note that in this estimate we have also used \eqref{measure}.

Fix  $ i\in I( z,\tau, r )$. Then
$$T_i\lesssim \hat T_i+\ti T_i,$$
where
\begin{align*}
\hat T_i&:=r^{-d}\iint_{C'_r( z,\tau)\cap I_i }\biggl (\frac{d_p(\tilde\Sigma(y,s),P_{Q(i)})}r\biggr )^2\, \d y\d s,
\end{align*}
and
\begin{align*}
\ti T_i&:=r^{-d}r_i^{d}\sup\biggl\{\biggl (\frac {d_p(Y,s,P)}r\biggr )^2:\ (Y,s)\in P_{Q(i)},\ d_p(Y,s,Q(i))\leq cr_i\biggr\}.
\end{align*}
Recall \eqref{2.19-} and \eqref{2.19}. Using that $ P_{Q(i)}$  is the graph of $B_i$,
\begin{align}\label{aeq5}
\hat T_i&=r^{-d}\iint_{C'_r( z,\tau)\cap I_i }\biggl (\frac{|\psi(y,s)-B_i(y,s)|}r\biggr )^2\, \d y\d s\lesssim \epsilon^2r^{-d}r_i^{d}(r_i/r)^2,
\end{align}
where we have used that $|\psi(y,s)-B_i(y,s)|\lesssim \epsilon r_i$ on $I_i$, see Lemma \ref{lemma2.21}.

To estimate $\ti T_i$ we use Lemma \ref{aux1} applied to $Q=Q(i)$. Indeed, let $(Z_j,\tau_j)\in Q(i)$ for $j=0,1,...,n$, be as in the statement of the lemma and let
$L_{Q(i)}=L_{n-1}$  be the spatial $(n-1)$-dimensional plane which passes through $Z_0,Z_1,...,Z_{n-1}$. Then
\begin{align}\label{planeapprox}
\ti T_i\lesssim r^{-d}r_i^{d}\sup_{j\in\{0,...,n-1\}}\biggl \{\biggl (\frac {d_p(Z_j,\tau_j,P)}r\biggr )^2
+\biggl (\frac {d_p(Z_j,\tau_j,P_{Q(i)})}r\biggr )^2\biggr \}.
\end{align}
To see this, we note that  by Lemma \ref{aux1} we have that the points $\{(Z_j,\tau_j)\}$ stay at an ample distance from each other, and the point $(Y_j,s_j)$, realizing the infimum
$$\inf_{(Y,s) \in P_{Q(i)}} d_p((Y,s),(Z_j,\tau_j)),$$
is within $\epsilon \diam Q(i)$ of $(Z_j,\tau_j)$.
Hence, for sufficiently small $\epsilon$, we can ensure that
$$\|(Y_j,s_j)-(Y_{j'},s_{j'})\| \gtrsim \diam Q(i)\mbox{ whenever }{j'}\neq j.$$
In particular, the points $\{(Y_j,s_j)\}$ generate the plane $P_{Q(i)}$, and  for each $(Y,s) \in P_{Q(i)}$ such that $d_p(Y,s,Q(i)) \leq cr_i$, we have
$$ d_p(Y,s,L_{n-1}) \lesssim \max_{j \in \{0,...,n-1\}} d_p(Y_j,s_j,Z_j,\tau_j) = \max_{j \in \{0,...,n-1\}} d_p(Z_j,\tau_j,P_{Q(i)}).$$
An even simpler argument shows that $d_p(L_{n-1},P)$ is controlled by $$\max_{j \in \{0,...,n-1\}} d_p(Z_j,\tau_j, P).$$
In this case we do not need to find a generating
set for $P$. Put together we can conclude that \eqref{planeapprox} holds.

 Using that $Q(i)\in \sbf$ we see that
\begin{align*}
\sup_{j\in\{0,...,n\}}\biggl \{\biggl (\frac {d_p(Z_j,\tau_j,P_{Q(i)})}r\biggr )^2\biggr \}\lesssim\epsilon^2(r_i/r)^2,
\end{align*}
and hence, based on \eqref{planeapprox},
\begin{align}\label{aeq6}
\ti T_i\lesssim r^{-d}r_i^{d}\sup_{j\in\{0,...,n\}}\biggl \{\biggl (\frac {d_p(Z_j,\tau_j,P)}r\biggr )^2\biggr \}
+ r^{-d}r_i^{d}\epsilon^2(r_i/r)^2.
\end{align}
We now note that the statement in \eqref{aeq6} remains true with $(Z_j,\tau_j)$ replaced by any $(\tilde Z_j,\tilde \tau_j)\in \Sigma$ satisfying
$d_p(Z_j,\tau_j,\tilde Z_j,\tilde \tau_j)\leq \eta r_i$ for $\eta=\eta(n,M)$ sufficiently small. This can be seen from the proof of Lemma \ref{aux1} because we are free, at each iteration of the argument, to replace the points 
$(Z_j,\tau_j)$  with sufficiently close points $(\tilde Z_j,\tilde \tau_j)$, as this will preserve the desired properties of the (new) approximating planes and collection of points. Combining this observation, 
\eqref{aeq6}, \eqref{aeq3}, \eqref{aeq4}, \eqref{aeq5} and averaging
over such $(\tilde{Z}_j, \tilde{\tau}_j)$, we conclude that
\begin{multline}\label{aeq7}
(\hat\gamma ( z,\tau, r ))^2 \\[4pt]
\lesssim \, (\gamma(X,t,Kr/10))^2\,+\sum_{i\in I( z,\tau, r )} \epsilon^2r^{-d}r_i^{d}(r_i/r)^2 \,
+\sum_{i\in I( z,\tau, r )} r^{-d}r_i^{d}\beta_{Q(i)}\,,
\end{multline}
where
\begin{align*}\label{aeq7}
\beta_{Q(i)}:=\biggl(\bariint_{2Q(i)} \biggl (\frac {d_p(Y,s,P)}r\biggr )^{2/3}\d\sigma (Y,s)\biggr )^{3}\,.
\end{align*}
We now let $J(i)$ be the subset of $ I( z,\tau, r )$ which consists of those $j$ which are such that
$\diam (Q(j))\leq \diam (Q(i))$ and $2Q(j)\cap 2Q(i)\neq\emptyset$. We let
\begin{align*}
N_i(Y,s):=\sum_{j\in J(i)}\chi_{2Q(j)}(Y,s).
\end{align*}
Then using Lemma \ref{count} we have, for $\beta>1$ given, that there exists a constant $c=c(n,M,\beta)$, $1\leq c<\infty$, such that
\begin{equation}\label{aeq8}
\sum_i(N_i(Y,s))^{-\beta}\chi_{2Q(i)}(Y,s)\le c,
\end{equation}
for all $(Y,s)$ and
\begin{equation}\label{aeq9}
\bariint_{2Q(i)}N_i(Y,s)\, \d\sigma (Y,s)\lesssim 1.
\end{equation}
To estimate $\beta_{Q(i)}$, %we let  $\gamma=1-\alpha$, 
we apply H{\"o}lder's inequality and use \eqref{aeq9} to deduce that
\begin{multline}
\beta_{Q(i)} =
\biggl(\bariint_{2Q(i)} \biggl (\frac {d_p((Y,s),P)}r\biggr )^{2/3}
(N_i(Y,s))^{-2/3}(N_i(Y,s))^{2/3}\d\sigma (Y,s)\biggr )^{3} \\[4pt]
\lesssim \,\biggl(\bariint_{2Q(i)} \biggl (\frac {d_p((Y,s),P)}r\biggr )^{2}(N_i(Y,s))^{-2}\d\sigma (Y,s)\biggr ).
\end{multline}
In particular we can rewrite \eqref{aeq7} as
\begin{multline}
(\hat\gamma ( z,\tau, r ))^2 \,
\lesssim \,(\gamma(X,t,Kr/10))^2\,+\sum_{i\in I( z,\tau, r )} \epsilon^2r^{-d}r_i^{d}(r_i/r)^2\\[4pt]
+ \,r^{-d}\sum_{i\in I( z,\tau, r )}
\iint_{2Q(i)} \biggl (\frac {d_p(Y,s,P)}r\biggr )^{2}(N_i(Y,s))^{-2}\d\sigma (Y,s).
\end{multline}
Then using \eqref{aeq8}  with $\beta = 2$,
we conclude that
\begin{multline}
(\hat\gamma ( z,\tau, r ))^2 
\,\lesssim \,(\gamma(X,t,Kr/10))^2 \,
+\sum_{i\in I( z,\tau, r )} \epsilon^2r^{-d}r_i^{d}(r_i/r)^2 \\[4pt]
+ \,r^{-d} \iint_{\cup_{i\in  I( z,\tau, r )}2Q(i)} \biggl (\frac {d_p(Y,s,P)}r\biggr )^{2}\d\sigma (Y,s).
\end{multline}

Thus, to complete  the proof of the lemma it now suffices to prove the estimate
\begin{align}\label{cubecontain}
r^{-d}\iint_{\cup_{i\in  I( z, \tau, r )}2Q(i)} \biggl (\frac {d_p(Y,s,P)}r\biggr )^{2}\d\sigma (Y,s)
\lesssim (\gamma(X,t,Kr/10))^2.
\end{align}
To this end, we note that in turn, it is enough to prove that, for sufficiently large $\kappa$,
\begin{align}\label{auxx-}
\mbox{$2Q(i) \subseteq C_{\kappa r/10}(X,t)$  for all $i \in I(z,\tau,r)$.}
\end{align}
To establish this inclusion, we 
observe that if $i \in I(z,\tau,r)$ and $(\hat z,\hat \tau) \in I_i$, then $D(\hat z, \hat \tau) \leq D(z,\tau) + r \lesssim r$, and hence $\diam I_i \lesssim r$.  Because $d_p(\pi(Q(i)),I_i) \lesssim \diam I_i$, it follows that
\begin{align*}
d_p(\pi(Q(i)),\pi(X,t)) \leq d_p(\pi(Q(i)),I_i) + d_p(I_i,(z,\tau)) + d_p(\pi(X,t),(z,\tau)) \lesssim r.
\end{align*}
Hence
\begin{align}\label{auxx}
d_p(\pi(Q(i)),\pi(X,t)) \lesssim r.
\end{align}
To see that $2Q(i) \subseteq C_{cr}(X,t)$, choose $(Y,s) \in 2Q(i)$.  %Then $d(Y,s) \lesssim r$.  
Since $Q(i) \in \sbf$, by \eqref{eq1+} and our previous observation that $\diam I_i \lesssim r$, we obtain that
\begin{equation}\label{5.34}
d(Y,s) \lesssim \diam(Q(i)) \lesssim \diam I_i \lesssim r\,.
\end{equation}
% by the observations above and \eqref{IiQi.eq}.
Set $\pi(Y,s) = (y,s)$, $\pi(X,t) = (x,t)$.  If
$d_p(Y,s,X,t)\lesssim r$, then there is nothing to prove.  
Otherwise, if $r\lesssim d_p(Y,s,X,t)$, then using \eqref{5.34}, %since $d(Y,s) \lesssim r$,
we may apply
 Lemma \ref{lem1} to conclude that
$$ |\pi^{\perp}(Y,s) - \pi^{\perp}(X,t)| \lesssim d_p\big(\pi(Y,s) , \pi(X,t)\big) \lesssim r\,,$$
where in the last step we have also used \eqref{auxx}.
Consequently, $(Y,s) \in C_{cr}(X,t)$, so that \eqref{auxx-} holds for $\kappa$ sufficiently large.

Combining the estimates deduced, we see that
\begin{align}\label{pl42-}
(\hat\gamma ( z,\tau, r ))^2&\lesssim (\gamma(X,t,Kr/10))^2+\sum_{i\in I( z,\tau, r )} \epsilon^2r^{-d}r_i^{d}(r_i/r)^2.
\end{align}
Furthermore, the argument can be repeated with any $(Y,s)\in \Sigma\cap C_r(X,t)$ in place of $(X,t)$. Hence,  averaging gives
\begin{equation*}%\label{pl42}
(\hat\gamma ( z,\tau, r ))^2 %\\[4pt]
\lesssim r^{-d}
\iint_{\Sigma\cap C_r(X,t)}(\gamma(Y,s,Kr/10))^2\, \d\sigma(Y,s)
\,+\,\sum_{i\in I( z,\tau, r )} \epsilon^2r^{-d}r_i^{d}(r_i/r)^2,
\end{equation*}
and we have arrived at the conclusion of the lemma.
\end{proof}

\subsection{Proof of Lemma \ref{count++}}
Recall that we are assuming that $\rho\leq \kappa R/10$ (see \eqref{local}).
Let $\Lambda'$ be the set of all $i\in\Lambda$ such 
that $I_i\cap C'_{2\rho}(\hat z,\hat\tau)\neq \emptyset$. Using the construction, \eqref{2.21++bo} and Taylor's theorem we immediately deduce that
  \begin{align}\label{pl41}
  \sum_{i\in\Lambda'} \int_0^{\min\{r_i,\rho\}}\iint_{ I_i \cap C'_{2\rho}(\hat z,\hat\tau)}(\hat\gamma ( z, \tau, r ))^2\frac {\d z\d \tau\d r}r&\lesssim \epsilon^2\rho^{d}.
\end{align}
It therefore remains to estimate $\hat\gamma ( z, \tau, r )$ when either $(z, \tau)\in \pi(F)\cap 
C'_{\rho}(\hat z,\hat\tau)$, or $(z, \tau)\in I_i\cap C'_{\rho}(\hat z,\hat\tau)$, $i\in\Lambda'$, but
$r_i<r\leq\rho$. By \eqref{10-60.eq}, both of these cases are covered if we assume $(z, \tau)\in C'_{\rho}(\hat z,\hat\tau)$ and that $D(z, \tau)/60\leq r\leq \rho$. With this set up we are in a position to apply
Lemma \ref{count+}. To do so  we let $ I( z,\tau, r )$ be the set of all $i\in \Lambda$ such that $C'_r(z,\tau)\cap I_i\neq \emptyset$. Let  $(X,t)=(X(z,\tau,r),t(z,\tau,r))\in\Sigma$ be such that
$(X,t)\in Q(\sbf)$ and $d_p(z,\tau,\pi(X,t))\lesssim r$.  Then, by Lemma \ref{count+}, 
we have estimate \eqref{5.17}, which we reproduce here for the reader's convenience:
\begin{multline}\label{pl42}
(\hat\gamma ( z,\tau, r ))^2
\\[4pt]
\lesssim r^{-d}
\iint_{\Sigma\cap C_r(X,t)}(\gamma(Y,s,Kr/10))^2\, \d\sigma(Y,s)
\,+\,\sum_{i\in I( z,\tau, r )} \epsilon^2(r_i/r)^{d+2}.
\end{multline}
 The estimate in \eqref{pl42} is valid for all $(z, \tau)\in C'_{\rho}(\hat z,\hat\tau)$ satisfying $D(z, \tau)/60\leq r\leq \rho$. Hence,
 \begin{align}\label{pl43}
\iint_{C'_{\rho}(\hat z,\hat\tau)}\int_{D(z, \tau)/60}^{\rho}(\hat\gamma ( z, \tau, r ))^2\frac {\d r\d z\d \tau}r\lesssim A+\epsilon^2 B,
\end{align}
where we define $A$ to equal
 \begin{align*}
 \iint_{ C'_{\rho}(\hat z,\hat\tau)}\int_{D(z, \tau)/60}^{\rho}\biggl (r^{-d}\iint_{\Sigma\cap C_r(X(z,\tau,r),t(z,\tau,r))}(\gamma(Y,s,Kr/10))^2\, \d\sigma(Y,s)\biggr )\frac {\d r\d z\d \tau}r,
\end{align*}
 and  \begin{align*}
B:=& \iint_{ C'_{\rho}(\hat z,\hat\tau)}\int_{D(z, \tau)/60}^{\rho}\biggl (\sum_{i\in I( z,\tau, r )}(r_i/r)^{d+2}\biggr )\frac {\d r\d z\d \tau}r.
\end{align*}
Note that in $A$ we have written $(X(z,\tau,r),t(z,\tau,r))$ to highlight the dependence on $(z,\tau,r)$.

To estimate $A$ we note that for any $(Y,s,z,\tau,r)$ arising in the integral we have $$d_p(z,\tau,\pi(Y,s))\lesssim r,$$
which also implies that  $D(Y,s) \lesssim D(z,\tau) + r \lesssim  r$. Using this we see that $A$ is bounded from above by
 \begin{multline*}
\iint_{ \Sigma\cap\pi^{-1}(C'_{K\rho}(\hat z,\hat\tau))}\int_{c^{-1}D(Y, s)}^\rho r^{-d}\biggl (\iint_{C_{cr}(\pi(Y,s))}\, \d z\d\tau\biggr ) (\gamma(Y,s,Kr/10))^2
\frac {\d r\d\sigma(Y,s)}r\\[4pt]
\lesssim \iint_{ \Sigma\cap\pi^{-1}(C'_{K\rho}(\hat z,\hat\tau))}\int_{c^{-1}D(Y, s)}^\rho(\gamma(Y,s,Kr/10))^2 \frac {\d r\d\sigma(Y,s)}r\\[4pt]
\lesssim \iiint_{E(\hat z,\hat\tau,\rho)} (\ga  ( Z, \tau, Kr))^2 \, \frac{\d \si ( Z, \tau)\d r} {r}.
\end{multline*}

To estimate $B$ we first note  that if $i\in I( z,\tau, r )$, then $r_i\lesssim r$. Hence
 \begin{align*}
\sum_{i\in I( z,\tau, r )}(r_i/r)^{d+2}\lesssim \frac 1 {r^{d+1}}\sum_{i\in I( z,\tau, r )}r_i^{d+1},
\end{align*}
and we see that
\begin{align*}
B&\lesssim \iint_{ C'_{\rho}(\hat z,\hat\tau)}\int_{D(z, \tau)/60}^{\rho}\biggl (\sum_{i\in I( z,\tau, r )}r_i^{d+1}\biggr )\frac {\d r\d z\d \tau}{r^{d+2}}\notag\\
&\lesssim  \sum_{i\in \Lambda,\ I_i\cap  C'_{2\rho}(\hat z,\hat\tau)\neq\emptyset}r_i^{d+1}\iint_{ C'_{\rho}(\hat z,\hat\tau)}\int_{r_i/c}^{\infty}\mathbbm{1}_{\{(z,\tau): d_p(I_i,(z,\tau)) \lesssim r\}}(z,\tau)\frac {\d r \d z \d \tau}{r^{d+2}}\notag\\
&\lesssim  \sum_{i\in \Lambda,\ I_i\cap  C'_{2\rho}(\hat z,\hat\tau)\neq\emptyset}r_i^{d+1}\int_{r_i/c}^{\infty}(r_i+r)^d\frac {\d r}{r^{d+2}}\notag\\
&\lesssim  \sum_{i\in \Lambda,\ I_i\cap  C'_{2\rho}(\hat z,\hat\tau)\neq\emptyset}r_i^{d}\lesssim \rho^d.
\end{align*}
We can justify the second inequality in this deduction as follows.  The bounds of integration for the integral in $r$ hold because $D(z,\tau) \lesssim r \approx r_i+r$ when $i \in I(z,\tau,r)$, and the indicator appears because when $i \in I(z,\tau,r)$, $d_p(I_i,(z,\tau)) \lesssim r$ must hold. This completes the proof of Lemma \ref{count++}.

\subsection{Proof of Lemma \ref{pushlemma1}} To prove Lemma \ref{pushlemma1} we fix $(\hat z,\hat\tau)\in\mathbb R^n$ and $\rho>0$, and  we divide the proof into the following cases:
\begin{align}\label{cases}
(i)&\ C'_{2\rho}(\hat z,\hat\tau)\subset C'_{2\kappa R}(x_{ Q(\sbf)},t_{ Q(\sbf)}),\ \rho\leq\kappa R/10,\notag\\
(ii)&\ C'_{2\rho}(\hat z,\hat\tau)\subset C'_{2\kappa R}(x_{ Q(\sbf)},t_{ Q(\sbf)}),\ \rho\geq\kappa R/10,\notag\\
(iii)&\ C'_{2\rho}(\hat z,\hat\tau)\cap(\mathbb R^n\setminus C'_{2\kappa R}(x_{ Q(\sbf)},t_{ Q(\sbf)}))\neq\emptyset.
\end{align}
If $(\hat z,\hat\tau,\rho)$ is as in $(i)$ then the estimate of Lemma \ref{pushlemma1} follows immediately from Lemma \ref{count++} and the fact that $\Sigma$ is parabolic UR. Assume that $(\hat z,\hat\tau,\rho)$ is as in $(ii)$ and consider $(z, \tau)\in C'_{\rho}(\hat z,\hat\tau)$, $\rho\geq \kappa R/10$. In this case
$$\hat\nu(\hat z,\hat\tau,\rho)\lesssim\hat\nu(x_{ Q(\sbf)},t_{ Q(\sbf)},\kappa R)\lesssim (\epsilon^2+||\nu||)(\kappa R)^d\lesssim (\epsilon^2+||\nu||)\rho^d$$
by a covering argument and $(i)$. Assume that $(\hat z,\hat\tau,\rho)$ is as in $(iii)$. Recall that
$\psi\equiv 0$ on $\mathbb R^n\setminus C'_{4\kappa R}(x_{ Q(\sbf)},t_{ Q(\sbf)})$ and hence we can without loss of generality assume that
$$C'_{2\rho}(\hat z,\hat\tau)\cap C'_{4\kappa R}(x_{ Q(\sbf)},t_{ Q(\sbf)}))\neq\emptyset.$$
This case can be handled as in \eqref{pl41} by using the construction, \eqref{2.21++bo} and Taylor's theorem. We omit the remaining `routine' details and claim that the proof of Lemma \ref{pushlemma1} is complete.

% indeed, by the compact support condition \eqref{supportpsi}, it 
% suffices that we have checked the Carleson measure condition on Carleson regions 
% $C_\rho \times (0,\rho)$ that have been localized as in \eqref{local}.

Lemma \ref{pushlemma1} also proves that $\psi$ is a {\it regular} Lip(1,1/2) graph.

\begin{lemma}\label{pushlemma2}
$$\| D^{1/2}_t \psi \|_* \,\lesssim (\delta^2+||\nu||)^{1/2}.$$
In particular,
$ \psi_{\sbf}:=\psi = \psi ( x, t ) : \mathbb R^{n-1}\times\mathbb R\to \mathbb R$ is a regular  Lip(1,1/2)  function with parameters $b_1\lesssim\delta$ and $b_2=b_2(n,M,\Gamma)$.
\end{lemma}

To prove Lemma \ref{pushlemma2}, we observe that
$\hat{\nu}(\hat z, \hat \tau, \rho)$ is a (parabolic) Carleson measure on 
$\mathbb{R}^n \times \mathbb{R}_+$, by Lemma \ref{pushlemma1}. 
 After making this observation, one may use the arguments in Section 2 of \cite{HLN1} or the arguments in  Section 2 and at the beginning of Section 3 in \cite{HLN},  to produce the desired estimate on $\| D^{1/2}_t \psi \|_*$. Note that the measure used in \cite{HLN} is $\sigma^{\bf s}$, a measure which is equivalent to $\sigma$ in the setting of Lip(1,1/2) graphs and, consequently, this change does not change the validity of the arguments (see Remark \ref{r-measures}).

\begin{proof}
	We follow the argument in \cite{HLN1}; for the reader's convenience we include a detailed account of their result and make the appropriate changes for our purposes.
	
	Throughout the proof we will use simple variants of the following elementary estimate:
	\begin{equation}\label{eq.osscilationestimate}
		\sup_{C_r'(x,t)} \psi - \inf_{C_r'(x,t)}\psi \lesssim \delta r, \qquad (x,t)\in \RR^n,
	\end{equation}
	where the implicit constant is the same as for the $Lip(1,1/2)$ constant for $\psi$ obtained in Lemma \ref{graphlip.lem}.
	
	For the rest of the proof we will fix a point $(z,\tau)\in \RR^n$ and a scale $\rho>0$. Associated to these parameters we choose a cut-off function $\beta=\beta_{z,\tau,\rho} \in C_c^\infty(C_{2\rho}'(z,\tau))$, $0\leq \beta \leq 1$, and $\beta \equiv 1$ in $C_{3\rho/2}'(z,\tau)$. To simplify notation we will also set 
	$$
	\CC:= C_\rho'(z,\tau).
	$$
	
	We define 
	$$
	\psi_1 +\psi_2 := [(\psi-\psi_{\CC})\beta] + [ \beta\psi_{\CC} + (1-\beta)\psi],
	$$ 
	where as usual $\psi_\CC$ denotes the average (with respect to Lebesgue measure) of $\psi$ over $\CC$.
	
	Recall that we aim to control the $BMO$ norm of $D_t^{1/2}\psi$, or in other words, we wish to prove
	\begin{equation}\label{eq.BMOestimate}
		\iint_{\CC} |D_s^{1/2}\psi - (D_s^{1/2}\psi)_{\CC}|\, \d y \d s \lesssim (\delta^2+ \| \nu\|)\rho^d,
	\end{equation}
	with appropriate dependence of the implicit constants on allowable parameters\footnote{As will be seen from the proof, the implicit constants will depend on the same parameters as the $Lip(1,1/2)$ estimate from Lemma \ref{graphlip.lem}.}.
	
	To prove \eqref{eq.BMOestimate}, add and subtract a constant $\alpha$ (to be determined) and use the decomposition $\psi=\psi_1+\psi_2$ from before to get 
	\begin{equation*}
		\begin{split}
			\iint_{\CC} |D_s^{1/2}\psi- (D_s^{1/2}\psi)_{\CC}|\, \d y \d s & \leq \iint_{\CC} |D_s^{1/2}\psi_2 - \alpha|\, \d y \d s + |\CC||\alpha-(D_s^{1/2}\psi)_{\CC}| \\
			& \qquad + \iint_{\CC} |D_s^{1/2} \psi_1|\, \d y \d s.
		\end{split}
	\end{equation*}
	Notice that for the middle term we have 
	$$
	|\alpha - (D_s^{1/2}\psi)_{\CC}| \leq |\CC|^{-1}\iint_{\CC} |\alpha -D_s^{1/2}\psi_2 - D_s^{1/2}\psi_1|\, \d y \d s,
	$$
	and so 
	\begin{equation}
		\begin{split}
			\iint_{\CC} |D_s^{1/2}\psi- (D_s^{1/2}\psi)_{\CC}|\, \d y \d s& \leq 2\iint_{\CC} |D_s^{1/2}\psi_2 - \alpha|\, \d y \d s\\
			&\quad + 2\iint_{\CC} |D_s^{1/2} \psi_1|\, \d y \d s\\
			& =: I + II
		\end{split}
	\end{equation}

	To handle $II$ we need some preparation. Introduce first an approximate identity $P_\lambda$ and abuse notation to denote both the operator and it's kernel in this fashion. In particular, we require $P\in C_c^\infty(C_1'(0,0))$, even and nonnegative with $\iint_{\RR^n} P\, \d y \d s =1$, and as usual set $P_\lambda(y,s)= \lambda^{-d}P(y/\lambda, s/\lambda^2)$. 
	
	In what follows $\phi_\lambda(y,s)=\lambda^{-d}\phi(y/\lambda, s/\lambda^2)$ may denote either of the following functions:
	$$
	\lambda \dfrac{\partial P}{\partial \lambda}, \qquad \lambda^2 \dfrac{\partial P}{\partial s}, \qquad \lambda^2 \dfrac{\partial^2 P}{\partial y_i^2}\quad \textup{for} \quad 1\leq i\leq n-1.
	$$
	The main features of these functions for us are: 
	\begin{itemize}
		\item $\displaystyle \iint_{\RR^n} |\phi_{\lambda}|\, \d y\d s=\iint_{\RR^n} |\phi|\, \d y\d s \lesssim 1, \quad \sup_{\RR^n} |\phi_\lambda| = \lambda^{-d}\sup_{\RR^n} |\phi| \lesssim \lambda^{-d}.$
		\item For each $1\leq i\leq n$ we have 
		$$
		\iint_{\RR^n} \phi\, \d y\d s =0 = \iint_{\RR^n} y_i\phi\, \d y\d s.
		$$
	\end{itemize}
	By the second property of $\phi$ above we have, for any linear function $L$ of $y$ only, 
	\begin{equation*}
		\begin{split}
			(\phi_\lambda*\psi)^2(y,s) & = (\phi_\lambda *(\psi-L))^2(y,s)\\
			& \leq \| \phi\|_\infty^2 \lambda^{-2d} \left(\iint_{C_\lambda'(y,s)} |\psi-L|\, \d x\d t\right)\\
			& \lesssim_{n} \lambda^{-d} \iint_{C_{\lambda}' (y,s)} |\psi-L|^2\, \d x\d t,
		\end{split}
	\end{equation*}
	where we used H\"older's inequality in the first line, and Cauchy-Schwarz together with the first property of $\phi$ above on the second.
	
	Multiplying by $\lambda^{-3}$ and integrating the above equation over $C_\rho'(z,\tau)\times (0,\rho)$ we arrive at
	$$
	\int_0^\rho \iint_{C_\rho'(z,\tau)} \lambda^{-3}(\phi_\lambda* \psi)^2\, \d y\d s\d\lambda \lesssim_n \int_0^\rho \iint_{C_\rho(z,\tau)}\lambda^{-1} \hat\gamma(y,s)\, \d y\d s\d\lambda,
	$$
	or, in other words,
	\begin{equation}\label{eq.phiestimate}
		\int_0^\rho \iint_{C_\rho'(z,\tau)} \lambda^{-3}(\phi_\lambda* \psi)^2\, \d y\d s\d\lambda\lesssim_n \| \hat{\nu}\|\rho^d.
	\end{equation}
	We are now ready to begin estimating $II$ in \eqref{eq.BMOestimate}. Fix $0<\varepsilon<\rho$, then using the fundamental theorem of calculus 
	\begin{equation}
		\begin{split}
			\iint_{\RR^n} (D_s^{1/2}P_\varepsilon\psi_1)^2\, \d y\d s & = \iint_{\RR^n} (D_s^{1/2}P_\rho\psi_1)^2\, \d y\d s \\
			& \quad - \int_{\varepsilon}^\rho\iint_{\RR^n} 2[D_s^{1/2}(\partial_\lambda P_\lambda\psi_1)][D_s^{1/2}(P_\lambda\psi_1)]\, \d y\d s\d\lambda\\
			& = \iint_{\RR^n} (D_s^{1/2}P_\rho\psi_1)^2\, \d y\d s \\
			& \quad + c\int_{\varepsilon}^\rho \iint_{\RR^n} \Hil\left( \partial_\lambda (P_\lambda \psi_1)\right) \partial_t(P_\lambda \psi_1)\, \d y\d s \d\lambda,\\
			& = II_1 + II_2.
		\end{split}
	\end{equation}
	where $\Hil$ denotes the Hilbert transform and we have used the following facts:
	$$
	D_s^{1/2}D_s^{1/2}= \Hil \partial_s, \qquad \Hil^*= -\Hil.
	$$
	To handle $II_1$ we first claim that it's enough to show
	\begin{equation}\label{eq.psi1estimate}
		|D_s^{1/2}P_\rho \psi_1(y,s)| \lesssim_n \delta \min\left\{ 1, \dfrac{\rho^3}{|s-\tau|^{3/2}}\right\}  \mathbbm{1}_{|y-z|\leq 4\rho}.
	\end{equation}
	Assuming this for the moment, we integrate in $(y,s)$ to get 
	\begin{equation}
		II_1 \lesssim \iint_{\substack{|z-y|<4\rho\\ s\in \RR}} \delta^2 \min\left\{ 1, \dfrac{\rho^3}{|\tau-s|^{3/2}}\right\}^2\, \d y \d s \lesssim \delta^2 \rho^d,
	\end{equation}
	which is the right type of bound (see \eqref{eq.BMOestimate}). Therefore, it remains only to prove \eqref{eq.psi1estimate}. For this purpose recall that for nice functions\footnote{Such as $P_\rho \psi_1\in C_c^\infty(\RR^n)$.} $f$ we defined
	$$
	D_s^{1/2} f(y,s):= c\int_{\RR} \dfrac{f(y,t)-f(y,s)}{|t-s|^{3/2}}\, \d t.
	$$
	In particular, since $\psi_1$ is supported on $C_{2\rho}'(z,\tau)$, if $|z-\tau|>4\rho$ then $D_s^{1/2}P_\rho \psi_1(y,s)=0$. We have thus reduced \eqref{eq.psi1estimate} to proving
	\begin{equation}\label{dhalftimedecayspatialclose.eq}
	|D_s^{1/2} P_\rho \psi_1(y,s)|\lesssim_n \delta \min\left\{ 1, \dfrac{\rho^3}{|s-\tau|^{3/2}}\right\}, \qquad |y-z|\leq 4\rho.
	\end{equation}
	First we note that 
	\begin{equation}
		|\partial_s P_\rho \psi_1(y,s)|\leq \dfrac{\delta}{\rho}, \qquad (y,s)\in \RR^n.
	\end{equation}
	This follows from properties of the convolution:
	$$
	|\partial_s P_\rho \psi_1(y,s)|= \rho^{-2}| (\partial_sP)_\rho\psi_1(y,s)| \leq \rho^{-2}\| (\partial_sP)_\rho\|_{L^1(\RR^n)} \| \psi_1\|_{L^\infty(\RR^n)} \lesssim \dfrac{\delta}{\rho}, 
	$$
	where we used \eqref{eq.osscilationestimate} and the definition of $\psi_1$ in the last inequality.
	
	We first get the uniform bound; for this we break up the integral in the definition of the half-order time derivative and the above estimate, together with \eqref{eq.osscilationestimate}, to get
	\begin{equation*}
		\begin{split}
			|D_s^{1/2} P_\rho \psi_1(y,s)| & \leq \int_{|t-\tau|<\rho^2}  \dfrac{|P_\rho\psi_1(y,t)-P_\rho\psi_1(y,s)|}{|t-s|^{3/2}}\, \d t \\
			& \quad + \int_{|t-\tau|>\rho^2}  \dfrac{|P_\rho\psi_1(y,t)-P_\rho\psi_1(y,s)|}{|t-s|^{3/2}}\, \d t\\
			& \lesssim \dfrac{\delta}{\rho}\int_{|t-\tau|<\rho^2} |t-s|^{-1/2}\, \d t + \delta\rho \int_{|t-\tau|>\rho^2} |t-s|^{-3/2}\, \d t\\
			& \lesssim \delta.
		\end{split}
	\end{equation*}
	Therefore it remains to get the decay in estimate \eqref{dhalftimedecayspatialclose.eq}, and, for this, we need only consider\footnote{Perhaps worsening our implicit constant by an universal factor.} $|s-\tau|>100\rho^2$. For such $s$ we have $P_\rho\psi_1(y,s)=0$ and so 
	\begin{equation*}
		\begin{split}
			|D_s^{1/2}P_\rho \psi_1(y,s)| & \leq \int_{\RR} \dfrac{|P_\rho \psi_1(y,t)|}{|t-s|^{3/2}}\, \d t = \int_{|t-\tau|<16\rho^2}\dfrac{|P_\rho\psi_1(y,t)|}{|t-s|^{3/2}}\, \d t \\
			& \lesssim \rho \delta \int_{|t-\tau|<16\rho^2} |t-s|^{-3/2}\, \d t \approx \rho\delta \rho^2|\tau-s|^{-3/2},
		\end{split}
	\end{equation*}
	where we again used \eqref{eq.osscilationestimate} in the second line and the fact that $|t-s|\approx|\tau-s|$ with the given restrictions on $s,t$. This is exactly the desired decay in \eqref{dhalftimedecayspatialclose.eq}, and thus we've finished controlling $II_1$.
	
	$II_2$ is a bit more elaborate. We first apply Cauchy-Schwarz to get 
	\begin{equation*}
		\begin{split}
			|II_2|^2 &\lesssim_n \left(\int_\varepsilon^\rho \iint_{\rn} \lambda^{-1}|\Hil(\partial_\lambda(P_\lambda\psi_1))|^2\, \d y\d s\d\lambda \right) \\
			&\quad \times \left(\int_\varepsilon^\rho \iint_{\rn} \lambda |\partial_t(P_\lambda\psi_1)|^2\, \d y\d s\d\lambda\right)\\
			& =: II_{21}\cdot II_{22}.
		\end{split}
	\end{equation*}
	From the $L^2$ boundedness of the Hilbert transform on $\RR$ we see that 
	\begin{equation}\label{eq.II21}
		II_{21}\leq \int_\varepsilon^\rho \iint_{\RR^n} \lambda^{-1}|\partial_\lambda(P_\lambda \psi_1)|^2\, \d y\d s\d\lambda.
	\end{equation}
	With this in hand, the estimates for $II_{21}$ and $II_{22}$ will follow the same lines, the different powers of $\lambda$ accounting for the difference in scaling of $\partial_t P_\lambda$ and $\partial_\lambda P_\lambda$. In particular one should keep in mind the fact that both of these functions ``annihilate" first order (spatial) polynomials (see the properties of $\phi$ above).
	
	With the above in hand, we focus on $II_{21}$. 
	
	For fixed $(y,s)\in \CC$ consider the linear polynomial 
	$$
	L(x):= \beta(y,s)+ \sum_{i=1}^{n-1} \partial_{y_i}\beta(y,s)(y_i-x_i), \qquad x\in \RR^{n-1}.
	$$
	
	Using the alluded properties of $\partial_\lambda P_\lambda$ and the triangle inequality we see
	\begin{equation}
		\begin{split}
			|\partial_\lambda(P_\lambda \psi_1)(y,s)| & \leq \left| \iint_{\RR^n} \partial_\lambda P_\lambda(y-x,s-t)[\psi(x,t)-\psi_{\CC}][\beta(x,t)-L(x)]\, \d x\d t\right|\\
			& \quad + \sum_{i=1}^{n-1} \left| \partial_{y_i}\beta(y,s)\iint_{\RR^n} (y_i-x_i)\partial_\lambda P_\lambda(y-x,s-t)[\psi(x,t)-\psi(y,s)]\, \d x\d t \right|\\
			& \quad + |\beta(y,s)\partial_\lambda P_\lambda\psi (y,s)| 
			=: A+B+C.
		\end{split}
	\end{equation}
	To handle $A$ recall that $\| \partial_\lambda P_\lambda\|_{L^1}\lesssim \lambda^{-1}$, and \eqref{eq.osscilationestimate} (together with the fact that the integration is taking place over $C_\lambda'(y,s)$ and $\lambda<\rho$) to get 
	$$
	A \lesssim \lambda^{-1} \delta \rho \sup_{(x,t)\in C_\lambda(y,s)} |\beta(x,t)-L(x)|.
	$$
	To control the last term we appeal to Taylor's theorem: 
	\begin{equation*}
		\begin{split}
			|\beta(x,t)-L(x)| & \leq \left|\beta(x,t)-\beta(y,t)-\sum_{i=1}^{n-1}\partial_{y_i}\beta(y,t)(y_i-x_i) \right| \\
			&\quad + \left| \beta(y,t)-\beta(y,s)\right| + \sum_{i=1}^{n-1}\left| \partial_{y_i}\beta(y,t)-\partial_{y_i}\beta(y,s)\right||y_i-x_i|  \\
			& \lesssim \| D^2_y \beta(\cdot, t)\|_{L^\infty(\RR^{n-1})} |x-y|^2 + \| \partial_s \beta (y,\cdot)\|_{L^\infty(\RR)}|t-s| \\
			& \quad + \| \partial_s D_y \beta(y,\cdot)\|_{L^\infty(\RR)} |x-y||t-s|\\
			& \lesssim \rho^{-2}\lambda^2 + \rho^{-2}\lambda^2+ \rho^{-3} \lambda^3\\
			& \leq \rho^{-2}\lambda^2,
		\end{split}
	\end{equation*}
	where we used the derivative bounds for $\beta$ and in the second to last line the fact that $\lambda<\rho$. Plugging this estimate, we arrive at 
	$$
	A\lesssim \dfrac{\delta \lambda}{\rho}.
	$$
	
	Using again the derivative estimates for $\beta$ and \eqref{eq.osscilationestimate} we see 
	$$
	B\lesssim \rho^{-1} \lambda \lambda^{-1} \delta \lambda = \dfrac{\delta\lambda}{\rho}.
	$$
	
	Trivially, since $\beta\leq 1$ we have 
	$$
	C\lesssim |\partial_\lambda P_\lambda  \psi(y,s)|,
	$$
	and so, combining all these estimates into \eqref{eq.II21} we arrive at 
	\begin{equation*}
		\begin{split}
			II_{21} & \lesssim \int_{\varepsilon}^\rho \iint_{C_{4\rho}'(z,\tau)}\lambda^{-1}\left( \dfrac{\delta\lambda}{\rho} + |\partial_\lambda P_\lambda  \psi(y,s)|\right)^2\, \d y\d s \d\lambda\\
			& \lesssim \int_0^\rho \iint_{C_{4\rho}'(z,\tau)}\left(  \dfrac{\delta^2\lambda}{\rho^2} + \lambda^{-1}|\partial_\lambda P_\lambda \psi(y,s)|^2\right) \, \d y\d s \d\lambda\\
			& \lesssim \delta^2\rho^d + \| \hat{\nu}\|\rho^d,
		\end{split}
	\end{equation*}
	where we have used \eqref{eq.phiestimate} in the last line with $\phi= \lambda\partial_\lambda P_\lambda$.

	To handle $II_{22}$ the only change is that $\| \partial_s P_\lambda\|_{L^1(\RR^n)}\lesssim \lambda^{-2}$, which gives
	$$
	|\partial_s P_\lambda \psi_1(y,s)| \lesssim \dfrac{\delta}{\rho} + |\partial_s P_\lambda \psi(y,s)|,
	$$
	and thus 
	$$
	II_{22}\lesssim \delta^2\rho^d + \|\hat{\nu}\|\rho^d.
	$$
	These last two estimates, in turn, give 
	$$
	II_2 \lesssim (\delta^2+\|\hat{\nu}\|)\rho^{d}.
	$$
	Combining these estimates we have 
	$$
	\iint_{\RR^n} |D_s^{1/2}P_\varepsilon \psi_1|^2\, \d y\d s \lesssim (\delta^2+ \|\hat{\nu})\rho^{d},
	$$
	with bounds uniform in $\varepsilon$. In fact the implicit constants only depend on the parameters upon which the $Lip(1,1/2)$ bound in Lemma \ref{graphlip.lem} depends.
	
	By weak compactness, there is a subsequence (which we denote the same) and an $h\in L^2(\RR^n)$ such that $D_s^{1/2}P_\varepsilon \psi_1 \to h$, as $\varepsilon\to 0^+$. Since $P_\varepsilon\psi_1 \in C_c^\infty(\RR^n)$ and $\psi_1$ has compact support we see that $I_{1/2}*(D_s^{1/2} P_\varepsilon \psi_1)=P_\varepsilon\psi_1$ converges pointwise, uniformly, everywhere to $\psi_1$ and to $I_{1/2}h$ (here $I_{1/2}$ is the fractional integral of order $1/2$ in the time variable, appropriately normalized). This shows that $D_s^{1/2}\psi_1$ exists, and satisfies the estimate
	\begin{equation*}
		\iint_{\RR^n} |D_s^{1/2}\psi_1|^2\, \d y\d s\lesssim (\delta^2+\|\hat{\nu}\|)\rho^{d}.
	\end{equation*}
	
	Therefore
	$$
	II=\iint_{\CC} |D_s^{1/2}\psi_1|\, \d y\d s \lesssim \rho^{d/2} \left(\iint_{\CC}|D_s^{1/2}\psi_1|^2\, \d y\d s\right)^{1/2}\lesssim (\delta^2+\|\hat{\nu})^{1/2}\rho^d.
	$$
	
	It thus remains to control $I$ in \eqref{eq.BMOestimate}. First, notice that for $(y,s)\in C_{3\rho/2}'(z,\tau)$ we have $\psi_2(y,s)=\psi_{\CC}$, by definition of $\psi_2$, so that in particular the integral defining $D_s^{1/2}\psi_2$ is always convergent for such points. By the decomposition $\psi=\psi_1+\psi_2$ we conclude that $D_s^{1/2}\psi$ exists in $\CC$.
	
	Write 
	$$
	D_s^{1/2}\psi_2(y,s) = c\int_{|t-\tau|>9\rho^2/4} \dfrac{\psi_2(y,t)-\psi_2(y,s)}{|t-s|^{3/2}}\, \d t, \qquad (y,s)\in \CC,
	$$
	and set $\alpha$ to be 
	$$
	\alpha:= c\int_{|t-\tau|>9\rho^2/4} \dfrac{\psi(z,t)-\psi(z,\tau)}{|t-\tau|^{3/2}}\, \d t.
	$$
	If we can show, for this value of $\alpha$,
	\begin{equation}\label{eq.psi2estimate}
		|D_s^{1/2}\psi_2(y,s) -\alpha|\lesssim \delta, \qquad (y,s)\in \CC,
	\end{equation}
	then integrating would give 
	$$
	I= \iint_{\CC} |D_s^{1/2}\psi_2 -\alpha|\, \d y\d s \lesssim \delta\rho^d,
	$$
	which combined with the estimate for $II$, gives the desired $BMO$ bound in \eqref{eq.BMOestimate}. The rest of the proof is then devoted to showing \eqref{eq.psi2estimate}. 
	
	First notice that for $(y,s)\in \CC$ and $|t-\tau|>9\rho^2/4$, we have $|t-\tau|\approx |t-s|$; we will exploit this repeatedly below without mention.
	
	From the above formulas we have 
	\begin{equation*}
		\begin{split}
			|D_s^{1/2}\psi_2(y,s)-\alpha|  = c\left| \int_{|t-\tau|>9\rho^2/4} \left( \dfrac{\psi_2(y,t)-\psi_2(y,s)}{|t-s|^{3/2}} - \dfrac{\psi(z,t)-\psi(z,\tau)}{|t-\tau|^{3/2}}\right) \, \d t\right|
		\end{split}
	\end{equation*}
	Recall $\psi_2=(\psi_{\CC}-\psi)\beta + \psi$, so that we can rewrite the quantity in parentheses as
	$$
	\dfrac{(\psi_{\CC}-\psi(y,t))\beta(y,t)}{|t-s|^{3/2}} - \dfrac{(\psi_{\CC} -\psi(y,s))\beta(y,s)}{|t-s|^{3/2}} + \left(\dfrac{\psi(y,t)-\psi(y,s)}{|t-s|^{3/2}}- \dfrac{\psi(z,t)-\psi(z,\tau)}{|t-\tau|^{3/2}}\right)
	$$
	Call these terms $A_1, A_2, A_3$ respectively.
	
	For the first one, notice that $\beta(y,t)$, as a function of $t$ is supported on $|t-\tau|<4\rho^2$ by definition, so 
	$$
	\int_{|t-\tau>9\rho^2/4} |A_1|\, \d t = \int_{|t-\tau|\approx \rho^2} \dfrac{|\psi_{\CC} -\psi(y,t)|}{|t-s|^{3/2}}\, \d t\approx \int_{|t-\tau|\approx \rho^2} \dfrac{|\psi_{\CC} -\psi(y,t)|}{|t-\tau|^{3/2}}\, \d t,
	$$
	and since then $(y,t)\in C_{4\rho}'(z,\tau)$, we use \eqref{eq.BMOestimate} to bound the numerator by  $\delta\rho$, which gives
	$$
	\int_{|t-\tau|>9\rho^2/4} |A_1|\, \d t \lesssim \delta\rho \int_{t\approx \rho^2} t^{-3/2}\, \d t \lesssim \delta.
	$$
	
	For $A_2$ we use that $|\beta|\leq 1$ and $(y,s)\in \CC$ to again bound the numerator by $\rho\delta$ and so
	$$
	\int_{|t-\tau|>9\rho^2/4} |A_2|\, \d t \lesssim \rho\delta \int_{\rho^2}^\infty t^{-3/2}\, \d t \lesssim \delta.
	$$
	
	For $A_3$ we first put a common denominator on the quantity in parentheses and rewrite the result as 
	$$
	\dfrac{[(\psi(y,t)-\psi(z,t))-(\psi(y,s)-\psi(z,\tau))]|t-\tau|^{3/2}+(\psi(z,t)-\psi(z,\tau))(|t-\tau|^{3/2}-|s-t|^{3/2})}{|t-\tau|^{3/2}|t-s|^{3/2}}.
	$$
	
	Now notice that for any given $t$, the points $(y,t), (z,t)$ lie in a (parabolic) cylinder of size $<\rho$, and so each term inside the brackets can be controlled, using \eqref{eq.osscilationestimate}, by $\delta\rho$. This an the $Lip(1,1/2)$ estimate for $\psi$ directly then give 
	\begin{equation*}
		\begin{split}
			|A_3|&\lesssim \dfrac{\delta\rho}{|t-s|^{3/2}} + \delta\dfrac{|t-\tau|^{1/2}(|t-\tau|^{3/2}-|t-s|^{3/2})}{|t-\tau|^{3/2}|t-s|^{3/2}}\\
			& \lesssim \dfrac{\delta\rho}{|t-\tau|^{3/2}} + \delta\dfrac{|t-\tau|^{1/2}(|t-\tau|^{1/2}+|t-s|^{1/2})|s-\tau|}{|t-\tau|^{3/2}|t-s|^{3/2}}\\
			& \lesssim \dfrac{\delta\rho}{|t-\tau|^{3/2}} + \dfrac{\delta\rho^2}{|t-\tau|^2},
		\end{split}
	\end{equation*}
	where we used the elementary inequality $|a^p-b^p|\lesssim (a^{p-1}+b^{p-1})|a-b|$ in the second line. Integrating this gives
	$$
	\int_{|t-\tau|>9\rho^2/4}|A_3|\, \d t \lesssim \delta,
	$$
	which finishes the proof.
\end{proof}

\section{Packing the maximal cubes}\label{packsttime}
In this section we conclude the proof of Theorem \ref{main}, by proving the packing condition for the maximal cubes $\{Q(\sbf)\}_{\sbf \in \mathcal{F}}$.
In order to reduce matters, we need to introduce some additional notation. Recall that $m(\sbf)$ is the collection of minimal cubes in a stopping time regime $\sbf \in \mathcal{F}$. We let $m_0(\sbf)$ denote the cubes in $m(\sbf)$ which are such that $Q$ has a child in $\mathcal{B}'$ and we let $m_1(\sbf)$
denote the cubes in $m(\sbf)$ for which the angle between $P_{ Q(\sbf)}$ and $P_Q$ is greater or equal to $\delta/2$. Hence, by construction
$$m(\sbf)=m_0(\sbf)\cup m_1(\sbf).$$
Following \cite{DS1}, we now let
\begin{align}
\mathcal{F}_0&:=
\{\sbf\in\mathcal{F}:\ \sigma\bigl (\cup_{Q\in m_0(\sbf)}Q\bigr )\geq \sigma(Q(\sbf))/4\},\notag\\
\mathcal{F}_1&:=
\{\sbf\in\mathcal{F}:\ \sigma\bigl (\cup_{Q\in m_1(\sbf)}Q\bigr )\geq \sigma(Q(\sbf))/2\},\notag\\
\mathcal{F}_2&:=
\{\sbf\in\mathcal{F}:\ \sigma\bigl (Q(\sbf)\setminus\cup_{Q\in m(\sbf)}Q\bigr )\geq \sigma(Q(\sbf))/4\}.
\end{align}
Then $\mathcal{F}=\mathcal{F}_0\cup\mathcal{F}_1\cup \mathcal{F}_2$.  It follows (see \cite[Lemma 7.4]{DS1}) that
$$\sum_{\sbf\in\mathcal{F}_0\cup\mathcal{F}_2: Q(\sbf)\subset Q}\sigma\big(Q(\sbf)\big)\,\leq\,  c(n,M,\Gamma,\epsilon,\delta,K)\, \sigma(Q)\,,
\quad \forall Q\in \dd(\Sigma)\,.$$
The ideas behind this estimate are that to obtain the bound for $\mathcal{F}_2$, one uses that the sets $Q(\sbf) \setminus (\cup_{Q \in m(\sbf)} Q)$ are pairwise disjoint, while the bound for $\mathcal{F}_0$ follows from the packing condition which by construction holds for $\mathcal{B}'$. We include the details of these arguments below.
In the case of $\mathcal{F}_2$ we have for each $Q(\sbf)$ a set $F_{\sbf} \subseteq Q(\sbf)$ with each $F_{\sbf}$ disjoint and $\sigma(F_{\sbf}) \ge (1/4) \sigma(Q(\sbf))$. Thus, for $\widetilde{Q}$ any cube
\[\sum_{\substack{Q(\sbf) \subseteq \widetilde{Q} \\ \sbf \in \mathcal{F}_2} } \sigma(Q(\sbf)) \le 4 \sum_{\substack{Q(\sbf) \subseteq \widetilde{Q} \\ \sbf \in \mathcal{F}_2} } \sigma(F_{\sbf}) \le 4 \sigma(\widetilde{Q}),\]
where we used that $F_{\sbf} \subseteq Q(\sbf) \subseteq \widetilde{Q}$ are disjoint, so their measures cant add up to more than $\sigma(\widetilde{Q})$. 

The $\mathcal{F}_0$ stopping times have a different packing argument.  If $\sbf$ is in $\mathcal{F}_0$ we have for each $Q(\sbf)$ and associated collection of subcubes $Q_j$ of $Q(\sbf)$, call them $\mathcal{B}_\sbf$, which have a child $Q_j'$ in $\mathcal{B}'$. In particular, we take  $\mathcal{B}_\sbf$ to be the $m_0(\sbf)$ cubes and they satisfy
\[\sum_{Q \in \mathcal{B}_\sbf}  \sigma(Q) = \sigma\left(\bigcup_{Q \in \mathcal{B}_\sbf} Q\right) \ge (1/4) \sigma(Q(\sbf)).\]  
One notes that if $\sbf,\sbf'$ is in $\mathcal{F}_0$ are distinct then $\mathcal{B}_\sbf \cap \mathcal{B}_{\sbf'}$, which can be seen by the disjointness of the stopping times. Thus, if $\widetilde{Q}$ is a cube
\begin{align*}
\sum_{\substack{Q(\sbf) \subseteq \widetilde{Q} \\ \sbf \in \mathcal{F}_0} } \sigma(Q(\sbf)) &\le 4 \sum_{\substack{Q(\sbf) \subseteq \widetilde{Q} \\ \sbf \in \mathcal{F}_0} } \sum_{Q \in \mathcal{B}_\sbf} \sigma(Q) 
\\ & \le 4C_*\sum_{\substack{Q' \in \mathcal{B}' \\ Q' \subseteq \widetilde{Q}}} \sigma(Q')   \le 4C_*C\sigma(\widetilde{Q}),
\end{align*}
where the $C$ is the packing constant for the cubes $\mathcal{B}'$ and $C_*$ is a constant so that if $Q'$ is a child of $Q$ then $\sigma(Q) \le C_* \sigma(Q')$.  Thus, to pack the maximal cubes (and hence prove Theorem \ref{main}) it suffices to prove that
\begin{align}
\sum_{\sbf\in\mathcal{F}_1: Q(\sbf)\subset Q}\sigma\big(Q(\sbf)\big)\,\leq\,  c(n,M,\Gamma,\epsilon,\delta,K)\, \sigma(Q)\,,
\quad \forall Q\in \dd(\Sigma)\,.
\end{align}

Consider $\sbf\in\mathcal{F}$ and assume that
\begin{align}\label{pack++}\iiint_{E(\sbf)} (\ga  ( Z, \tau, Kr))^2 \, \frac{\d \si ( Z, \tau)\d r} {r}\leq\epsilon^2\sigma(Q(\sbf)),
\end{align}
where $E(\sbf):=E(x_{Q(\sbf)},t_{Q(\sbf)}, R)$ is the region appearing in Lemma \ref{count++}, and  $R:=\diam(Q(\sbf))$, see \eqref{nota2}. I.e. $E(\sbf)$ is the set of all $(Z,\tau,r)\in \big[\Sigma\cap\pi^{-1}(C'_{KR}(x_{Q(\sbf)},t_{Q(\sbf)}))\big]\times(0,\infty)$ such that
$ K^{-1}d(Z,\tau)\leq r\leq KR$.

We are the going to prove that there exists $\epsilon$ small,
independent of $\sbf$, such that if \eqref{pack++} holds then,
$\sbf\notin\mathcal{F}_1$. Hence, if $\sbf\in\mathcal{F}_1$ then the opposite inequality in \eqref{pack++} must hold and therefore, if $Q\in \dd(\Sigma)$ then
\begin{align}\label{pack+}
\sum_{\sbf\in\mathcal{F}_1: Q(\sbf)\subset Q}\sigma\big(Q(\sbf)\big)&\leq \epsilon^{-2}
\sum_{\sbf\in\mathcal{F}_1: Q(\sbf)\subset Q} \iiint_{E(\sbf)} (\ga  ( Z, \tau, Kr))^2 \, \frac{\d \si ( Z, \tau)\d r} {r}\notag\\
&\leq c(n,M,\Gamma,\epsilon,\delta,K)\, \sigma(Q),
\end{align}
by the facts that $\Sigma$ is parabolic UR, that $\sigma$ is a doubling measure, and that the sets $\{E(\mathcal{S})\}_{\mathcal{S} \in \mathcal{F}}$ have bounded overlap in $\Sigma \times (0,\infty)$.  Verification of the
latter fact follows by first noting, for each $(Z,\tau,r) \in E(\mathcal{S})$, that there is some cube $Q \in \mathcal{S}$ such that
$$d_p(Z,\tau, Q) + \diam Q \leq 2Kr.$$
Hence, upon replacing $Q$ with a suitable ancestor in $\mathcal{S}$, we have
$$d_p(Z,\tau, Q) \lesssim r \text{ and } \diam Q \approx r.$$
Clearly there are only a bounded number of cubes satisfying the inequality in this display for each $(Z,\tau,r)$, hence  the sets $\{E(\mathcal{S})\}_{\mathcal{S} \in \mathcal{F}}$ have bounded overlap in $\Sigma \times (0,\infty)$.

Based on this, and for the sake of obtaining a contradiction, we will from now on assume that \eqref{pack++} holds and that $\sbf \in \mathcal{F}_1$. In the following we will at instances use the short notation $C'_{\rho}:=C'_{\rho}(x_{ Q(\sbf)},t_{ Q(\sbf)})$, $\rho>0$. Using  this notation, the assumption in \eqref{pack++} and
 Lemma \ref{count++}, we can conclude the validity of the important estimate
\begin{align}\label{implication}
\int_0^{\kappa R/10}\iint_{ C'_{\kappa R}}(\hat\gamma ( z, t, r ))^2\frac {\d z\d \tau\d r}r&\lesssim \epsilon^2\sigma(Q(\sbf)).
\end{align}

We are now going to study the implications of \eqref{implication}. To do this, let $ \ph $ be an infinitely differentiable real valued
function on $ \mathbb R^n$  with compact support in $ C'_1 (0,0) $ and set
\begin{align*} \ph_\la (z, \tau ) \, &\equiv \, \la^{- d } \ph  (  z/ \la ,
\tau / \la^2 ).
\end{align*}
For $  f :
\mathbb R^n \rar \mathbb R, $ we introduce the usual convolution of $ \ph $ with $ f$
by
\begin{align}  \ph * f (z, \tau)  = \iint_{\mathbb R^n } \ph ( z
- y, \tau - s ) f ( y, s )\,   \d y \d s,
\end{align} whenever this
convolution makes sense. Furthermore, we suppose that
\begin{align}\label{fi2}\iint_{\mathbb R^n } \ph ( y, s ) \d y\d s  = 0
\, \, \mbox{ and }
 \iint_{\mathbb R^n } y_i \ph ( y, s ) \d y\d s = 0\mbox{ for $1 \leq i \leq n-1$}.
 \end{align}
If $L$ is a  linear function of $ y $ only, then
\begin{align} ( \ph_\la  * \psi )^2 ( z,  \tau ) &=
 [ \ph_\la * ( \psi -
 L ) ]^2  ( z,  \tau) \notag\\
 &\lesssim \la^{ - d } \,
\iint_{C'_\la ( z, \tau ) } \, |  \psi - L |^2 \d y\d s.
\end{align}
In particular,
\begin{align} ( \ph_\la  * \psi )^2 (  z,  \tau ) \lesssim \lambda^2(\hat\gamma(  z,  \tau,\lambda ))^2.
\end{align}
From this observation and \eqref{implication} we see that
\begin{align}\label{2.34}
 { \ds \int_0^\rho  \iint_{C'_\rho ( \hat z, \hat \tau) } }  (
\ph_\la
 *  \psi )^2 \, ( z, \tau )  \frac {\d z\d \tau d\la}{\lambda^3}&\lesssim
 {\ds \int_0^\rho \iint_{ C'_\rho ( \hat z, \hat \tau )}}
   (\hat\gamma ( z, \tau, \la ))^2
\frac {\d z\d \tau\d\la}{\lambda}\notag\\
&\lesssim  \epsilon^2 \sigma(Q(\sbf))
  \end{align}
  whenever $ C'_\rho ( \hat z, \hat \tau)\subset  C'_{\kappa R}(x_{ Q(\sbf)},t_{ Q(\sbf)})$ and $\rho\leq \kappa  R/10$.

We will now use a parabolic Calder\'on reproducing formula. We here follow the construction in  \cite[pp 227-228]{H} and we let $P_\lambda f:= k_\lambda * f$, where
\[k_\lambda(x,t):= \lambda^{-d} k(\lambda^{-1}x,\lambda^{-2}t)\,,\]
and where $k\in C_0^\infty(C'_1(0,0))$ is a non-negative
even function, 
with $$\iint_{\rn} k \d y\d s= 1.$$  Thus, $P_\lambda$ is a nice standard parabolic approximation of the identity.
Consequently,
\begin{align}\label{rep-}
I = -\int_0^\infty  \frac{\partial}{\partial \lambda} P_\lambda^2\, \d\lambda
=\int_0^\infty P_\lambda Q_\lambda\, \frac{\d\lambda}{\lambda},
\end{align}
in the strong operator topology on $\mathcal{B}(L^2(\rn))$,
where $Q_\lambda:= -2 \lambda  \frac{\partial P_\lambda}{\partial \lambda}$.  In particular, if we let $\phi_\lambda := -2 \lambda \frac{\partial k_\lambda}{\partial \lambda} $, then, as $k$ is even, we see that  \eqref{fi2} holds for $\phi_\lambda$, and by \eqref{rep-}
\begin{align}\label{rep}
\psi(z,\tau)=\int_0^\infty\big(k_\lambda*\phi_\lambda *\psi\big)(z,\tau)\frac {\d\lambda}\lambda\,,
\end{align}
in the sense of $L^2$ and with pointwise almost everywhere.

 Using the representation in \eqref{rep} we decompose $\psi$ as
  $$\psi=\psi_{1}+\psi_2,\ \psi_{1}=\psi_{11}+\psi_{12},$$
  where
  \begin{align*}
  \psi_{11}(z,\tau)&:= \,\int_R^\infty(k_\lambda*\phi_\lambda *\psi)(z,\tau)\frac {\d\lambda}\lambda= -\int_R^\infty  \frac{\partial}{\partial \lambda} P_\lambda^2\,\psi(z,\tau)\, \d\lambda,\notag\\
  \psi_{12}(z,\tau)&:= \int_0^R\iint_{\mathbb R^n\setminus C'_{200R}}
  k_\lambda(z-\hat z,\tau-\hat\tau)(\phi_\lambda *\psi)(\hat z,\hat \tau)\frac {\d\hat z\d\hat \tau\d\lambda}\lambda,\notag\\
  \psi_2(z,\tau)&:= \int_0^R\iint_{C'_{200R}}
  k_\lambda(z-\hat z,\tau-\hat\tau)(\phi_\lambda *\psi)(\hat z,\hat \tau)\frac {\d\hat z\d\hat \tau\d\lambda}\lambda,
  \end{align*}
  for $(z,\tau)\in\mathbb R^n$  and we again recall that $R=\diam(Q(\sbf))$.
 \begin{lemma}\label{decomp1} The function $\psi_1:\mathbb R^{n}\to \mathbb R$ satisfies
 \begin{eqnarray}\label{1.1+}
|\psi_1(x,t)-\psi_1(y,s)|\lesssim \delta  d_p(x,t,y,s),
\end{eqnarray}
whenever $(x,t), (y,s)\in C'_{100R}(x_{ Q(\sbf)},t_{ Q(\sbf)})$. Furthermore,
\begin{eqnarray}\label{1.1++}
|\nabla_x^2\psi_1(x,t)|+|\partial_t\psi_1(x,t)|\lesssim\delta R^{-1}\mbox{ on }C'_{100R}(x_{ Q(\sbf)},t_{ Q(\sbf)}).
\end{eqnarray}
  \end{lemma}
 \begin{proof} Note that if  $(x,t)\in C'_{100R}(x_{ Q(\sbf)},t_{ Q(\sbf)})$,
   then  $\psi_{12}(x,t)=0$ by construction and hence
   it suffices to treat $\psi_{11}$. Recall that by Lemma \ref{graphlip.lem},
  \begin{eqnarray}\label{1.1+a}
|\psi(x,t)-\psi(y,s)|\lesssim \delta d_p(x,t,y,s)
\end{eqnarray}
for all $(x,t),\ (y,s)\in \mathbb R^n$.   Observe that by construction,
\[\psi_{11} = \widetilde{P}_R  \psi = \widetilde{k}_R*\psi\,,
\]
where $\widetilde{k}_R:= k_R*k_R$, and $ \widetilde{P}_R:=P^2_R$.
Since $\widetilde{k}_R$ is non-negative with integral 1, we see that $|\psi_{11}(x,t)-\psi_{11}(y,s)|$ can be bounded by
\begin{align*}
&\lesssim
 \iint_{\rn}\widetilde{k}_R(z,\tau) \,|\psi(x-z,t-\tau)-\psi(y-z,s-\tau)|\, \d z\d\tau\\ 
&\lesssim\, \delta d_p(x,t,y,s)\,,
\end{align*}
by \eqref{1.1+a} and the translation invariance of the parabolic distance $d_p$. This proves \eqref{1.1+}.

To prove \eqref{1.1++}, we first observe that
\[ \left| \nabla_x^2 \,\widetilde{k}_R(x,t) \right|  \,+\, \left| \partial_t \, \widetilde{k}_R(x,t) \right| \,
\lesssim R^{-d-2}
\mathbbm{1}_{C'_{2R}(0,0)}(x,t)\,.
\]
Moreover,
$$\iint_{\rn} \nabla_x^2  \, \widetilde{k}_R(x,t)\, \d x\d t = 0
= \iint_{\rn} \partial_t  \, \widetilde{k}_R(x,t)\, \d x\d t.$$
Therefore $|\nabla_x^2\psi_{11}(x,t)| +|\partial_t\psi_{11}(x,t)|$ can be bounded by
  \begin{align*}
&\lesssim \, R^{-d-2}\iint_{\rn} \mathbbm{1}_{C'_{2R}(0,0)}(x-y,t-s)\, |\psi(y,s)-\psi(x,t)|\, \d x\d t\\
&\lesssim \, \delta  R^{-d-2}\iint_{\rn} \mathbbm{1}_{C'_{2R}(0,0)}(x-y,t-s)\, d_p(x,t,y,s)\,
 \d x \d t \lesssim\, \delta R^{-1},
  \end{align*}
 where in the next-to-last step we have again used \eqref{1.1+a}.\end{proof}

    \begin{lemma}\label{decomp2} The function $\psi_2:\mathbb R^{n}\to \mathbb R$ satisfies
\begin{eqnarray}\label{1.1+2}
|\psi_2(x,t)-\psi_2(y,s)|\lesssim \delta d_p(x,t,y,s),
\end{eqnarray}
whenever $(x,t),\ (y,s)\in C'_{100R}(x_{ Q(\sbf)},t_{ Q(\sbf)})$. Furthermore,
\begin{align}\label{decomp+}
  \iint_{\mathbb R^n}|\nabla_x\psi_2|^2+|\dhalf\psi_2|^2\, \d z\d\tau\lesssim \epsilon^2 \sigma(Q(\sbf)).
  \end{align}
  \end{lemma}
\begin{proof} \eqref{1.1+2} follows immediately from \eqref{1.1+} and \eqref{1.1+a}. To prove
\eqref{decomp+}, let  $f\in L^2(\mathbb R^n)$ and note that
  \begin{align*}
  &\iint_{\mathbb R^n}\partial_{x_i}\psi_2(z,\tau)\, f(z,\tau) \d z\d\tau = -
  \int_0^R\iint_{C'_{200R}}((\partial_{x_i}k)_\lambda*f) (\phi_\lambda *\psi)\frac {\d\hat z\d\hat \tau\d\lambda}{\lambda^2}.
  \end{align*}
  Hence,
      \begin{align*}
      \iint_{\mathbb R^n}|\nabla_x\psi_2|^2\, \d z\d\tau\lesssim \int_0^R\iint_{C'_{200R}}|\phi_\lambda *\psi|^2\frac {\d\hat z\d\hat \tau\d\lambda}{\lambda^3}
      \lesssim \epsilon^2 \sigma(Q(\sbf)),
  \end{align*}
  by Littlewood-Paley theory  and  \eqref{2.34} (provided we ensure $\kappa\geq 2000$ so that $\kappa R/10 \ge 200R$). Similarly,
  \begin{align*}
  &\iint_{\mathbb R^n}\dhalf\psi_2(z,\tau)\, f(z,\tau) \d z\d\tau
  =\int_0^R\iint_{C'_{200R}}((\dhalf k)_\lambda*f) (\phi_\lambda *\psi)\frac {\d\hat z\d\hat \tau\d\lambda}{\lambda^2},
  \end{align*}
  and again
  \begin{align*}
      \iint_{\mathbb R^n}|\dhalf\psi_2|^2\, \d z\d\tau\lesssim \int_0^R\iint_{C'_{200R}}|\phi_\lambda *\psi|^2\frac {\d\hat z\d\hat \tau\d\lambda}{\lambda^3}
      \lesssim \epsilon^2 \sigma(Q(\sbf))\,,
  \end{align*}
where we have used that
    \begin{align*}
\int_0^\infty\iint_{\mathbb R^n}|(\dhalf k)_\lambda*f(\hat z,\hat\tau)|^2 \frac {\d\hat z\d\hat \tau\d\lambda}{\lambda}\lesssim \iint_{\mathbb R^n}|f(\hat z,\hat\tau)|^2 {\d\hat z\d\hat \tau}.
  \end{align*}
In turn, the latter estimate follows readily via Plancherel's theorem in $\rn$, using that
 the half time derivative corresponds to the Fourier multiplier $|\tau|^{\frac{1}{2}}$ as follows.

 First, the equivalence between our notion of half-time derivative and the Fourier multiplier one can be found, for instance, in \cite{DPV} (see in particular the beginning of Section 3 and Proposition 3.3 there).
 
Recall that $k\in C_0^\infty(C'_1(0,0))$. Taking the Fourier transform in $(\hat{z},\hat{\tau})$ and using Plancherel's formula we have, using $(\xi,\zeta)\in \RR^{n-1}\times \RR$ for the dual variables, 
\begin{equation*}
	\begin{split}
		\iint_{\RR^n} |(D_t^{1/2}k)_\lambda * f(\hat{z}, \hat{\tau})|^2 \, \d\hat{z} \d\hat{\tau} & = c_n \iint_{\rn} |\lambda^2\zeta||\hat{k}(\lambda \xi, \lambda^2 \zeta)\hat{f}(\xi,\zeta)|^2\, \d \xi \d\zeta \\
		& \lesssim \iint_{\rn}\min\{ \lambda^2|\zeta|,(\lambda^{2}|\zeta|)^{-2}\} |\hat{f}(\xi,\zeta)|^2\, \d \xi\d\zeta,
	\end{split}
\end{equation*}
where we used the fact that $\hat{k}$ is a Schwartz function in the last inequality; more specifically we used the bounds
$$
|\hat{k}(\xi,\zeta)|\lesssim \min\{ 1, |\zeta|^{-3/2}\}, \qquad (\xi,\zeta)\in \rn.
$$
Using the previous estimate and integrating in $\lambda$ we arrive at 
\begin{equation*}
	\begin{split}
		\int_0^\infty\iint_{\mathbb R^n}|(\dhalf k)_\lambda&*f(\hat z,\hat\tau)|^2 \frac {\d\hat z\d\hat \tau\d\lambda}{\lambda}  \\
		&\lesssim \iint_{\rn}|\hat{f}(\xi,\zeta)|^2 \left(\int_0^\infty \min\{ \lambda^2|\zeta|,(\lambda^{2}|\zeta|)^{-2}\}\,\dfrac{\d\lambda}{\lambda}\right)\, \d\xi\d\zeta\\
		& = \iint_{\rn} |\hat{f}(\xi,\zeta)|^2\, \d\xi\d\zeta.
	\end{split}
\end{equation*}

\end{proof}

  We next prove two lemmas concerning the oscillations of $\psi_2$ and $\psi$.

  \begin{lemma}\label{l1} Consider $C'_r(z,\tau)\subset\mathbb R^n$, let
  $$m_{C'_r(z,\tau)}\psi_2:=\bariint_{C'_r(z,\tau)}\psi_2\d\hat z\d\hat\tau,$$
  and
   $$\osc_{C'_r(z,\tau)}\psi_2:=\sup_{(\hat z,\hat\tau)\in C'_r(z,\tau)}\biggl |\psi_2(\hat z,\hat \tau)-m_{C'_r(z,\tau)}\psi_2 \biggr |.$$
  Assume that $C'_r(z,\tau)\subset C'_{100R}(x_{ Q(\sbf)},t_{ Q(\sbf)})$. Then
  $$\osc_{C'_r(z,\tau)}\psi_2\lesssim r\biggl \{r^{-1}\bariint_{C'_r(z,\tau)}\biggl |\psi_2(\hat z,\hat \tau)-m_{C'_r(z,\tau)}\psi_2\biggr |\d\hat z\d\hat\tau\biggr \}^{1/(d+1)}\delta^{d/(d+1)}.$$
  \end{lemma}
  \begin{proof} Let $(\hat z,\hat\tau)\in \overline{C'_r(z,\tau)}$ be such that
 $$\biggl |\psi_2(\hat z,\hat \tau)-m_{C'_r(z,\tau)}\psi_2 \biggr |=\osc_{C'_r(z,\tau)}\psi_2=:\mu$$
  Using Lemma \ref{decomp1} and Lemma \ref{decomp2} we note that
  \begin{align}\label{lip}
 |\psi_2(\hat z,\hat \tau)- \psi_2(\tilde z,\tilde \tau)|\leq
  \beta\delta  d_p(\hat z,\hat \tau,\tilde z,\tilde \tau),
  \end{align}
   whenever $(\hat z,\hat\tau),\ (\tilde z,\tilde \tau)\in \overline{C'_r(z,\tau)}$,
   and for every function  $\tilde\psi$ in the set $\{\psi,\psi_1,\psi_2\}$. Here $\beta$ is the implicit constant in \eqref{1.1+2}.
   If $(\tilde z,\tilde \tau)\in C'_r(z,\tau)$ is such that
   $d_p(\hat z,\hat \tau,\tilde z,\tilde \tau)\leq \mu/(2\beta\delta )$, then
    \begin{align}\label{lip+}|\psi_2(\hat z,\hat \tau)- \psi_2(\tilde z,\tilde \tau)|\leq\mu/2,
      \end{align}
      by \eqref{lip} and hence
   \begin{align}\label{lip++}\biggl |\psi_2(\tilde  z,\tilde \tau)-m_{C'_r(z,\tau)}\psi_2 \biggr |\geq \mu/2,
   \end{align}
   by the definition of $\mu$ and $(\hat z,\hat\tau)$.

   Assume that
   $\mu/(2\beta\delta)\leq r$. Then, using \eqref{lip++} we see that
  \begin{equation*}
 \iint_{C'_r(z,\tau)}\biggl |\psi_2(\tilde  z,\tilde \tau)-m_{C'_r(z,\tau)}\psi_2
 \biggr |\d\tilde z\d\tilde \tau\geq c \mu (\mu/(\beta\delta))^{d}
   = c \mu^{d+1}\beta ^{-d}\delta^{-d},
   \end{equation*}
   for a constant $c=c(n)$. Hence, in this case, with implicit constants depending only on $n$, we have
\begin{align}\label{lip+++}\mu^{d+1}&\lesssim
 \beta^{d}\delta^{d}\iint_{C'_r(z,\tau)}\biggl |\psi_2(\tilde  z,\tilde \tau)-m_{C'_r(z,\tau)}\psi_2 \biggr |
   \d\tilde z\d\tilde\tau\notag\\
   &\approx \beta^{d}\delta^{d}r^{d+1}r^{-1}
   \bariint_{C'_r(z,\tau)}\biggl |\psi_2(\tilde  z,\tilde \tau)-m_{C'_r(z,\tau)}\psi_2 \biggr |
   \d\tilde z\d\tilde\tau.\end{align}
   This completes the proof of the lemma in this case.

   Assume that $\mu/(2\beta\delta)\geq r$. In this case, using \eqref{lip} and \eqref{lip+} we see that
     $$\biggl |\psi_2(\tilde  z,\tilde \tau)-m_{C'_r(z,\tau)}\psi_2 \biggr |\geq \mu/2$$
     on an ample portion of $C'_r(z,\tau)$ and hence, for $c=c(n)$, we have
     \begin{align}\label{lip++++}\bariint_{C'_r(z,\tau)}\biggl |\psi_2(\tilde  z,\tilde \tau)-m_{C'_r(z,\tau)}\psi_2\biggr |\d\tilde z\d\tilde\tau\geq c^{-1}\mu.
     \end{align}
     Moreover, using \eqref{lip}, 
     we also have
      \begin{align}\label{lip+++++}r^{-1}
      \bariint_{C'_r(z,\tau)}\biggl |\psi_2(\tilde  z,\tilde \tau)-m_{C'_r(z,\tau)}\psi_2
      \biggr |\d\tilde z\d\tilde\tau\lesssim \beta\delta.
      \end{align}
      In particular, combining \eqref{lip++++} and \eqref{lip+++++} we deduce
      \begin{align}
     \mu&\lesssim\, r  \biggl \{
     r^{-1}\bariint_{C'_r(z,\tau)}
     \biggl |\psi_2(\tilde  z,\tilde \tau)-m_{C'_r(z,\tau)}\psi_2\biggr |
     \d\tilde z\d\tilde\tau \biggr \} 
     \notag\\
       &\lesssim\, r\biggl \{r^{-1}\bariint_{C'_r(z,\tau)}\biggl |\psi_2(\hat z,\hat \tau)-m_{C'_r(z,\tau)}\psi_2\biggr |
       \d\tilde z\d\tilde\tau\biggr \}^{1/(d+1)}(\beta\delta)^{d/(d+1)}\,,
       \end{align}
       where the various implicit constants are purely dimensional, i.e. depend only on $n$. The conclusion of the lemma follows also in this case.
  \end{proof}

    \begin{lemma}\label{l2} Let
    $$N(\psi_2)(\tilde z,\tilde \tau):= \sup_{C'_\rho(z,\tau)\ni (\tilde z,\tilde\tau)}
    \biggl \{\rho^{-1}\bariint_{C'_\rho(z,\tau)}\biggl |\psi_2(\hat z,\hat \tau)-m_{C'_\rho(z,\tau)}\psi_2\biggr |\d\hat z\d\hat\tau\biggr \}.$$
    Let $\theta\in (0,1)$  be a degree of freedom and let $$\ti F:=\{(\tilde z,\tilde \tau)\in C'_{50R}(x_{ Q(\sbf)},t_{ Q(\sbf)}):\ N(\psi_2)(\tilde z,\tilde \tau)\leq\theta^{d+1}\delta\}.$$
    Let $\eta\in (0,1)$, consider $(z_0,\tau_0)\in \mathbb R^n$ and assume that
    $\ti F\cap C'_r(z_0,\tau_0)\neq\emptyset$ for some $r\leq \eta R$. Then
       \begin{equation}\label{semi-taylor}
   \sup_{(z,\tau)\in C'_r(z_0,\tau_0)}|\psi(z,\tau)-\psi(z_0,\tau_0)-\nabla_z\psi_1(z_0,\tau_0)\cdot (z-z_0)|\lesssim (\theta+\eta)\delta r.
   \end{equation}
  \end{lemma}
    \begin{proof} Let $(z,\tau)\in C'_r(z_0,\tau_0)$. Then
        \begin{multline*}
        |\psi(z,\tau)-\psi(z_0,\tau_0)-\nabla_z\psi_1(z_0,\tau_0)\cdot (z-z_0)| \\[4pt]
     \leq\,   |\psi_2(z,\tau)-\psi_2(z_0,\tau_0)|\,\,
        +\,\,|\psi_1(z,\tau)-\psi_1(z_0,\tau_0)-\nabla_z\psi_1(z_0,\tau_0)\cdot (z-z_0)|\\[4pt]
   \leq \,\,2\osc_{C'_r(z_0,\tau_0)}\psi_2
  \, \,+\,\, |\psi_1(z,\tau)-\psi_1(z_0,\tau_0)-\nabla_z\psi_1(z_0,\tau_0)\cdot (z-z_0)|.
    \end{multline*}
    Using Lemma \ref{decomp1}, and in particular \eqref{1.1++}), and Taylor's Theorem, we see that
        \begin{equation}\label{decomp1a}
        |\psi_1(z,\tau)-\psi_1(z_0,\tau_0)-\nabla_z\psi_1(z_0,\tau_0)\cdot (z-z_0)|
              \, \lesssim \,\delta R^{-1}r^2 \,\lesssim \, \delta\eta r\,,
    \end{equation}
    by our assumption that $r\leq \eta R$. Furthermore, as $\ti F\cap C'_r(z_0,\tau_0)\neq\emptyset$ we can apply Lemma \ref{l1} to conclude that
    $$\osc_{C'_r(z_0,\tau_0)}\psi_2\lesssim \theta\delta r.$$
Combining the latter estimate with \eqref{decomp1a}, we obtain \eqref{semi-taylor}.
 \end{proof}

      \begin{lemma}\label{l3}  Let $N(\psi_2)$, $\theta$ and $\ti F$ be defined as in Lemma \ref{l2}.
  Then
  $$\H^{n+1}_p\left(C'_{50R}(x_{ Q(\sbf)},t_{ Q(\sbf)})\setminus \ti F\right)\lesssim \delta^{-2}\theta^{-2(d+1)}\epsilon^2\sigma(Q(\sbf)).$$
  \end{lemma}
    \begin{proof} Note that
    \begin{align} \H^{n+1}_p( C'_{50R}(x_{ Q(\sbf)},t_{ Q(\sbf)})\setminus \ti F)\leq \delta^{-2}\theta^{-2(d+1)}\iint_{\mathbb R^n} |N(\psi_2)(\tilde z,\tilde \tau)|^2\, \d \tilde z\d \tilde\tau.
    \end{align}
    To control $N(\psi_2)(\tilde z,\tilde t)$ we need to use a Poincar{\'e} inequality. To this end,
    for locally integrable $g: \mathbb R^{n} \to\mathbb R $, we let $\Max (g) $ be the
    $n$-dimensional parabolic Hardy-Littlewood maximal function
\begin{eqnarray*}
\Max (g) ( x,t) := \sup_{\varrho > 0} \bariint_{C'_\varrho ( x, t)} \, |g|  \d y\d s,
\end{eqnarray*}
and we let $\Max_x$ and $\Max_t$ denote (respectively)
the standard Hardy-Littlewood maximal operators in the $x$ and $t$ variables only. Consider $(\tilde z,\tilde\tau)\in\mathbb R^n$ and assume that $(\tilde z,\tilde\tau)\in C'_r(z,\tau)$. Then $C'_r(z,\tau)\subset C'_{2r}(\tilde z,\tilde\tau)$ and
\begin{align*}
&r^{-1}\bariint_{C'_r(z,\tau)}\biggl |\psi_2(\hat z,\hat \tau)-m_{C'_r(z,\tau)}\psi_2\biggr |\d\hat z\d\hat\tau\\
&\lesssim\,
r^{-1}\bariint_{C'_{2r}(\tilde z,\tilde\tau)}\biggl |\psi_2(\hat z,\hat \tau)-m_{C_{2r}(\tilde z,\tilde\tau)}\psi_2\biggr |\d\hat z\d\hat\tau.
\end{align*}
Using Lemma 2.2 in \cite{AEN} we see that
\begin{align*}
&r^{-1}\bariint_{C'_{2r}(\tilde z,\tilde\tau)}\biggl |\psi_2(\hat z,\hat \tau)-m_{C_{2r}(\tilde z,\tilde\tau)}\psi_2\biggr |\d\hat z\d\hat\tau\\
 &\lesssim \, \Max (|\nabla_{x}\psi_2|)(\tilde z,\tilde\tau)+ \Max_x \Max_t (|\HT \dhalf \psi_2|)(\tilde z,\tilde\tau).
\end{align*}
In particular,
\begin{eqnarray*}
N(\psi_2)(\tilde z,\tilde \tau) \lesssim \Max (|\nabla_{x}\psi_2|)(\tilde z,\tilde \tau)+ \Max_x \Max_t (|\HT \dhalf \psi_2|)(\tilde z,\tilde \tau).
\end{eqnarray*}
Hence,
\begin{align*} \H^{n+1}_p(C'_{50R}(x_{ Q(\sbf)},t_{ Q(\sbf)})\setminus \ti F)&\lesssim \delta^{-2} \theta^{-2(d+1)} \iint_{\mathbb R^n}|\nabla_x\psi_2|^2+|\dhalf\psi_2|^2\, \d z\d\tau\notag\\
&\lesssim  \delta^{-2} \theta^{-2(d+1)}\epsilon^2 \sigma(Q(\sbf)),
    \end{align*}
    by Lemma \ref{decomp2}. \end{proof}

Let $N(\psi_2)$, $\theta$ and $\ti F$ be defined as in Lemma \ref{l2}. Let $\eta\in (0,1)$, consider $(z_0,\tau_0)\in \mathbb R^n$, and assume that
    $\ti F\cap C'_r(z_0,\tau_0)\neq\emptyset$ for some $r\leq \eta R$. By Lemma \ref{l2} we have
    $$\sup_{(z,\tau)\in C'_r(z_0,\tau_0)}|\psi(z,\tau)-\psi(z_0,\tau_0)-\nabla_z\psi_1(z_0,\tau_0)\cdot (z-z_0)|\lesssim (\theta+\eta)\delta r.$$
    Letting  $\theta:=\eta$ we can conclude that
    \begin{align}\label{geo1}
    \sup_{(z,\tau)\in C'_r(z_0,\tau_0)}
    |\psi(z,\tau)-\psi(z_0,\tau_0)-\nabla_z\psi_1(z_0,\tau_0)\cdot (z-z_0)|\lesssim \eta \delta r.
    \end{align}
    Let $\tilde P=\tilde P_{C'_r(z_0,\tau_0)}$ be the time-independent plane which is the graph of the affine function
  $$L(z):= \psi(z_0,\tau_0)-\nabla_z\psi_1(z_0,\tau_0)\cdot (z-z_0).$$
  Then by  \eqref{geo1} we have
  \begin{align}\label{geo2}\sup_{(Z,\tau)\in\tilde\Sigma\cap\pi^{-1}(C'_r(z_0,\tau_0))}r^{-1}d_p(Z,\tau,\tilde P)\lesssim\eta\delta.
   \end{align}
Recall that the we are seeking a contradiction to the assumption that \eqref{pack++} holds for some $\sbf \in\mathcal{F}_1$. We will  achieve this by proving that the conclusion in \eqref{geo2} is incompatible  with $\sbf\in\mathcal{F}_1$.  Next, we prove two
  additional auxiliary lemmas to be used in the final argument.

   \begin{lemma}\label{lem3} Let $\eta$, $0<\eta\ll 1$ and $A\gg 1$ be given. Then there exists
   $\epsilon_0=\epsilon_0(n,M,\eta,A)>0$
   such that if $\epsilon\leq \epsilon_0$, then the angle between $P_Q$ and $P_{Q(\sbf)}$ is $\leq \delta/100$ for all $Q\in \sbf$ with $\diam (Q)\geq \eta R/A$.
  \end{lemma}
  \begin{proof} Consider $Q\in \sbf$, let $Q_0=Q$ and let $Q_1,...,Q_M$ denote the successive ancestors of $Q$ with $Q_M=Q(\sbf)$. By definition, see \eqref{3++}, we have
  \begin{equation}\label{3++a}
 d_p(Y,s,P_{Q_i})\leq\epsilon\diam(Q_i)\mbox{ for all }(Y,s)\in 8KQ_i\mbox{ and for all $i\in\{0,...,M\}$}.
\end{equation}
This implies, see Lemma \ref{aux2}, that the angle between $P_{Q_i}$ and $P_{Q_{i+1}}$ is bounded by $\lesssim \epsilon$. In particular,
the angle between $P_{Q}$ and $P_{Q(\sbf)}$ is bounded by
$$\lesssim\epsilon M\lesssim \epsilon\log(2R/\diam(Q)),$$
from which the result follows.\end{proof}

   \begin{lemma}\label{lem4}  If $Q\in m_1(\sbf)$, then $d_p(\ti F,\pi(Q))>\diam(Q)$.
  \end{lemma}
  \begin{proof} Assume, for the sake of obtaining a contradiction, that there exists $Q\in m_1(\sbf)$ such that  $d_p(\ti F,\pi(Q))\leq \diam(Q)$. Let
  $(X_Q,t_Q)$ be the center of $Q$, let $(z_0,\tau_0):=(x_Q,t_Q):=\pi(X_Q,t_Q)$, $r=10\diam(Q)$.  Consider $C'_r(z_0,\tau_0)$. By the assumption and the construction  $\ti F\cap C'_r(z_0,\tau_0)\neq \emptyset$. Recall also that $m_1(\sbf)$
denotes the cubes in $m(\sbf)$ for which the angle between $P_Q$  and $P_{ Q(\sbf)}$ is greater or equal to $\delta/2$. Hence, by Lemma \ref{lem3} we can conclude that $r\leq \eta R$ if $\epsilon$ sufficiently small. Furthermore, $C'_r(z_0,\tau_0)\subseteq  C'_{100R}(x_{ Q(\sbf)},t_{ Q(\sbf)})$.
By Lemma \ref{closegraph.lem}, 
for $(Z,\tau)\in Q$ we have
  \begin{align}\label{geo3}d_p(Z,\tau, (z,\psi(z,\tau),\tau)))\lesssim \epsilon d(z,\tau)\lesssim \epsilon \diam(Q),\ (z,\tau)=\pi(Z,\tau).
   \end{align}
   Furthermore, by \eqref{geo2} we also know that
   \begin{align*}\sup_{(Z,\tau)\in\tilde\Sigma\cap\pi^{-1}(C'_r(z_0,\tau_0))}d_p(Z,\tau,\tilde P)\lesssim\eta\delta r,
   \end{align*}
   for a time independent plane $\tilde P \in \mathcal{P}$. The inequalities in the last two displays imply
   that if $(Z,\tau)\in Q$, and if $\epsilon\lesssim \eta \delta$, then
   \begin{align}\label{geo4} d_p(Z,\tau,\tilde P)\lesssim\eta\delta r.
    \end{align}
  Upon observing that we can produce $n+1$ points of $Q$ with ample distance from each other using Lemma \ref{aux1}, the estimate \eqref{geo4}  implies that the angle between
    $\tilde P$ and $P_Q$ is $\lesssim\eta\delta<\delta/100$ if we choose $\eta$ and $\epsilon$ small.
    Let $Q^\ast$ be the largest ancestor to $Q$ such that $10\diam(Q^\ast)\leq \eta R$. The same argument as above then gives that
  the angle between $\tilde P$ and $P_{Q^\ast}$ is $\leq\delta/100$.
  In particular, the angle between $P_{Q^\ast}$ and $P_Q$  must be $\leq\delta/50$.
 On the other hand, we may apply Lemma \ref{lem3} to $Q^*$ to deduce that
the angle between $P_{Q^*}$ and $P_{Q(\sbf)}$ is at most
$\delta/100$, and thus that the angle between $P_Q$ and $P_{Q(\sbf)}$ is at most
$\delta/33$.
  This is incompatible with $Q\in m_1(\sbf)$ and hence we have reached a contradiction.\end{proof}

We are now ready finally to complete the proof of Theorem \ref{main} by proving that $\sbf\notin\mathcal{F}_1$.  We see that by Lemma \ref{lem4} it suffices to prove that \eqref{pack++} implies that there is at least one $Q\in m_1(\sbf)$ with $d_p(\ti F,\pi(Q))\leq\diam(Q)$ where $\ti F$ is defined as in Lemma \ref{l2}.  Let
   \begin{align}\label{geo5}\mathbf{F}={\bigcup_{Q\in m_1(\sbf)}Q.}
   \end{align}
  There exists a collection $\mathbf{S}'\subseteq m_1(\sbf)$ such that
   $\mathbf{F}$ can be covered by $$\{C_{20\diam(Q)}(Z_Q,\tau_Q)\}_{Q\in \mathbf{S}'}$$ in such a way that
   $$\{C_{3\diam(Q)}(Z_Q,\tau_Q)\}_{Q\in \mathbf{S}'}$$
   is a disjoint collection. Hence,
      \begin{align}\label{geo6-}\sigma(\mathbf{F})\lesssim
      \sigma{\left(\bigcup_{Q\in \mathbf{S}'}C_{20\diam(Q)}(Z_Q,\tau_Q) \cap \Sigma\right)}\lesssim
      \sum_{Q\in \mathbf{S}'} (\diam(Q))^d.
   \end{align}
   Using this fact, and our usual notation
$(z_Q,\tau_Q)=\pi(Z_Q,\tau_Q)$,  we can use by proof of Lemma \ref{lem1} {(specifically,
the fact that \eqref{6} holds whenever $(X,t),\ (Y,s)\in 16 \kappa Q(\sbf)$ and
$\min\{d(X,t),d(Y,s)\} \le 1000 d_p(X,t, Y,s)$)}, to conclude that the collection
$$\{C'_{\diam(Q)}(z_Q,\tau_Q)\}_{Q\in \mathbf{S}'}$$
is also pairwise disjoint. Hence
   \begin{align}\label{geo6}\H^{n+1}_p{\left(\bigcup_{Q\in \mathbf{S}'}C'_{\diam(Q)}(z_Q,\tau_Q)\right)}\gtrsim
   \sigma(\mathbf{F})\,.
   \end{align}
Moreover,  every set in $\{C'_{\diam(Q)}(z_Q,\tau_Q)\}_{Q\in \mathbf{S}'}$ is contained in $C'_{50R}(x_{ Q(\sbf)},t_{ Q(\sbf)})$ and by Lemma \ref{lem4},
none of the sets in the collection
   $\{C'_{\diam(Q)}(z_Q,\tau_Q)\}_{Q\in \mathbf{S}'}$ intersect $\ti F$. Hence, using this observation
   and \eqref{geo6}, we see that
   \begin{align}\label{geo7}\sigma(\mathbf{F})\lesssim \H^{n+1}_p{\left(\bigcup_{Q\in \mathbf{F}'}C'_{\diam(Q)}(z_Q,\tau_Q)\right)}\lesssim \H^{n+1}_p(C'_{50R}(x_{ Q(\sbf)},t_{ Q(\sbf)})\setminus \ti F ).
   \end{align}
Using Lemma \ref{l3} with $\theta=\eta$ we have
    \begin{align}\label{geo8}\sigma(\mathbf{F})\lesssim
    \delta^{-2}\theta^{-2(d+1)}\epsilon^2\sigma(Q(\sbf))\approx
    \delta^{-2}\eta^{-2(d+1)}\epsilon^2\sigma(Q(\sbf))\,.
   \end{align}
   Hence,
   \begin{align}\label{geo8+}\sigma(\mathbf{F})\leq c\delta^{-2}\eta^{-2(d+1)}\epsilon^2\sigma(Q(\sbf)),
   \end{align}
   for some $c=c(n,M,\Gamma,\kappa,K)\geq 1$. Given $\delta$ and $\eta$, we let $\epsilon_1:=\delta\eta^{d+1}/\sqrt{4c}$. Let also $\epsilon_0=\epsilon_0(n,M,\eta,A)>0$, with $A\gg 1$ fixed, be as in the statement of Lemma \ref{lem3}. We now choose
       \begin{align}\label{geo9-}
       \epsilon:=\min\{\epsilon_0,\epsilon_1\}.
   \end{align}
    Then $\epsilon=\epsilon(n,M,\Gamma,\kappa,K,\delta,\eta,A)$, and in fact, as $\eta$ and $A$ can be chosen as  constants with no particular dependencies, $\epsilon$ can be seen simply as a function of $n,M,\Gamma,\kappa,K$ and $\delta$.   By the choice for $\epsilon$,
       \begin{align}\label{geo9}\sigma(\mathbf{F})\leq \sigma(Q(\sbf))/4.
   \end{align}
    The estimate \eqref{geo9} contradicts the statement that $\sbf \in \mathcal{F}_1$. As we had reduced matters to showing $\sbf \not \in \mathcal{F}_1$, we have proved Theorem \ref{main}.

\section{The proof of Theorem \ref{main.bil.thrm}}\label{s5}
In this section we prove Theorem \ref{main.bil.thrm} by improving Theorem \ref{main}(4) to give a
bilateral approximation as in Theorem \ref{main.bil.thrm}(4).   The latter is a parabolic version
of a result appearing in \cite{HMM1}.
The most important observation here is that in the proof of Theorem \ref{main} the condition \eqref{dico} (B), for the stopping time regimes, was {only} used to pack the maximal cubes of the stopping time regimes (Section \ref{packsttime}). In particular, the construction of the graphs requires only
property \eqref{dico} (A) of the stopping time regimes (and the fact that $\Sigma$ is p-UR). With this in mind, we are going to use the packing of the stopping time regimes in Theorem \ref{main} to prove the packing of stopping time regimes where \eqref{dico} (A) still holds, along with a stronger, {\it bilateral} approximation by a $t$-independent plane. Then we can run the machine that produces the regular Lip(1,1/2) graphs and never carry out the arguments of Section \ref{packsttime} (since the stopping time regimes will already pack) and we will use this bilateral approximation by planes to improve the approximation by graphs to a bilateral one.

\begin{definition}\label{PBWGL}
Let $\Sigma$ be a parabolic Ahlfors-David Regular set with dyadic grid $\mathbb{D}(\Sigma)$. For $Q \in \mathbb{D}(\Sigma)$ and $K > 1$ we let $b\beta_\infty(Q):= b\beta_\infty(Q,K)$ denote the number
\begin{equation*}
 \inf_{P \in \mathcal{P}} \diam(Q)^{-1} 
 \left[ \sup_{\{(Y,s)\in 8KQ\}} d_p(Y,s,P) \,\,\,+ 
 \sup_{\{(X,t)\in C_{8K\diam(Q)} \cap P \}} d_p(X,t,\Sigma)  \right].
\end{equation*}
Given $\epsilon > 0$, $K > 1$, we say that the $(\epsilon, K)$-parabolic bilateral weak geometric lemma ($(\epsilon, K)$-PBWGL) holds for $\Sigma$, if there exists a constant $c(\epsilon,K)$ such that
\[\sum_{\substack{b\beta_\infty(Q,K)  \ge \epsilon \\ Q \subseteq Q^*}} \sigma(Q) \le c(\epsilon,K) \sigma(Q^*), \quad \forall Q^* \in \mathbb{D}(\Sigma). \]
\end{definition}

Combining Theorem \ref{main} with our previous work \cite{BHHLN1} we immediately obtain the following\footnote{The notation GPG in \cite{BHHLN1} means `good parabolic graph', that is, a regular Lip(1,1/2) graph.}.
\begin{theorem}[{\cite[Theorem 4.15]{BHHLN1}}]\label{URimpPBWGL}
Let $\Sigma$ be a parabolic Ahlfors-David regular set. If $\Sigma$ is parabolic UR with constants $M$ and $\Gamma$, then $\Sigma$ satisfies the $(\epsilon, K)$-PBWGL for every $\epsilon \in (0,1]$ and $K > 1$.
\end{theorem}
The reader may notice that the $\beta$ numbers used here are defined slightly differently compared to \cite{BHHLN1}, and that the (dilation) parameter $K$ does not appear in Theorems 4.15 \cite{BHHLN1}. However, these differences represent no problems,  see \cite[Remark 2.24]{BHHLN1}. In particular, it turns out that $(\epsilon, K)$-PBWGL for some fixed $K$ and all $\epsilon> 0 $ implies that the $(\epsilon, K)$-PBWGL for all $K$ and all $\epsilon> 0$, so that we may use the result of \cite{BHHLN1} to first prove the $(\epsilon, 2)$-PBWGL for all $\epsilon> 0$ and obtain $(\epsilon, K)$-PBWGL for all $K$ and all $\epsilon> 0$. This improvement of $K$ be proved with techniques in \cite[I. 3.2]{DS2} specifically using Lemma 3.27 in part I.

The following lemma will be important in the proof of Theorem \ref{main.bil.thrm}.
\begin{lemma}\label{BWGLcorona.lem}
Suppose $\Sigma$ is a parabolic UR set with parameters $M$ and $\Gamma$. Let $\epsilon \ll\delta$ and $K \gg\kappa$ be as in the proof of Theorem \ref{main}.
If we define
\begin{align*}
\mathcal{G}^*&: = \mathcal{G}^*(\epsilon, K) := \{Q \in \mathbb{D}(\Sigma): b\beta_\infty(Q,K) < \epsilon\},\\
\mathcal{B}^*&: = \mathcal{B}^*(\epsilon, K) = \{Q \in \mathbb{D}(\Sigma): b\beta_\infty(Q,K) \ge \epsilon\},
\end{align*}
then $\mathbb{D}(\Sigma) = \mathcal{G}^* \cup \mathcal{B}^*$, and the following hold.
\begin{enumerate}
\item The cubes $\mathcal{G}^*$ can be decomposed into disjoint, stopping time regimes $\{\sbf^*\}_{\sbf^* \in \mathcal{F}^*}$, such that each $\sbf^*$ is coherent and
\[\Ang(P_Q, P_{Q(\sbf^*)}) \le 4 \delta, \quad \forall Q \in \sbf^*,\]
where $P_Q \in \mathcal{P}$ is the plane minimizing the quantity in the definition of $b\beta_\infty(Q,K)$.
\item With the decomposition in (1), the `bad' cubes $\mathcal{B}^*$ and the maximal cubes $\{Q(\sbf^*)\}_{\sbf^* \in \mathcal{F}^*}$ satisfy a packing condition
\[\sum_{\substack{Q \in \mathcal{B}^* \\ Q \subseteq R}} \sigma(Q) + \sum_{\sbf^*: \substack{Q(\sbf^*)  \subseteq Q^*}} \sigma(Q(\sbf))) \leq c(\epsilon,K) \sigma(Q^*), \quad \forall Q^* \in \mathbb{D}(\Sigma).\]
\end{enumerate}
\end{lemma}
\begin{proof} The packing for the cubes $\mathcal{B}^*$ is a consequence of Lemma \ref{URimpPBWGL}.
By the proof of Theorem \ref{main} the decomposition above holds with $\mathcal{G}^*(\epsilon,K)$ replaced by
\[\mathcal{G} = \mathcal{G}(\epsilon,K)= \{Q \in \mathbb{D}(\Sigma): \beta_\infty(Q,K) < \epsilon\},  \] and with $\mathcal{B}^*$ replaced by
the complement of $\mathcal{G}$ in $\mathbb{D}(\Sigma)$.  In particular, there exists a decomposition of $\mathcal{G}$ into stopping time regimes $\{\sbf\}_{\sbf \in \mathcal{F}}$ such that
\[\Ang(P'_Q, P'_{Q(\sbf)}) \le \delta, \quad \forall Q \in \sbf,\]
where $P'_Q$ is the plane that minimizes $\beta_\infty(Q,K)$. Now this implies that
\[\Ang(P'_Q, P'_R) \le 2\delta, \quad \forall Q,R \in \sbf.\]
Moreover, since $\epsilon \ll \delta$, Lemma \ref{aux2} implies that
\[\Ang(P_Q, P_R) \le \delta, \quad \forall Q,R \in \sbf \cap \mathcal{G}^*.\]
Now to prove the lemma, we follow the most natural course of action and simply check that for each $\sbf$, we can decompose $\sbf \cap \mathcal{G}^*$ into stopping time regimes $\{\sbf^*\}_{\mathcal{F}^*_\sbf}$ in a way that maintains the packing condition. This can be done exactly as in \cite[Lemma 2.2]{HMM1}. We describe the process of partitioning $\sbf \cap \mathcal{G}^*$, and we leave the verification\footnote{This verification is carried out in detail in \cite{HMM1}.} of the packing condition to the interested reader. Given $\sbf \in \mathcal{F}$, we let $\mathcal{M}_\sbf$ be the set of $Q \in \sbf \cap \mathcal{G}^*$ for which either $Q = Q(\sbf)$ or else the dyadic parent of $Q$ or one of its siblings belongs to $\mathcal{B}^*$. For each $Q \in \mathcal{M}_\sbf$ we form a new stopping time regime $\sbf^* = \sbf^*(Q)$ as follows. Set $Q(\sbf^*) = Q$ and `subdivide' $Q(\sbf^*)$ `dyadically' until we reach a subcube $Q'$ such that $Q' \not\in \sbf$, or else $Q'$ or one of its siblings belongs to $\mathcal{B}^*$. In any such scenario, $Q'$ and all of its siblings are omitted from $\sbf^*$ and $Q'$ is a minimal cube of $\sbf^*$.  It then follows that $\sbf^* \subset \sbf \cap \mathcal{G}^*$ and that $\sbf^*$ is coherent. Moreover,
\[\sbf \cap \mathcal{G}^* =  {\bigcup_{Q \in \mathcal{M}_\sbf}} \sbf^*(Q),\]
and since $\mathcal{G} = \cup_{\sbf \in \mathcal{F}} \sbf$ and $\mathcal{G}^* \subseteq \mathcal{G}$, it holds that
\[\mathcal{G}^* = {\bigcup_{\mathcal{F}^*}}\, \sbf^*,\]
where $\mathcal{F}^* = \{\sbf^*(Q)\}_{Q \in \mathcal{M}_\sbf, \sbf \in \mathcal{F}}$. This concludes the proof of the lemma.
\end{proof}

\begin{proof}[Proof of Theorem \ref{main.bil.thrm}]
We will prove Theorem \ref{main.bil.thrm} with $\delta$ in (4) replaced by $2\delta$.
We begin, as in Theorem \ref{main}, by considering  $\epsilon \ll \delta$ and $K \gg \kappa$, and compared to  Theorem \ref{main}, a difference is that we now work with the stopping time regimes $\{\sbf^*\}_{\sbf^* \in \mathcal{F}^*}$ produced in Lemma \ref{BWGLcorona.lem}. Here we have no intention, or need, to carry out the analysis in Section \ref{packsttime}, and we can repeat the arguments verbatim up until that point. Since we already have proven the packing condition for the maximal cubes in Lemma \ref{BWGLcorona.lem}, all conclusions of  Theorem \ref{main.bil.thrm} hold with the potential exception of the estimate
\begin{equation}\label{bilgoal.eq}
\sup_{(Y,s)\in C_{\kappa \diam(Q)} \cap \Sigma_{\sbf^*}} d_p(Y,s, \Sigma) \le 2\delta\diam(Q), \quad \forall Q \in \sbf^*.
\end{equation}
Hence, to prove Theorem \ref{main.bil.thrm} we have to prove \eqref{bilgoal.eq}. We continue to use the notation $P_Q$ of Lemma \ref{BWGLcorona.lem}.
Fix $Q \in \sbf^*$ and $(Y,s) \in C_{\kappa \diam(Q)} \cap \Sigma_{\sbf^*}$.  Since $Q \in \sbf^*$ we have that $P_Q$ satisfies
\[\diam(Q)^{-1} \big[ \sup_{(Y,s)\in 8KQ} d_p(Y,s,P_Q) + \sup_{(X,t)\in C_{8K\diam(Q)} \cap P_Q } d_p(X,t,\Sigma)\big] < \epsilon,\]
and
\[\Ang(P_Q, P_{Q(\sbf)}) \le 4 \delta.\]
We assume, by translation and a change of coordinates, that $P_Q = \mathbb{R}^{n-1} \times \{0\} \times \re$.
Recalling that $\Sigma_{\sbf^*}$ is the graph of a Lip(1,1/2) function with Lip(1,1/2) constant $\lesssim \delta$, with respect to the coordinates induced by $P_{Q(\sbf^*)}$, it
follows, if we choose $\delta$ to be sufficiently small, that $\Sigma_{\sbf^*}$ is the graph of a Lip(1,1/2) function $\ti \psi: \rn \to \re$ with Lip(1,1/2) constant $\lesssim \delta$. In particular if $(x,0,t),(y,0,t) \in P_Q$ then
\begin{equation}\label{graphsmallpq.eq}
d_p((x,\ti\psi(x,t), t), (y,\ti\psi(y,s), t)) \le c(1 + \delta)d_p(x,t,y,s).
\end{equation}
Now let $(y,0,s)$ be the projection of $(Y,s)$ into $P_Q$. Then,  $(y,0,s) \in C_{(3/2)K\diam(Q)}(X_Q,t_Q) \cap P_Q$ as $b\beta_\infty(Q,K) < \epsilon$.
Again using that $b\beta_\infty(Q,K) < \epsilon$, we can conclude that there exists $(X,t) \in \Sigma \cap (7/8)\kappa Q \subseteq 2\kappa Q(\sbf^*)$ such that
\begin{equation}\label{Xtcloseys.eq}
d_p(X,t, (y,0,s)) < \epsilon \diam(Q),
\end{equation}
and
\begin{equation}\label{Xtprojcloseys.eq}
d_p(X,t, P_Q) = d_p((X,t), (x,0,t)) < \epsilon\diam(Q).
\end{equation}
The inequalities \eqref{Xtcloseys.eq} and \eqref{Xtprojcloseys.eq} imply
\[
d_p((x,0,t), (y,0,s)) \le 2\epsilon \diam(Q),
\]
and hence, by \eqref{graphsmallpq.eq},
\begin{equation}\label{graphsmallpq2.eq}
\begin{split}
d_p((x,\ti\psi(x,t), t), (Y,s)) & = d_p((x,\ti\psi(x,t), t), (y,\ti\psi(y,s), t))
\\& \le c(1 + \delta)2\epsilon \diam(Q) \le c \epsilon \diam(Q).
\end{split}
\end{equation}
By Lemma \ref{closegraph.lem}, which holds by construction, we have
\begin{equation}\label{ult}
d_p((X,t), (x,\ti\psi(x,t), t)) \le 2d_p(X,t, \Sigma_{\sbf^*}) \le c\kappa \epsilon \diam(Q),
\end{equation}
where we have used that $\ti\psi(x,t)$ is Lip(1,1/2) with small constant. Assuming that $\epsilon \ll \delta/\kappa$, the inequality in \eqref{ult} yields
\[d_p((X,t), (x,\ti\psi(x,t), t)) \le \delta \diam(Q).\]
Combining this inequality with \eqref{graphsmallpq2.eq}, we find that
\begin{align*}
d_p(X,t, Y,s) &\le d_p((X,t), (x,\ti\psi(x,t), t)) + d_p((x,\ti\psi(x,t), t), (Y,s))\\
& \le \delta + c\epsilon < 2\delta.
\end{align*}
This proves \eqref{bilgoal.eq} and  Theorem \ref{main.bil.thrm}.
\end{proof}

\appendix
\section{Two lemmas concerning approximating planes}\label{subsplanes}
In this appendix we prove two auxiliary and technical lemmas on approximating hyperplanes.
\begin{lemma}
\label{aux1} Let $\Sigma \subset \mathbb{R}^{n+1}$ be a parabolic Ahlfors-David regular set with dyadic cubes $\mathbb{D}(\Sigma)$. Then there exists a constant $A=A(n,M)\gg 1$ such that the following holds. Let $Q\in \mathbb D=\mathbb D(\Sigma)$. Suppose that $(Z_0,\tau_0)\in Q$. Then there exist points $(Z_i,\tau_i)\in Q$ for $i=1,...,n$, such that if we let,  for $j=1,...,n$, $L_{j-1}$ be the spatial $(j-1)$-dimensional plane which passes through $Z_0,Z_1,...,Z_{j-1}$, then
$$d_p(Z_j,\tau_j,L_{j-1}\times \mathbb R)\geq A^{-1}\diam(Q)$$ for $j=1,...,n$.
Moreover, the same result holds if $Q$ is replaced by $\lambda Q$ for some $\lambda \geq1$, with $A$ now depending on $\lambda$ as well.
\end{lemma}
\begin{proof} We only prove the result for $Q$ as the same proof works for any dilate $\lambda Q$ of $Q$, $\lambda\geq 1$.
 We will prove the lemma by induction on $j$ for $j=1,...,n$.  Let $L_0:=\{Z_0\}$. Then $L_0$ is a $0$-dimensional plane.
 Assume that
$d_p(Z,\tau,L_{0}\times \mathbb R)\leq A^{-1}\diam(Q)$ for all $(Z,\tau)\in Q$. Then $Q$ can be covered by roughly $A^2$ cubes of size $2A^{-1}\diam(Q)$. Hence
$$(\diam(Q))^{n+1}\lesssim\sigma(Q)\lesssim A^2(A^{-1}\diam(Q))^{n+1}$$
which is impossible if $A$ is large enough. Hence there exists a point $(Z_1,\tau_1)\in Q$ such that
$d_p(Z_1,\tau_1,L_{0}\times R)\geq A^{-1}\diam(Q)$ and the conclusion is true for $j=1$ and we let $L_1$ be the linear span of $Z_0$ and $Z_1$. Let  $1\leq j<n-1$ and suppose that we have
found  points $(Z_i,\tau_i)\in Q$ for $i=1,...,j-1$ as stated and that $L_{j-1}$ is the (spatial) $(j-1)$-dimensional plane spanned by $Z_0,Z_1,...,Z_{j-1}$. Assume that
$d_p(Z,\tau,L_{j-1}\times \mathbb R)\leq A^{-1}\diam(Q)$ for all $(Z,\tau)\in Q$. Then $Q$ can be covered by roughly $A^{j-1}A^2$ cubes of size $2A^{-1}\diam(Q)$ and, again using that $\Sigma$ has the parabolic ADR property, we see that
$$1\lesssim A^{j-1}A^2A^{-n-1}$$
which again is impossible, as $j<n$, if $A=A(n,M)\gg 1$ is large enough. Hence there exists $(Z_j,\tau_j)\in Q$ such that
$d_p(Z_j,\tau_j,L_{j-1}\times \mathbb R)\geq A^{-1}\diam(Q)$ and we let $L_j$ be the  (spatial) $j$-dimensional plane spanned by $Z_0,Z_1,...,Z_{j}$. The lemma now follows from the induction hypothesis.
\end{proof}

\begin{lemma} \label{aux2} Let $\Sigma \subset \mathbb{R}^{n+1}$ be a parabolic Ahlfors-David regular set with dyadic cubes $\mathbb{D}(\Sigma)$. Let $Q\in \mathbb D=\mathbb D(\Sigma)$ and assume that $P_1$ and $P_2$ are two time-independent hyperplanes such that $d_p(X,t,P_i)\leq\epsilon \diam(Q)$ for all
$(X,t)\in Q$ and for $i\in\{1,2\}$. Then
\begin{align}\label{a3}
(i)&\ d_p(Y,s,P_2)\lesssim \epsilon(\diam(Q)+d_p(Y,s,Q)),\ (Y,s)\in P_1,\notag\\
(ii)&\ d_p(Y,s,P_1)\lesssim \epsilon(\diam(Q)+d_p(Y,s,Q)),\ (Y,s)\in P_2.
\end{align}
In particular, the angle between $P_1$ and $P_2$ is bounded by $\lesssim \epsilon$.  

Moreover, \eqref{a3} continues to hold (with slightly different implicit constants) with $Q$ replaced by any 
``surface ball" $\Delta = \Delta_r(X,t) = C_r(X,t)\cap\Sigma$,  where $(X,t) \in \Sigma$, provided that
$d_p(X,t,P_i)\leq\epsilon \diam(\Delta)$ for all
$(X,t)\in 2\Delta$ and for each $i\in\{1,2\}$. 
%dilate $\lambda Q$ of $Q$, $\lambda \geq 1$.
\end{lemma}
\begin{proof}  It is enough to prove the result for a cube $Q$: indeed, given a surface ball $\Delta=\Delta_r$ for 
which $d_p(X,t,P_i)\leq\epsilon \diam(\Delta)$ for all
$(X,t)\in 2\Delta$ and for each $i\in\{1,2\}$, we may cover $\Delta$
by a bounded number of disjoint cubes $Q_j$, with $\diam(Q_j) \approx r$, such that
$\cup_j Q_j \subset 2\Delta$, and observe that \eqref{a3}
(with $\epsilon$ replaced by $C\epsilon$) may be applied to each $Q_j$.
 %as the proof for $\lambda Q$  is completely analogous.  
 
 Fix now a dyadic cube $Q\subset \Sigma$, satsifying the hypotheses of the lemma.
 Let $(Z_i,\tau_i)\in Q$ for $i=0,1,...,n$, be as in the statement of Lemma \ref{aux1} and let
$L_{Q}=L_{n-1}$  be the spatial $(n-1)$-dimensional plane which passes through $Z_0,Z_1,...,Z_{n-1}$. Let $P:=L_{Q}\times\mathbb R$.  Consider
$(Y,s)\in P_1$. Then
\begin{align}
d_p(Y,s,P_2)&\leq d_p(Y,s,P)+d_p(P,P_2)\notag\\
&\lesssim \max_{i\in \{0,...,n\}} d_p(Y,s,Z_i,\tau_i)+ \max_{i\in \{0,...,n\}} d_p(Z_i,\tau_i,P_2)\notag\\
&\lesssim \epsilon\diam(Q)+\max_{i\in \{0,...,n\}} d_p(Y,s,Z_i,\tau_i)\notag\\
&\lesssim \epsilon(\diam(Q)+d_p(Y,s,Q)).\end{align}
This proves the lemma.
\end{proof}

\section{Slice-wise vs. Parabolic Hausdorff Measure}\label{slicevspara.sect}

In this appendix we prove the (non-trivial) assertions made in Remark \ref{r-measures}. Recall $\mu$ is the (global) slice-wise measure defined as
\[\mu(E) = \int_{\mathbb{R}} \H^{n-1}(E_t) \, d\H^1(t),\]
where for $t \in \mathbb{R}$, $E_t := E \cap (\mathbb{R}^n \times \{t\})$.

The following proposition proves Remark \ref{r-measures}(vi).
\begin{proposition}\label{mulepara.prop}
There exists a constant $c(n)$ such that $\mu \le c(n)\H^{n+1}_p$.
\end{proposition}
\begin{proof}
In the proof, for any parabolic cube\footnote{Here we are taking an `honest' parabolic cube, different from the $C_r(X,t)$ defined above. This means $Q = \{(X,t): |X_j - X^*_j| < \ell/2, |t - t^*| < \ell^2/2\}$, where $(X^*,t^*)$ is the center of the cube and $\ell$ is the side length.} $Q$, we will write $Q= Q' \times I$, where $I$ is an interval, so that the side length of $Q'$, $\ell(Q')$, and the length of $I$, $\ell(I)$, satisfy $\ell(I)^2 = \ell(Q')$ and we let $\ell(Q) := \ell(Q')$. Suppose that $E \subset \mathbb{R}^{n+1}$, a Borel set, is such that $\H^{n+1}_p(E) < \infty$ (otherwise there is nothing to prove).
Recall that $H^{n+1}_{p, \delta}(E)$ is a decreasing function in $\delta$. Let $\delta > 0$ be arbitrary.  By definition of $\H^{n+1}_{p}(E)$ there exists a countable collection of cubes $Q_i$ such that
\[\sum_{i} \ell(Q_i)^{n+1} \le 2^{n+1} \H^{n+1}_{p}(E)  + \delta, \quad E \subseteq \cup_i Q_i \quad \text{ and } \quad  \diam(Q_i) \le \delta.\]
To see this, we use a covering $\{E_i\}$ which nearly minimizes the quantity in the definition of $H^{n+1}_{p,\delta'}(E)$, where $(2\sqrt{n} +  \sqrt{2})\delta'= \delta$, then for each $i$ we take $Q_i = C_{r_i}(X_i,t_i)$ with $\ell(Q_i)/2 = r_i = \diam(E_i)$ and $(X_i,t_i)$ an arbitrary point in $E_i$.

Let $\mathcal{I}(t) = \{i: t \in I_i\}$.  Then as $E \subset Q_i$ we have $E_t \subseteq \cup_{i \in \mathcal{I}(t)} Q'_i$ and $\diam(Q'_i) \le \delta$. Therefore
\[\H_\delta^{n-1}(E_t) \le c'(n) \sum_{i \in \mathcal{I}(t)} \ell(Q_i')^{n-1} = c'(n) \sum_{i \in \mathcal{I}(t)} \ell(Q_i)^{n-1},\]
where $c'(n) := (\sqrt{n})^{n-1}$.  Hence
\begin{align*}
\int_{\mathbb{R}} \H^{n-1}_\delta(E_t) \, d\H^1(t) &\le c'(n) \int_{\mathbb{R}}  \sum_{i \in \mathcal{I}(t)} \ell(Q_i)^{n-1} \, d\H^1(t) \\
& \le c'(n)\sum_i  \int_{I_i}  \ell(Q_i)^{n-1} \, d\H^1(t) = c'(n) \sum_i   \ell(Q_i)^{n+1} \\
& \le c'(n) 2^{n+1} \H^{n+1}_{p}(E)  + c'(n) \delta.
\end{align*}
The result now follows from the monotone convergence theorem as $\H^{n-1}_\delta(E_t)$ increases as $\delta$ decreases.
\end{proof}

The following proposition proves Remark \ref{r-measures}(i) for the measure $\sigma^{\bf s}$.
\begin{proposition}\label{muhpsim.prop}
Let $E \subset \mathbb{R}^{n+1}$ be closed. There exists $c(n) > 0$ such that the following holds. If there is a constant $c$ such that
\[c^{-1}r^{n+1} \le  \mu(C_r(X,t) \cap E) \le cr^{n+1}, \quad \forall (X,t) \in E, r > 0,\]
then
\[c(n)^{-1} c^{-1}r^{n+1} \le  \H^{n+1}_p(C_r(X,t) \cap E) \le c(n)c r^{n+1}, \quad \forall (X,t) \in E, r > 0.\]
\end{proposition}
\begin{proof}
The lower bound follows directly from the previous proposition. To prove the other inequality fix $(X,t) \in E, r > 0$, let $\delta \in (0,r)$ be arbitrary and let $r_\delta$ be such that $\diam(C_{5r_\delta}(0)) = \delta$. Clearly,
\[C_r(X,t) \cap E \subseteq \bigcup_{(Y,s) \in C_r(X,t) \cap E}C_{r_\delta}(Y,s) \cap E.\]
By the 5r-covering lemma (see e.g. \cite[Theorem 2.1]{Mat}) there exists a countable collection of $(Y_i,s_i)\in E \cap C_r(X,t)$  such that the cubes $C_{r_\delta}(Y_i,s_i)$ are disjoint and
\[C_r(X,t) \cap E \subset \cup_i C_{5r_\delta}(Y_i,s_i) \cap E.\]
It follows that
\begin{align*}
 \H^{n+1}_{p,\delta}(C_r(X,t) \cap E)&\lesssim \sum_i \diam(C_{5r_\delta}(Y_i,s_i))^{n+1} \lesssim c\sum_i \mu(C_{r_\delta}(Y_i,s_i) \cap E)^{n+1} \\
&\lesssim \mu(C_{2r}(X,t) \cap E) \lesssim c r^{n+1},
\end{align*}
where now $\lesssim$ means that constants only depend on $n$, and where we have used that $\diam(C_{r_\delta}(Y_i,s_i)) < \delta < r$ in the second-to-last inequality.
Letting $\delta \to 0^+$ yields the upper bound.
\end{proof}

On the other hand, $\H^{n+1}_p \not \ll \mu$, even when the set is p-ADR with respect to $\H^{n+1}_p$, as the following example shows.
\begin{example}
Let $E= [C_{1-1/2n}]^n \times C_{3/4}$, where $C_s$ denotes a (generalized) $s$-dimensional Cantor-type set in $\mathbb R$. Then $E$ is p-ADR with respect to $\H^{n+1}_p$, but $\mu(E) = 0$. The fact that $\mu(E) = 0$ follows from the fact that $\H^1(C_{3/4}) = 0$. Moreover, one can verify that $E$ is p-ADR with respect to the measure $\H^{n-1/2} \times \H^{3/4}$ and hence p-ADR with respect to $\H^{n+1}_p$.
\end{example}

We will require the following definition.

\begin{definition}
Let $E\subset \mathbb{R}^{n+1}$ be a p-ADR set (with respect to $\H^{n+1}_p$) and let $\mathcal{S}$ be a collection of p-ADR sets, with uniform control on the p-ADR constant. We say $E$ is big pieces of $\mathcal{S}$, written $E$ is $BP(\mathcal{S})$, if there exists $\theta > 0$ such that for every $(X,t) \in E$ and $r \in (0,\diam(E))$ there exists $\Gamma \in \mathcal{S}$ with
\[\H^{n+1}_p(E \cap C_r(X,t)  \cap \Gamma) \ge \theta \H^{n+1}_p(E \cap C_r(X,t)).\]
We say $E$ is big pieces squared of $\mathcal{S}$, written $E$ is $BP^2(\mathcal{S})$ if there exists a constant $M$ such that $E$ is $BP(BP(\mathcal{S},M))$, where $BP(\mathcal{S},M)$ is the collection of all p-ADR sets with p-ADR constant less than $M$ which are $BP(\mathcal{S})$.
\end{definition}

\begin{proposition}\label{bpmuadr.prop}
Let $\mathcal{S}$ be a collection of closed subsets of $\mathbb{R}^{n+1}$, which are uniformly p-ADR with respect to $\mu$, that is, there exists $M > 1$ such that for all $\Gamma \in \mathcal{S}$ it holds
\[M^{-1} r^{n+1} \le \mu(\Gamma \cap C_r(X,t)) \le M r^{n+1}, \quad \forall (X,t) \in \Gamma, r > 0.\]
If $E$ is p-ADR (with respect to $\H^{n+1}_p$) and $E$ is $BP(\mathcal{S})$, then $E$ is p-ADR with respect to $\mu$ with constant depending only on $M$, $n$, the ADR constant of $E$, and the constant $\theta$ in the definition of $BP(\mathcal{S})$. In particular, $\mu|_E \approx \H^{n+1}_p|_E$, with implicit constant depending only on $M$, $n$, the p-ADR constant of $E$, and the constant $\theta$ in the definition of $BP(\mathcal{S})$.
\end{proposition}
\begin{proof}
Let $(X,t) \in E$ and $r \in (0,\diam(E))$. By Proposition \ref{mulepara.prop}
\[\mu(C_r(X,t) \cap E) \le c(n) H_p^{n+1}(E \cap C_r(X,t) )\lesssim r^{n+1},\]
where the implicit constant depends on the p-ADR constant of $E$ and $n$. To prove the lower bound, we note that for $\Gamma \in \mathcal{S}$, $\H^{n+1}_p|_\Gamma \approx \mu|_\Gamma$, with implicit constant depending on $M$ since any set which is p-ADR with respect to $\mu$ is p-ADR with respect to $\H^{n+1}_p$ (by Proposition \ref{muhpsim.prop}). Then using that $E$ is p-ADR and $E$ is $BP(\mathcal{S})$ there exists $\Gamma \in \mathcal{S}$ such that
\begin{align*}
r^{n+1} \approx \H^{n+1}(E \cap C_r(X,t)) &\le \theta^{-1} \H^{n+1}(E \cap C_r(X,t) \cap \Gamma)
\\ & \approx_M \theta^{-1} \mu(E \cap C_r(X,t) \cap \Gamma)
\\ &\lesssim \theta^{-1} \mu(E \cap C_r(X,t)) .
\end{align*}
This proves the proposition. \end{proof}

The following corollary finishes the proof of the assertion in Remark \ref{r-measures}(v).
\begin{corollary}
If $E$ is parabolic UR then $\mu|_E \approx \H^{n+1}_p|_E$, that is, $\sigma^s \approx \sigma$. Here the implicit constants depend on dimension and the parabolic UR constants for $E$.\end{corollary}
\begin{proof}
We have shown in Theorem \ref{main} that if $E$ is parabolic UR, then $E$ admits a corona decomposition with respect to regular Lip(1,1/2) graphs, with (uniform) control on the Lip(1,1/2) constant in terms of the parabolic UR constants for $E$. Then it follows from \cite[Theorem 1.1]{BHHLN1} that $E$ is $BP^2(\mathcal{S})$, where $\mathcal{S}$ is a collection of Lip(1,1/2) graphs with uniform control on the Lip(1,1/2) constant. In particular, $\mathcal{S}$ is a collection of sets which are uniformly p-ADR with respect to $\mu$. Applying Proposition \ref{bpmuadr.prop} twice yields the corollary.
\end{proof}

\end{document}